\documentclass[pdflatex,sn-mathphys-num]{sn-jnl}

\usepackage{geometry}
\usepackage{graphicx}%
\usepackage{threeparttable}
\usepackage{multirow}%
\usepackage{amssymb,amsmath,amsfonts,amsthm,mathrsfs,mathtools}
\usepackage[title]{appendix}%
\usepackage{xcolor}
\usepackage{textcomp}
\usepackage{manyfoot}
\usepackage{booktabs}%
\usepackage{algorithm}%
\usepackage{algorithmicx}%
\usepackage{algpseudocode}%
\usepackage{listings}%
\usepackage{enumitem}
\numberwithin{equation}{section}

\usepackage{tabularx}
\usepackage{microtype}
\usepackage{ragged2e}

\definecolor{journalblue}{RGB}{22,58,110}
\definecolor{accentteal}{RGB}{0,112,118}
\definecolor{pitchgray}{RGB}{78,84,92}
\definecolor{rulegray}{RGB}{190,197,205}
\definecolor{panelblue}{RGB}{239,244,250}

\newcommand{\pitchsection}[1]{%
  \vspace{0.65em}%
  {\large\bfseries\color{journalblue}#1}\par
  \vspace{0.18em}%
  {\color{rulegray}\hrule height 0.6pt}%
  \vspace{0.38em}%
}

\newcommand{\advanceitem}[2]{%
  \item \textbf{\color{journalblue}#1}\;#2%
}
\theoremstyle{thmstyleone}%
\newtheorem{theorem}{Theorem}[section]
\newtheorem{proposition}[theorem]{Proposition}
\newtheorem{lemma}[theorem]{Lemma}
\newtheorem{corollary}[theorem]{Corollary}
\theoremstyle{definition}
\newtheorem{definition}[theorem]{Definition}

\newtheorem{example}[theorem]{Example}
\theoremstyle{remark}
\newtheorem{remark}[theorem]{Remark}

\begin{document}
\small
\title{From Flows to Maps: Sampling Laws for Attractor Intensity and Bounded-Noise Escape}

\author[1]{\fnm{Jiguang} \sur{Yu}}
\email{jyu678@bu.edu}
\equalcont{These authors contributed equally to this work.}

\author*[2]{\fnm{Louis Shuo} \sur{Wang}}
\email{wang.s41@northeastern.edu}
\equalcont{These authors contributed equally to this work.}

\affil[1]{
  \orgdiv{College of Engineering}, 
  \orgname{Boston University}, 
  \orgaddress{
    \city{Boston}, 
    \state{MA}, 
    \postcode{02215}, 
    \country{United States}
  }
}

\affil[2]{
  \orgdiv{Department of Mathematics}, 
  \orgname{Northeastern University}, 
  \orgaddress{
    \city{Boston}, 
    \state{MA}, 
    \postcode{02115}, 
    \country{United States}
  }
}

\abstract{
Intensity of attraction quantifies the largest amplitude of a persistent bounded disturbance that an attractor can withstand without loss of controlled confinement in its basin. Although intensity has been formulated separately for flows and maps, its behavior under temporal sampling has remained unresolved. We establish an explicit correspondence between the intensity $\mu(A)$ of a continuous-time attractor and the intensity $\mu_h(A)$ of its exact time-$h$ map. For an $L$-Lipschitz vector field,
\[
\frac{\mu(A)}{1+Lh}
\leq
\frac{\mu_h(A)}{h}
\leq
\mu(A)\frac{e^{Lh}-1}{Lh},
\]
and hence $\mu_h(A)/h\to\mu(A)$. The resulting first-order rate is sharp in general, while smooth scalar escape geometries can exhibit second-order convergence. We extend the framework to one-step numerical methods through a stability theory for block intensity and to attracting invariant graphs over compact invertible nonautonomous bases, obtaining uniform sampling convergence over the forcing phase. For bounded-support random perturbations, normalized discrete intensity is identified with the pathwise safety threshold; above it, finite escape follows under an explicit finite-exit condition, while escape probabilities require additional assumptions on the noise law. We also show that the discrete state--normal boundary map converges to the normalized Pontryagin boundary system governing extremal reachable-set boundaries. Exact scalar benchmarks, a grazing resilience model, planar Duffing escape, anisotropic disturbances, periodic and quasiperiodic forcing, and transfer-operator computations illustrate the theory. These results give intensity estimated from discrete observations or simulations a sampling-independent continuous-time meaning.}

\keywords{
intensity of attraction;
attractor robustness;
resilience;
controlled reachability;
reachable-set boundary dynamics}

\pacs[MSC Classification]{37C70, 93B03, 37B55, 37M15, 37H30}

\maketitle

% ==============================================================================
% RESEARCH HIGHLIGHTS
% ==============================================================================
\begingroup
\setlength{\parindent}{0pt}
\setlength{\parskip}{0pt}
\setlength{\fboxsep}{8pt}
\renewcommand{\arraystretch}{1.08}
\vspace{2em}

\newpage

{\small\bfseries\color{accentteal}\MakeUppercase{Research highlights}}
\hfill
\vspace{0.35em}

{\fontsize{21}{23}\selectfont\bfseries\color{journalblue}
From Flows to Maps:}\par
\vspace{0.18em}
{\large\itshape Sampling Laws for Attractor Intensity and Bounded-Noise Escape}\par
\vspace{0.55em}

% --- Editorial hook ----------------------------------------------------------
\colorbox{panelblue}{%
  \parbox{\dimexpr\textwidth-2\fboxsep\relax}{%
    \textbf{The gap.}
    Intensity quantifies the smallest amplitude of bounded disturbance capable
    of driving trajectories from an attractor beyond its basin.  The existing theories for flows
    and maps, however, use different physical units, while observations,
    numerical integrators, and bounded-noise models act at discrete sampling
    times.
    \textbf{Our resolution.}
    We identify the exact normalization and prove a quantitative bridge from
    continuous controlled reachability to discrete pseudotrajectories.  The
    resulting framework extends uniformly to nonautonomous invariant graphs,
    numerical one-step maps, convex disturbance geometries, and bounded-support
    random perturbations.
  }%
}

% --- Central result ----------------------------------------------------------
\pitchsection{Central theorem}

For the time-$h$ map $F_h=\phi^h$ of a Lipschitz flow, under the stated compact
trapping and controlled-reachability hypotheses,
\[
\quad
  \frac{\mu(A)}{1+Lh}
  \;\leq\;
  \frac{\mu_h(A)}{h}
  \;\leq\;
  \mu(A)\frac{e^{Lh}-1}{Lh}
  \quad
  \qquad\Longrightarrow\qquad
  \frac{\mu_h(A)}{h}\longrightarrow\mu(A).
\]
Thus a displacement per sampling step becomes a disturbance rate only after
division by $h$.  The proof is constructive in both directions: discrete kicks
are interpolated by bounded controls, and controlled trajectories are converted
into discrete pseudotrajectories using Gr\"onwall estimates.

% --- Innovations -------------------------------------------------------------
\pitchsection{Five mathematical advances}

\begin{enumerate}[
  label=\textbf{\color{accentteal}\arabic*.},
  leftmargin=1.75em,
  itemsep=0.28em,
  topsep=0pt,
  parsep=0pt
]
  \advanceitem{Exact sampling law.}{Two-sided bounds give a dimensionally
  consistent and sampling-independent interpretation of intensity, with the
  general estimate $\mu_h/h=\mu+O(h)$.}

  \advanceitem{Stable numerical approximation.}{On a common trapping block,
  block intensity is Lipschitz with respect to the one-step map.  Consequently,
  an order-$q$ scheme, together with the required attractor approximation,
  yields $\mu_{B,h}^{\Psi}/h=\mu_B+O(h)+O(h^q)$.}

  \advanceitem{Uniform nonautonomous theory.}{For an attracting invariant graph
  over an invertible compact base flow, fiberwise intensity satisfies the same
  sampling bounds uniformly over all forcing phases.}

  \advanceitem{Sharp bounded-support threshold.}{For
  $X_{k+1}=\phi^h(X_k)+\xi_k$ with
  $\operatorname{supp}\xi_k\subseteq hrC$, the worst-case threshold is
  $r_{\mathrm{esc},C}(h)=\mu_{h,C}(A)/h$.  Below it every admissible realization
  is safe; above it an admissible finite escape sequence exists.}

  \advanceitem{Boundary dynamics and optimal control.}{At smooth exposed
  boundary points, the discrete inverse-transpose normal map converges to the
  Pontryagin boundary system.  Convex disturbance bodies enter naturally
  through their support functions.}
\end{enumerate}

% --- Evidence and significance ----------------------------------------------
\pitchsection{Evidence, scope, and significance}

\begin{tabularx}{\textwidth}{@{}>{\RaggedRight\arraybackslash}X@{\hspace{1.35em}}>{\RaggedRight\arraybackslash}X@{}}
  \textbf{Computational evidence.}
  An exactly solvable scalar benchmark verifies the bounds and their sharpness;
  independent boundary-system, sampled-map, grid, replay, and optimization
  calculations recover the same planar Duffing threshold; periodic and
  quasiperiodic graph experiments confirm phase-uniform convergence; and an
  absorbing transfer operator, validated by Monte Carlo, separates pathwise
  safety from observable escape probability.
  &
  \textbf{Conceptual impact.}
  Intensity measures global basin escape and therefore contains information
  absent from local eigenvalues, recovery rates, and hyperbolic continuation.
  The results also show why long-time trajectory agreement with an averaged
  model does not, by itself, preserve basin geometry or intensity.  Unbounded
  noise is deliberately excluded from the pathwise threshold theory and instead
  belongs to exit-time, quasipotential, and large-deviation analysis.
\end{tabularx}

\vspace{0.75em}
\colorbox{panelblue}{%
  \parbox{\dimexpr\textwidth-2\fboxsep\relax}{%
    {\small\color{pitchgray}
    \textbf{Editorial relevance.}
    The paper connects rigorous nonlinear dynamics, control-theoretic
    reachability, numerical analysis, nonautonomous systems, and noise-induced
    transitions through one computable global resilience quantity.}%
  }%
}
\vspace{2em}
\endgroup
% ==============================================================================
\newpage

\section{Introduction and main results}
\label{sec:introduction}

\subsection{Motivation}
\label{subsec:intro-motivation}

Resilience describes the ability of a dynamical system to retain its qualitative behavior under perturbations. While classical theory often relies on structural stability to guarantee persistence under sufficiently small errors \citep{conley1971isolated}, capturing true robustness requires determining the exact threshold of structural alteration a system can endure before an original attractor is destroyed. In systems with several competing regimes, a basic manifestation is the ability of an attractor to withstand disturbances without being driven across the boundary of its basin. This arises in ecology, climate dynamics, epidemiology, engineering, and control \citep{holling1973resilience,scheffer2001catastrophic,wang2025analysis,lenton2008tipping}. A loss of resilience may precede a classical bifurcation \citep{scheffer2009early,kuehn2011mathematical}, but it may also occur without any local change in the attractor itself: the attractor can remain hyperbolic while transient dynamics or basin geometry makes escape increasingly easy, and transitions may be induced by noise or by the rate of parameter drift rather than by a static bifurcation \citep{ashwin2012tipping,wieczorek2011excitability,wang2025analysis1,ashwin2017parameter}. From a modeling perspective, physically relevant noise is often intrinsically bounded \citep{d2013bounded}. Under such bounded disturbances, trajectories may be confined to forward-invariant structures known as minimal attractors, making the analysis of their boundaries crucial for understanding system robustness.

Many resilience indicators are local---eigenvalues, recovery rates, or the response to small displacements \citep{krakovska2024resilience}---and do not determine the amplitude of a persistent disturbance needed to force the system out of its basin; global probabilistic or geometric measures such as basin stability address part of this gap \citep{menck2013basin,menck2014dead,liang2025global,mitra2015integrative,klinshov2016stability,wang2025multi}. The \emph{intensity of attraction} measures this global robustness deterministically. McGehee gave a metric notion for continuous maps \citep{mcgehee1988some}; Meyer and McGehee formulated it for ordinary differential equations through controlled reachability \citep{meyer2022intensity}: intensity is the largest uniform amplitude of an additive control that can be applied indefinitely without allowing trajectories to escape. Evaluating resilience through these bounded, exogenous continuous-time disturbances captures the dynamic strength across transient regions, effectively reframing worst-case environmental forcing as invariant reachable sets in differential inclusions \citep{li2007morse,meyer2016mathematical}. It is therefore complementary to hyperbolicity and to variance- or autocorrelation-based early-warning signals: it reports the finite disturbance amplitude a state can still absorb, and (as we show for a bistable grazing model) vanishes at a tipping fold while remaining informative away from it.

The continuous-time theory does not resolve an issue unavoidable in applications: models are formulated as differential equations, but observations, simulations, control, and disturbances act at discrete times.
If $\phi^t$ is the flow of an ODE and $A$ an attractor, a simulation with step $h>0$ studies the time-$h$ map $F_h=\phi^h$. Flow and sampled map share the unperturbed invariant set, but their perturbations differ in meaning: a continuous disturbance is a rate, whereas a perturbation of $F_h$ is a displacement per step. Hence the discrete intensity $\mu_h(A)$ is not directly comparable with the continuous intensity $\mu(A)$; the dimensionally consistent quantity is $\mu_h(A)/h$. This paper answers: \emph{does the normalized intensity of the sampled dynamics converge to the intensity of the flow as $h\downarrow0$?} We answer affirmatively, with explicit two-sided bounds, via quantitative inclusions between continuous controlled reachable sets and discrete perturbed reachable sets.
This connection tightly links the infinite-dimensional problem of attractor persistence to reachable set computation in modern control theory \citep{de2011level,wang2026algebraic,rieger2016euler}.
Furthermore, while observing the boundaries of minimal attractors under discrete bounded perturbations via traditional set-oriented numerical methods is highly computationally expensive \citep{dellnitz2001algorithms}, our analytical inclusions provide a rigorous and more direct mathematical bridge. The same inclusions fix the scaling under which bounded kicks, numerical one-step errors, and bounded-support noise approximate a continuous control, and they extend to numerical methods, to attracting invariant graphs over a nonautonomous base, and to the reachable-set boundary dynamics.

\subsection{Relation to existing theories}
\label{subsec:intro-literature}

Our attractor and basin notions follow the standard theory of attractors and isolated invariant sets \citep{conley1978isolated,hurley1982attractors,wang2026damage,milnor1985concept,norton1995fundamental}. Intensity in the map setting is McGehee's \citep{mcgehee1988some}; in the flow setting it is Meyer and McGehee's controlled-reachability quantity \citep{meyer2022intensity}. Our results differ from structural stability and hyperbolic continuation \citep{hirsch1970invariant,robinson1998dynamical,wang2026breakdown}, which give persistence of a nearby invariant object but not the largest persistent-disturbance amplitude preventing basin escape. Instead of relying on the qualitative persistence guaranteed by sufficiently small errors, our approach asks in precise metric terms how different a vector field can become while still retaining an original attractor. This explicitly reduces the infinite-dimensional problem of attractor persistence to the well-studied computation of reachable sets in a control framework \citep{baier2013approximation,yu2026pattern,de2011level}.

The bounded-noise results are distinct from large-deviation and quasipotential theory. Freidlin--Wentzell theory quantifies rare transitions under small \emph{unbounded} noise through action minimization and quasipotentials
\citep{freidlinrandom,kifer1988random,dembo2009large,berglund2006noise}, with a numerical literature on minimum-action paths \citep{heymann2008geometric,cameron2012finding,grafke2019numerical,liu2025bidirectional}. Intensity is instead an amplitude-constrained worst-case quantity; for bounded support it decides whether escape is impossible for every admissible realization or possible for some admissible finite sequence. By reframing a dynamical system subjected to bounded, time-varying disturbances as a differential inclusion, the set reachable from a deterministic attractor remains robustly invariant under the worst-case constraints \citep{thieme2022conley}. Consequently, bounded-noise dynamics are naturally set-valued \citep{zmarrou2007bifurcations,botts2012hopf,callaway2017dichotomy}. 

The nonautonomous theory uses cocycles and skew-products, consistent with pullback, uniform, and random attractors \citep{arnold2006random,cai2026optimal,kloeden2011nonautonomous,carvalho2012attractors,yu2026rigorous,crauel1994attractors,crauel1997random,rasmussen2007attractivity}. The boundary equations relate to the Pontryagin maximum principle
\citep{pontryagin2018mathematical,lee1967foundations,clarke1998nonsmooth,yu2026beyond} and to boundary descriptions of differential inclusions and reachable sets \citep{cardaliaguet1999set,mitchell2005time,yu2026microscopic,baier2007stability,gao2022rolling}; a general nonautonomous boundary-system formalism is developed in \citep{kourliouros2026invariant,wang2026elliptic}, and bounded-noise boundary maps in \citep{lamb2026bifurcations}. We do not claim priority for that formalism: our contribution is the quantitative bridge between exact sampling, intensity, bounded-noise escape, and boundary propagation.

\subsection{Main results}
\label{subsec:intro-contributions}

Let $f:\mathbb R^d\to\mathbb R^d$ generate a forward-complete flow $\phi^t$; let $A$ be a compact attractor with basin $\mathcal D(A)$, continuous intensity $\mu(A)$, and $F_h=\phi^h$ with discrete intensity $\mu_h(A)$ (Section~\ref{sec:intensity}).

\emph{(1) Exact sampling.} If $f$ is globally $L$-Lipschitz, then Theorem~\ref{thm:sampling} gives
\begin{equation}
\label{eq:intro-sampling-bound}
    \frac{\mu(A)}{1+Lh}
    \le
    \frac{\mu_h(A)}{h}
    \le
    \mu(A)\,\frac{e^{Lh}-1}{Lh},
    \qquad\text{so}\qquad
    \frac{\mu_h(A)}{h}=\mu(A)+O(h),
\end{equation}
and the first-order rate is sharp (Example~\ref{ex:sharp-scalar}). The two bounds come from complementary reachability inclusions: a discrete $\varepsilon$-kick is interpolated by a control of amplitude $(1/h+L)\varepsilon$, and a control of amplitude $r$ produces one-step error at most $r(e^{Lh}-1)/L$ by Gr\"onwall.

\emph{(2) Numerical discretization.} For a one-step method $\Psi_h$ of order $q$ with defect $O(h^{q+1})$ on a common trapping block $B$, a block-intensity stability theorem (Theorem~\ref{thm:block-intensity-map-stability}) yields, under the continuation hypotheses of Theorem~\ref{thm:numerical-intensity}, $\mu^{\Psi}_{h,B}(A_h)/h=\mu_B(A)+O(h)+O(h^{q})$, equal to $\mu(A)+O(h)+O(h^q)$ when $B$ realizes the full intensity.

\emph{(3) Bounded-noise escape.} For $X_{k+1}=\phi^h(X_k)+\xi_k$ with $\operatorname{supp}\xi_k\subseteq \overline B_{hr}(0)$, the worst-case escape threshold satisfies $r_{\mathrm{esc}}(h)=\mu_h(A)/h\to\mu(A)$
(Corollary~\ref{cor:bounded-noise}): if $hr<\mu_h(A)$ escape is impossible for every admissible realization, while if $hr>\mu_h(A)$ an admissible finite escape sequence exists. Turning existence into positive-probability or almost-sure escape needs further probabilistic hypotheses (Proposition~\ref{prop:probabilistic-escape}).

\emph{(4) Nonautonomous invariant graphs.} Over a compact metric flow $(\Omega,\theta)$ with $\dot x=f(\theta_t\omega,x)$ possessing a continuous attracting invariant graph $\mathcal A$, we define a uniform fiberwise intensity $\mu_{\mathrm{na}}(\mathcal A)$. If $f(\omega,\cdot)$ is uniformly $L$-Lipschitz and the base flow is invertible, Theorem~\ref{thm:nonautonomous-sampling} gives the bounds \eqref{eq:intro-sampling-bound} with $\mu,\mu_h$ replaced by
$\mu_{\mathrm{na}},\mu_{\mathrm{na},h}$; invertibility aligns arbitrary continuous endpoints with the sampling grid.

\emph{(5) Boundary dynamics.} For Euclidean kicks of radius $hr$, the discrete boundary map on the state--normal bundle,
\[
    \beta_{h,r}(x,n)
    =\bigl(\phi^h(x)+hr\,\eta_h,\ \eta_h\bigr),
    \qquad
    \eta_h=\frac{D\phi^h(x)^{-T}n}{\|D\phi^h(x)^{-T}n\|},
\]
is a first-order-consistent discretization (Theorem~\ref{thm:boundary-limit}) of the normalized Pontryagin system $\dot x=f(x)+rn$, $\dot n=-(I-nn^\top)Df(x)^\top n$, with Hamiltonian $H_r(x,p)=p^\top f(x)+r\|p\|$; the anisotropic case replaces $rn$ by $r\nabla h_C(n)$.

The five results and their locations across the paper are as follows. Section~\ref{sec:intensity} sets up continuous and discrete intensity; Section~\ref{sec:sampling} proves the sampling and numerical theorems; Section~\ref{sec:nonautonomous} treats invariant graphs and averaging; Section~\ref{sec:noise-boundary} gives the bounded-noise threshold and boundary systems; Section~\ref{sec:examples} reports the numerical experiments. 

\section{Intensity and controlled reachability}
\label{sec:intensity}

This section establishes a common reachability framework for continuous flows, discrete maps, bounded disturbances, and numerical one-step methods. Our definitions are based on the metric formulation for maps in \citep{mcgehee1988some} and the control-theoretic formulation for ordinary differential equations in \citep{meyer2022intensity}. We make the compact-containment requirement explicit throughout. This requirement is stronger than merely asking that every reachable point lie in the basin, and it is essential for obtaining estimates that are stable under sampling and numerical perturbation.

Throughout the paper, $\|\cdot\|$ denotes the Euclidean norm on $\mathbb R^d$. For a set $S\subset\mathbb R^d$ and a point $x\in\mathbb R^d$, we write $\operatorname{dist}(x,S):=\inf_{y\in S}\|x-y\|$. If $K$ is a compact subset of an open set $U$, we write $K\Subset U$. Equivalently, $\operatorname{dist}\bigl(K,\mathbb R^d\setminus U\bigr)>0$. Open and closed Euclidean balls are denoted by $B_r(x)$ and $\overline B_r(x)$, respectively.

\subsection{Continuous-time and discrete-time intensity}
\label{subsec:flow-intensity}

Consider the autonomous differential equation
\begin{equation}
\label{eq:autonomous-system}
    \dot x=f(x),\qquad x\in\mathbb R^d,
\end{equation}
where $f:\mathbb R^d\to\mathbb R^d$ is locally Lipschitz and \eqref{eq:autonomous-system} is forward complete. We denote its flow by $\phi^t(x)=\phi(t,x)$, $t\geq0$.

\begin{definition}[Attractor and basin]
\label{def:attractor-basin}
A nonempty compact set $A\subset\mathbb R^d$ is an attractor for \eqref{eq:autonomous-system} if it is invariant, $\phi^t(A)=A$, $t\geq0$, and there exists an open neighborhood $U$ of $A$ such that $\lim_{t\to\infty}\operatorname{dist}\bigl(\phi^t(x),A\bigr)=0$ for every $x\in U$. The basin of attraction of $A$ is \[
    \mathcal D(A)
    :=
    \left\{
    x\in\mathbb R^d:
    \lim_{t\to\infty}
    \operatorname{dist}\bigl(\phi^t(x),A\bigr)=0
    \right\}.
\]
\end{definition}

Under Definition~\ref{def:attractor-basin}, $\mathcal D(A)$ is open and positively invariant. More restrictive definitions of an attractor, for example uniform attraction of compact subsets of a neighborhood, may be used without affecting the results below.

We perturb \eqref{eq:autonomous-system} by an additive measurable control:
\begin{equation}
\label{eq:controlled-system}
    \dot x=f(x)+u(t).
\end{equation}
For an interval $I\subset[0,\infty)$ and $r>0$, define
\begin{equation}
\label{eq:control-class}
    \mathcal U_r(I)
    :=
    \left\{
    u\in L^\infty(I;\mathbb R^d):
    \|u\|_{\infty,I}<r
    \right\},
\end{equation}
where
    $\|u\|_{\infty,I}
    :=
\operatorname*{ess\,sup}_{t\in I}\|u(t)\|$. We set $\mathcal U_0(I):=\{0\}$. The solution of \eqref{eq:controlled-system} with $x(0)=x_0$ is denoted by $x_u(t;x_0)$. Whenever necessary, forward completeness of the controlled system is understood for the control amplitudes under consideration. It follows automatically, for example, if $f$ is globally Lipschitz.

For a set $S\subset\mathbb R^d$, the exact-time reachable set is
\begin{equation}
\label{eq:exact-reachable-flow}
    \mathcal R_r(S,t)
    :=
    \left\{
    x_u(t;x_0):
    x_0\in S,\;
    u\in\mathcal U_r([0,t])
    \right\},
    \qquad t\geq0.
\end{equation}
The finite-time reachable tube and the unrestricted reachable set are defined by
\begin{align}
    \mathcal P_r(S;T)
    &:=
    \bigcup_{0\leq t\leq T}\mathcal R_r(S,t),
    \qquad T\geq0, \nonumber\\
\label{eq:infinite-reachable-flow}
    \mathcal P_r(S)
    &:=
    \bigcup_{T>0}\mathcal P_r(S;T)
     =
    \bigcup_{t\geq0}
    \bigcup_{\substack{x_0\in S\\
                      u\in\mathcal U_r([0,t])}}
    \{x_u(t;x_0)\}.
\end{align}
Thus $\mathcal P_r(A)$ contains every point that can be reached from the attractor by a control of uniform amplitude strictly smaller than $r$, over an arbitrary finite time interval. Figure~\ref{fig:basin} demonstrates continuous intensity, the compact safety margin inside the basin, and the interpolation of a discrete perturbed orbit.

\begin{figure}[htbp]
    \centering
    \includegraphics[width=0.45\linewidth]{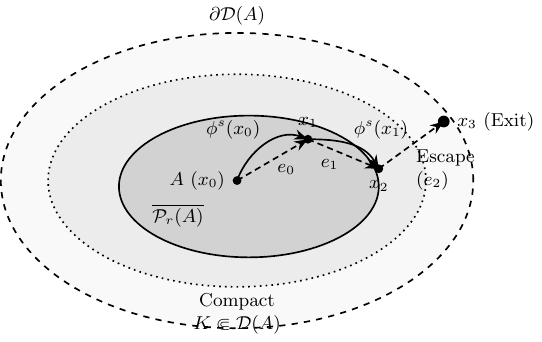}
    \caption{The reachability and exact sampling schematic. It demonstrates continuous intensity, the compact safety margin inside the basin, and the interpolation of a discrete perturbed orbit.}
    \label{fig:basin}
\end{figure}

\begin{definition}[Continuous-time intensity]
\label{def:continuous-intensity}
Let $A$ be an attractor of \eqref{eq:autonomous-system}. Its continuous-time intensity is
\begin{equation}
\label{eq:continuous-intensity}
    \mu(A)
    :=
    \sup\left\{
    r\geq0:
    \overline{\mathcal P_r(A)}
    \subset K\Subset\mathcal D(A)
    \text{ for some compact }K
    \right\}.
\end{equation}
\end{definition}

Equivalently, $\overline{\mathcal P_r(A)}$ is compact and contained in $\mathcal D(A)$; the intermediate $K$ makes the uniform safety margin $\operatorname{dist}(\overline{\mathcal P_r(A)},\mathbb R^d\setminus\mathcal D(A))>0$ explicit.

\begin{remark}[Why compact containment is necessary]
\label{rem:compact-containment}
The weaker requirement $\mathcal P_r(A)\subset\mathcal D(A)$ gives no uniform distance to $\partial\mathcal D(A)$ and permits escape to infinity within an unbounded basin; neither is stable under perturbation of $F_h$. Compact containment rules out both.
\end{remark}

The strict inequality in \eqref{eq:control-class} is useful when comparing reachable sets at different amplitudes. Replacing it by $\|u\|_{\infty,I}\leq r$ and taking closures leads to the same critical value under the standard continuous-dependence assumptions. We retain the strict convention because it agrees naturally with open metric perturbations of a map.

Having established the intensity for continuous-time flows, we now develop a parallel framework for discrete-time systems. This is an essential step toward analyzing sampled-data systems and numerical integration schemes, where the continuous flow is replaced by a time-$h$ map $F_h \approx \phi^h$.

Let $F:\mathbb R^d\to\mathbb R^d$ be continuous and let $A$ be a compact attractor for the discrete dynamical system
\begin{equation}
\label{eq:discrete-system}
    x_{k+1}=F(x_k).
\end{equation}
Its basin is
\[
    \mathcal D_F(A)
    :=
    \left\{
    x\in\mathbb R^d:
    \lim_{k\to\infty}
    \operatorname{dist}\bigl(F^k(x),A\bigr)=0
    \right\}.
\]
An $\varepsilon$-perturbed orbit of $F$ is a sequence satisfying
\begin{equation}
\label{eq:perturbed-map}
    x_{k+1}=F(x_k)+e_k,
    \qquad
    \|e_k\|<\varepsilon.
\end{equation}
Equivalently, $x_{k+1}\in B_\varepsilon(F(x_k))$. For a set $S\subset\mathbb R^d$, define recursively
\begin{align}
\label{eq:discrete-exact-reachable}
    \mathcal R_\varepsilon^F(S,0)
    &:=
    S,\\
    \mathcal R_\varepsilon^F(S,k+1)
    &:=
    \bigcup_{x\in\mathcal R_\varepsilon^F(S,k)}
    B_\varepsilon(F(x)). \nonumber
\end{align}
The finite-step and unrestricted reachable sets are
\begin{align*}
    \mathcal P_\varepsilon^F(S;N)
    &:=
    \bigcup_{k=0}^N
    \mathcal R_\varepsilon^F(S,k),\\
    \mathcal P_\varepsilon^F(S)
    &:=
    \bigcup_{N\geq0}
    \mathcal P_\varepsilon^F(S;N)
     =
    \bigcup_{k\geq0}
    \mathcal R_\varepsilon^F(S,k).
\end{align*}
For $\varepsilon=0$, the recursion is interpreted as $\mathcal R_0^F(S,k+1):=F(\mathcal R_0^F(S,k))$, so that $\mathcal P_0^F(S)$ is the forward orbit of $S$; the literal reading of \eqref{eq:discrete-exact-reachable} would give $B_0=\varnothing$. When $S=A$ is invariant, both readings give $\mathcal P_0^F(A)=A$, but the convention matters for the block intensities of Section~\ref{sec:sampling}, whose seed sets need not be invariant.

\begin{definition}[Discrete-time intensity]
\label{def:discrete-intensity}
The intensity of the attractor $A$ for the map $F$ is
\[
    \mu_F(A)
    :=
    \sup\left\{
    \varepsilon\geq0:
    \overline{\mathcal P_\varepsilon^F(A)}
    \subset K\Subset\mathcal D_F(A)
    \text{ for some compact }K
    \right\}.
\]
\end{definition}

When $F=F_h:=\phi^h$ is the time-$h$ map of \eqref{eq:autonomous-system}, the discrete and continuous basins coincide. Since this identity is used throughout the paper---the discrete intensity of Definition~\ref{def:discrete-intensity} is measured against $\mathcal D_{F_h}(A)$, while the sampling arguments of Section~\ref{sec:sampling} work with compact containment in $\mathcal D(A)$---we record it with a proof.

\begin{lemma}[Basin identity under sampling]
\label{lem:basin-identity}
Let $A$ be an attractor of \eqref{eq:autonomous-system} in the sense of Definition~\ref{def:attractor-basin}, and let $F_h=\phi^h$ for some $h>0$. Then $A$ is a compact attractor of the discrete system $x_{k+1}=F_h(x_k)$, and
\begin{equation}
\label{eq:basin-identity}
    \mathcal D_{F_h}(A)=\mathcal D(A).
\end{equation}
\end{lemma}

\begin{proof}
Invariance $F_h(A)=\phi^h(A)=A$ is immediate. If $x\in\mathcal D(A)$, then $\operatorname{dist}(\phi^t(x),A)\to0$ along all $t\to\infty$, in particular along the grid times $t=kh$; hence $\mathcal D(A)\subset\mathcal D_{F_h}(A)$, and $A$ attracts the neighborhood $U$ of Definition~\ref{def:attractor-basin} under iteration of $F_h$, so $A$ is a discrete attractor.

Conversely, let $x\in\mathcal D_{F_h}(A)$ and set $x_k:=F_h^k(x)=\phi^{kh}(x)$, so that $\operatorname{dist}(x_k,A)\to0$. The set $N:=\{y\in\mathbb R^d:\operatorname{dist}(y,A)\leq1\}$ is compact, and $x_k\in N$ for all sufficiently large $k$. By compactness of $A$, choose $a_k\in A$ with $\|x_k-a_k\|=\operatorname{dist}(x_k,A)\to0$. The map $(s,y)\mapsto\phi^s(y)$ is continuous, hence uniformly continuous on the compact set $[0,h]\times N$, and therefore
\[
    \sup_{0\leq s\leq h}
    \bigl\|\phi^s(x_k)-\phi^s(a_k)\bigr\|
    \longrightarrow0
    \qquad(k\to\infty).
\]
Since $\phi^s(a_k)\in\phi^s(A)=A$ by invariance, every $t=kh+s$ with $s\in[0,h)$ satisfies
\[
    \operatorname{dist}\bigl(\phi^t(x),A\bigr)
    =
    \operatorname{dist}\bigl(\phi^s(x_k),A\bigr)
    \leq
    \bigl\|\phi^s(x_k)-\phi^s(a_k)\bigr\|,
\]
and the right-hand side tends to zero as $t\to\infty$. Hence $x\in\mathcal D(A)$, which proves $\mathcal D_{F_h}(A)\subset\mathcal D(A)$.
\end{proof}

Note that the proof uses only continuity of the flow together with compactness and invariance of $A$. In view of \eqref{eq:basin-identity}, we abbreviate $\mu_h(A):=\mu_{F_h}(A)$.

\begin{remark}[Physical dimensions]
\label{rem:dimensions}
With state unit $[X]$ and time unit $[T]$, $[\mu(A)]=[X]/[T]$ while the discrete kick has unit $[X]$, so $[\mu_F(A)]=[X]$; the dimensionally consistent continuous-time quantity for a time-$h$ map is $\mu_h(A)/h$, which Section~\ref{sec:sampling} shows converges to $\mu(A)$.
\end{remark}

The same distinction applies to a numerical one-step map $\Psi_h$. Its discrete intensity is measured per numerical step, while $\mu_h^\Psi/h$ is the corresponding disturbance rate per unit physical time.

\subsection{Reachability algebra and trapping regions}
\label{subsec:reachability-properties}

While Section~\ref{subsec:flow-intensity} establish the theoretical framework for robust intensity, calculating the infinite-time reachable sets directly is generally intractable. In this subsection, 
we provide the operational tools to estimate these sets. We first establish the basic algebraic properties of reachability, and then introduce geometric trapping regions to obtain verifiable lower bounds on the intensities.
Although the proofs in this part are standard in control theory \citep{colonius2000dynamics}, we include them to fix the precise strict-control convention.

\begin{lemma}[Monotonicity and concatenation]
\label{lem:reachability-algebra}
For $S,S_1,S_2\subset\mathbb R^d$: (i) $S_1\subset S_2$, $r_1\le r_2$ imply $\mathcal R_{r_1}(S_1,t)\subset\mathcal R_{r_2}(S_2,t)$ and $\mathcal P_{r_1}(S_1)\subset\mathcal P_{r_2}(S_2)$; (ii) $\mathcal R_r(\mathcal R_r(S,t),s)=\mathcal R_r(S,t+s)$ \label{eq:continuous-concatenation}; (iii) $\mathcal P_r(\mathcal P_r(S))=\mathcal P_r(S)$ is the smallest controlled forward invariant set containing $S$ \label{eq:reachable-idempotent}; (iv) if $A$ is flow-invariant then $\phi^\tau(\mathcal P_r(A))\subset\mathcal P_r(A)$ for $\tau\ge0$ \label{eq:flow-forward-invariance}; and (v) the analogous identities hold for the discrete reachable sets $\mathcal R^F_\varepsilon,\mathcal P^F_\varepsilon$
\label{eq:discrete-concatenation}.
\end{lemma}

\begin{proof}
The first assertion follows immediately by enlarging the set of initial conditions and the admissible control class. To prove \eqref{eq:continuous-concatenation}, concatenate a control on $[0,t]$ with a control on $[t,t+s]$. The concatenated control remains admissible because the maximum of two numbers strictly smaller than $r$ is still strictly smaller than $r$. Conversely, every admissible control on $[0,t+s]$ can be restricted to the two subintervals. This proves equality.

Equation \eqref{eq:reachable-idempotent} follows from concatenation, while the inclusion in the opposite direction follows by choosing zero additional time. It also proves the minimality statement. To obtain \eqref{eq:flow-forward-invariance}, append the zero control for time $\tau$ to a control that reaches the given point from $A$. The discrete statements follow by concatenating finite sequences of perturbations.
\end{proof}

The next result gives a convenient sufficient condition for positive intensity.

\begin{definition}[Robust trapping region]
\label{def:robust-trapping-region} 
Let $B\subset\mathbb R^d$ be compact with $B=\overline{\operatorname{int}B}$, $A\subset\operatorname{int}B \subset B\subset\mathcal D(A)$. We call $B$ an $r$-robust trapping region if every controlled trajectory of \eqref{eq:controlled-system} starting in $B$ remains in $B$ for all forward times whenever $\|u\|_\infty<r$.
\end{definition}

\begin{proposition}[Inward-pointing criterion]
\label{prop:inward-trapping}
Suppose that $B$ has $C^1$ boundary and let $n_B(x)$ be its outward unit normal at $x\in\partial B$. If there exist $r>0$ and $\gamma>0$ such that
\begin{equation}
\label{eq:inward-condition-ball}
    \langle n_B(x),f(x)\rangle+r
    \leq-\gamma
    \qquad
    \text{for every }x\in\partial B,
\end{equation}
then $B$ is an $r$-robust trapping region. Consequently,
$\overline{\mathcal P_r(A)}
    \subset B\Subset\mathcal D(A)$
and
$\mu(A)\geq r$.
\end{proposition}

\begin{proof}
For every $u$ satisfying $\|u\|_\infty<r$ and almost every time at which a controlled trajectory lies on $\partial B$,
\[
    \langle n_B(x),f(x)+u(t)\rangle\leq
    \langle n_B(x),f(x)\rangle+\|u(t)\|<
    \langle n_B(x),f(x)\rangle+r
    \leq-\gamma.
\]
Thus all admissible velocities point strictly into $B$ at its boundary. A first-exit argument, equivalently the standard viability theorem, implies positive invariance of $B$. Since $A\subset B$, every point reachable from $A$ lies in $B$. Compactness of $B$ and $B\subset\mathcal D(A)$ give the asserted compact containment.
\end{proof}

For nonsmooth trapping regions, condition \eqref{eq:inward-condition-ball} may be replaced by a tangent-cone condition in the sense of viability theory \citep{aubin1982differential}. The smooth formulation is sufficient for the geometric applications considered here.

There is a direct discrete analogue.

\begin{proposition}[Discrete trapping criterion]
\label{prop:discrete-trapping}
Let $F:\mathbb R^d\to\mathbb R^d$ be continuous and let
    $A\subset\operatorname{int}B
    \subset B\Subset\mathcal D_F(A)$,
where $B$ is compact. If
\begin{equation}
\label{eq:discrete-trapping-condition}
    \overline B_\varepsilon(F(x))
    \subset\operatorname{int}B
    \qquad
    \text{for every }x\in B,
\end{equation}
then
$\overline{\mathcal P_\varepsilon^F(A)}
    \subset B$
and
$\mu_F(A)\geq\varepsilon$.
\end{proposition}

\begin{proof}
Condition \eqref{eq:discrete-trapping-condition} and induction imply that every $\varepsilon$-perturbed orbit beginning in $A$ remains in $B$. Taking the union over all steps and then its closure yields the result.
\end{proof}

\subsection{Boundary-contact characterization and structural properties}
\label{subsec:first-contact}

In this subsection, we establish three structural properties of intensity: its equivalence to a boundary-contact threshold, its covariance under time rescaling, and its generalization to anisotropic disturbance geometries.
The compact-containment definition suggests characterizing intensity as the smallest disturbance amplitude at which the reachable set contacts the basin boundary. Such a characterization is valid when loss of compact containment cannot occur by escape to infinity. The qualification is automatic for bounded basins but must be stated for unbounded ones.

Define
    $\mu_{\partial}(A)
    :=
    \inf\left\{
    r>0:
    \overline{\mathcal P_r(A)}
    \cap\partial\mathcal D(A)
    \neq\varnothing
    \right\}$,
where the infimum of the empty set is $+\infty$.

\begin{proposition}[First-contact characterization]
\label{prop:first-contact}
Suppose that one of the following conditions holds:
\begin{enumerate}
\item the basin $\mathcal D(A)$ is bounded; or
\item there exists a compact set $Q\subset\mathbb R^d$ such that
$\overline{\mathcal P_r(A)}\subset Q$ for every $0\leq r<\mu_{\partial}(A)$.
\end{enumerate}
Then
\begin{equation}
\label{eq:intensity-first-contact}
    \mu(A)=\mu_{\partial}(A).
\end{equation}
Equivalently, intensity is the smallest control amplitude at which the closure of the reachable set first contacts the basin boundary.
\end{proposition}

\begin{proof}
Let $0\leq r<\mu_{\partial}(A)$. By definition,
    $\overline{\mathcal P_r(A)}
    \cap\partial\mathcal D(A)=\varnothing$.
A controlled trajectory beginning in $A\subset\mathcal D(A)$ cannot leave the open basin without first meeting its boundary. Hence
    $\overline{\mathcal P_r(A)}
    \subset\mathcal D(A)$.
Under condition (ii), the reachable-set closure is a closed subset of the compact set $Q$ and is therefore compact. Under condition (i), the same conclusion follows from
    $\overline{\mathcal P_r(A)}
    \subset\overline{\mathcal D(A)}$,
because the closure of a bounded basin in $\mathbb R^d$ is compact. Thus every $r<\mu_{\partial}(A)$ is admissible in Definition~\ref{def:continuous-intensity}, and $\mu(A)\geq\mu_{\partial}(A)$. Conversely, let $r>\mu_{\partial}(A)$. By monotonicity of reachable sets in the amplitude, the set of amplitudes at which $\overline{\mathcal P_r(A)}$ meets $\partial\mathcal D(A)$ is upward closed; by the definition of the infimum it therefore contains every $r>\mu_{\partial}(A)$, so $\overline{\mathcal P_r(A)}$ meets $\partial\mathcal D(A)$. It therefore cannot be contained in any compact subset of $\mathcal D(A)$. Consequently, $r$ is not admissible in \eqref{eq:continuous-intensity}, and $\mu(A)\leq\mu_{\partial}(A)$. Combining the inequalities proves \eqref{eq:intensity-first-contact}.
\end{proof}

\begin{remark}[Escape to infinity]
\label{rem:escape-infinity}
Without a hypothesis of Proposition~\ref{prop:first-contact} the identity can fail: if $\mathcal D(A)=\mathbb R^d$ then $\partial\mathcal D(A)=\varnothing$ yet strong controls make the reachable set unbounded, so $\mu_\partial(A)=+\infty$ while $\mu(A)$ is finite.
\end{remark}

For discrete maps, a perturbed orbit may jump from the basin to its complement without meeting the boundary at an intermediate state. The appropriate discrete threshold therefore refers to the complement of the basin rather than to literal boundary contact. We now make this precise; the result is used in Section~\ref{sec:noise-boundary} to identify sampled intensity with a worst-case bounded-noise escape threshold.

\begin{proposition}[Discrete first-contact characterization]
\label{prop:discrete-first-contact}
Let $F:\mathbb R^d\to\mathbb R^d$ be continuous and let $A$ be a compact attractor of \eqref{eq:discrete-system} that attracts an open neighborhood $U$. Define the exit threshold
\begin{equation}
\label{eq:discrete-exit-threshold}
    \varepsilon_{\mathrm{exit}}(A)
    :=
    \inf\left\{
    \varepsilon>0:
    \begin{array}{l}
    \text{some finite $\varepsilon$-perturbed orbit
    \eqref{eq:perturbed-map}}\\
    \text{with }x_0\in A
    \text{ satisfies }
    x_N\notin\mathcal D_F(A)
    \end{array}
    \right\},
\end{equation}
with $\inf\varnothing:=+\infty$. Suppose that there exists a compact set $Q\subset\mathbb R^d$ such that
\begin{equation}
\label{eq:discrete-no-escape-infinity}
    \overline{\mathcal P_\varepsilon^F(A)}\subset Q
    \qquad
    \text{for every }
    0\leq\varepsilon<\varepsilon_{\mathrm{exit}}(A).
\end{equation}
Then
\begin{equation}
\label{eq:discrete-first-contact-identity}
    \mu_F(A)=\varepsilon_{\mathrm{exit}}(A),
\end{equation}
and, moreover, $\overline{\mathcal P_\varepsilon^F(A)}\subset\mathcal D_F(A)$ for every $0\leq\varepsilon<\varepsilon_{\mathrm{exit}}(A)$.
\end{proposition}

\begin{proof}
We first note that $\mathcal D_F(A)$ is open. If $x\in\mathcal D_F(A)$, then $F^k(x)\in U$ for some $k\geq0$; conversely, $F^k(x)\in U$ for some $k$ implies $\operatorname{dist}(F^j(x),A)\to0$ as $j\to\infty$, because $A$ attracts $U$. Hence
\[
    \mathcal D_F(A)
    =
    \bigcup_{k\geq0}F^{-k}(U)
\]
is a union of preimages of an open set under a continuous map.

\emph{Step 1: $\mu_F(A)\leq\varepsilon_{\mathrm{exit}}(A)$.}
We may assume $\varepsilon_{\mathrm{exit}}(A)<\infty$. Let $\varepsilon>\varepsilon_{\mathrm{exit}}(A)$. Since the admissible kick sequences grow with $\varepsilon$, the defining set in \eqref{eq:discrete-exit-threshold} is upward closed, so there is a finite $\varepsilon$-perturbed orbit from $A$ with endpoint $x_N\notin\mathcal D_F(A)$. Then $x_N\in\mathcal P_\varepsilon^F(A)\setminus\mathcal D_F(A)$, so $\overline{\mathcal P_\varepsilon^F(A)}$ is contained in no compact subset of $\mathcal D_F(A)$, and $\varepsilon$ is not admissible in Definition~\ref{def:discrete-intensity}. Letting $\varepsilon\downarrow\varepsilon_{\mathrm{exit}}(A)$ proves the inequality.

\emph{Step 2: $\mu_F(A)\geq\varepsilon_{\mathrm{exit}}(A)$.}
Let $0\leq\varepsilon<\varepsilon_{\mathrm{exit}}(A)$ and choose $\varepsilon'\in(\varepsilon,\varepsilon_{\mathrm{exit}}(A))$. Since $\varepsilon'<\varepsilon_{\mathrm{exit}}(A)$, every finite $\varepsilon'$-perturbed orbit from $A$ remains in $\mathcal D_F(A)$---any point of such an orbit is the endpoint of a truncated orbit---so $\mathcal P_{\varepsilon'}^F(A)\subset\mathcal D_F(A)$.

We claim that
$\overline{\mathcal P_\varepsilon^F(A)}
    \subset
    A\cup\mathcal P_{\varepsilon'}^F(A)$.
Let $z=\lim_j z_j$ with $z_j\in\mathcal P_\varepsilon^F(A)$. If infinitely many $z_j$ are zero-step endpoints, they lie in $A$, and $z\in A$ by closedness. Otherwise, fix $j$ with $\|z-z_j\|<\varepsilon'-\varepsilon$ such that $z_j$ is the endpoint of an orbit with $N\geq1$ steps and final kick $e_{N-1}$; replacing that kick by $e_{N-1}+(z-z_j)$, of norm smaller than $\varepsilon+(\varepsilon'-\varepsilon)=\varepsilon'$, produces an $\varepsilon'$-perturbed orbit from $A$ with endpoint $z$, so $z\in\mathcal P_{\varepsilon'}^F(A)$. This proves the claim, and hence   
$\overline{\mathcal P_\varepsilon^F(A)}
    \subset
    A\cup\mathcal P_{\varepsilon'}^F(A)
    \subset
    \mathcal D_F(A)$.
By \eqref{eq:discrete-no-escape-infinity}, $\overline{\mathcal P_\varepsilon^F(A)}$ is a compact subset of the open set $\mathcal D_F(A)$ and is therefore compactly contained in it. Taking $K=\overline{\mathcal P_\varepsilon^F(A)}$ in Definition~\ref{def:discrete-intensity} shows that $\varepsilon$ is admissible, so $\mu_F(A)\geq\varepsilon$ for every $\varepsilon<\varepsilon_{\mathrm{exit}}(A)$. 
\end{proof}

\begin{remark}[Variants]
\label{rem:discrete-first-contact-variants}
The infimum is unchanged if the kick constraint is relaxed to $\|e_k\|\le\varepsilon$; the statement carries over to the gauge $\gamma_C$ (using $0\in\operatorname{int}C$); and under \eqref{eq:discrete-no-escape-infinity} the exit threshold equals the amplitude at which $\overline{\mathcal P^F_\varepsilon(A)}$ first meets $\mathbb R^d\setminus\mathcal D_F(A)$.
\end{remark}

Beyond the boundary-contact characterization, the intensity measure admits two fundamental structural extensions regarding time and space. First, as a rate, intensity transforms covariantly under time rescaling. Second, the Euclidean disturbance ball can be replaced by an anisotropic convex body to account for heterogeneous physical scales.

For $\lambda>0$, define $f_\lambda(x):=\lambda f(x)$ and let $\mu_{f_\lambda}(A)$ denote the intensity of $A$ for $\dot x=\lambda f(x)$.

\begin{proposition}[Covariance under time rescaling]
\label{prop:time-rescaling}
For every $\lambda>0$,
\begin{equation}
\label{eq:intensity-vector-field-scaling}
    \mu_{f_\lambda}(A)
    =
    \lambda\mu_f(A).
\end{equation}
More precisely, their reachable sets satisfy
\begin{equation}
\label{eq:reachable-time-scaling}
    \mathcal P_r^{\,f_\lambda}(A)
    =
    \mathcal P_{r/\lambda}^{\,f}(A).
\end{equation}
\end{proposition}

\begin{proof}
Consider
$\dot x(t)=\lambda f(x(t))+u(t)$
and introduce the new time variable $s=\lambda t$. Then
    $\frac{dx}{ds}
    =
    f(x)+\frac1\lambda u\left(\frac{s}{\lambda}\right)$.
The rescaled control
    $v(s):=
    \frac1\lambda u\left(\frac{s}{\lambda}\right)$
satisfies
    $\|v\|_\infty=\frac1\lambda\|u\|_\infty$.
Because both systems are considered over arbitrary finite time intervals, this correspondence is bijective at the level of reachable sets, proving \eqref{eq:reachable-time-scaling}. Note also that $f$ and $f_\lambda$ have the same orbits up to time reparametrization, so $A$ is an attractor for both systems with the same basin $\mathcal D(A)$; the two intensities are therefore measured against the same set. Substitution into Definition~\ref{def:continuous-intensity} yields \eqref{eq:intensity-vector-field-scaling}.
\end{proof}

To state the result directly in terms of time units, suppose that $s=\kappa t$. If the original equation is written in $t$-time as
    $\frac{dx}{dt}=f(x)$,
then in $s$-time it becomes $\frac{dx}{ds}=\frac1\kappa f(x)$, and therefore $\mu_s(A)=\frac1\kappa\mu_t(A)$.
Equivalently, if $t=\kappa s$, then
\[
    \frac{dx}{ds}=\kappa f(x),
    \qquad
    \mu_s(A)=\kappa\mu_t(A).
\]
For example, if $\kappa$ is the number of seconds in one year, then an intensity expressed per year equals $\kappa$ times the same physical intensity expressed per second. This covariance must be taken into account when resilience values obtained using different time units are compared.

The second extension is anisotropic disturbance sets. The Euclidean disturbance ball assumes that all state-space directions are equally accessible and equally costly. This is often inappropriate: different state variables may have different physical scales, some directions may be inaccessible, and forcing may obey asymmetric constraints. We therefore introduce an anisotropic version of intensity.

Let $C\subset\mathbb R^d$ be a compact convex set satisfying $0\in\operatorname{int}C$. The set need not be centrally symmetric. Its Minkowski functional is
$\gamma_C(v)
    :=
    \inf\{\rho\geq0:v\in\rho C\}$.
If $C$ is centrally symmetric, $\gamma_C$ is a norm whose closed unit ball is $C$. In the asymmetric case it remains positively homogeneous but generally satisfies $\gamma_C(v)\neq\gamma_C(-v)$. For $r>0$, define
\[
    \mathcal U_{r,C}(I)
    :=
    \left\{
    u\in L^\infty(I;\mathbb R^d):
    \operatorname*{ess\,sup}_{t\in I}\gamma_C(u(t))<r
    \right\}.
\]
The closure of the corresponding controlled dynamics is described by the differential inclusion
\begin{equation}
\label{eq:convex-differential-inclusion}
    \dot x(t)\in f(x(t))+rC.
\end{equation}
Reachable sets $\mathcal R_{r,C}(S,t)$ and $\mathcal P_{r,C}(S)$ are defined as in \eqref{eq:exact-reachable-flow}--\eqref{eq:infinite-reachable-flow}, with $\mathcal U_r$ replaced by $\mathcal U_{r,C}$.

\begin{definition}[$C$-intensity]
\label{def:C-intensity}
The intensity of $A$ relative to the disturbance body $C$ is
\[
    \mu_C(A)
    :=
    \sup\left\{
    r\geq0:
    \overline{\mathcal P_{r,C}(A)}
    \subset K\Subset\mathcal D(A)
    \text{ for some compact }K
    \right\}.
\]
\end{definition}

When $C=\overline B_1(0)$, the Minkowski functional is the Euclidean norm and $\mu_C(A)=\mu(A)$. Scaling the disturbance body and scaling the intensity parameter are redundant:
\[
    \mu_{\alpha C}(A)
    =
    \frac1\alpha\mu_C(A),
    \qquad \alpha>0.
\]
The support function of $C$ is
\begin{equation}
\label{eq:support-function}
    h_C(n)
    :=
    \max_{c\in C}\langle n,c\rangle,
    \qquad n\in\mathbb R^d.
\end{equation}
It gives the maximal instantaneous disturbance velocity in direction $n$. Consequently, the anisotropic version of
\eqref{eq:inward-condition-ball} is
\begin{equation}
\label{eq:anisotropic-inward-condition}
    \langle n_B(x),f(x)\rangle
    +
    r h_C\bigl(n_B(x)\bigr)
    \leq-\gamma,
    \qquad x\in\partial B.
\end{equation}
Under \eqref{eq:anisotropic-inward-condition}, the proof of Proposition~\ref{prop:inward-trapping} shows that $B$ is invariant for \eqref{eq:convex-differential-inclusion} and hence $\mu_C(A)\geq r$.

The following comparison relates anisotropic and Euclidean intensity.

\begin{proposition}[Comparison of disturbance geometries]
\label{prop:C-intensity-comparison}
Suppose that
\begin{equation}
\label{eq:C-ball-comparison}
    \rho\,\overline B_1(0)
    \subset C
    \subset
    R\,\overline B_1(0)
\end{equation}
for some $0<\rho\leq R<\infty$. Then
    $\frac{\mu(A)}{R}
    \leq
    \mu_C(A)
    \leq
    \frac{\mu(A)}{\rho}$.
\end{proposition}

\begin{proof}
The inclusions in \eqref{eq:C-ball-comparison} imply
    $\mathcal P_{r\rho}(A)
    \subset
    \mathcal P_{r,C}(A)
    \subset
    \mathcal P_{rR}(A)$.
If $rR<\mu(A)$, then $\mathcal P_{rR}(A)$ is compactly contained in the basin, and so is $\mathcal P_{r,C}(A)$. Hence $\mu_C(A)\geq\frac{\mu(A)}{R}$. Conversely, if $\mathcal P_{r,C}(A)$ is compactly contained in the basin, then so is $\mathcal P_{r\rho}(A)$, implying $r\rho\leq\mu(A)$. Taking the supremum over admissible $r$ gives the upper bound.
\end{proof}

A discrete $C$-intensity is defined analogously by replacing \eqref{eq:perturbed-map} with $x_{k+1}=F(x_k)+e_k$, $\gamma_C(e_k)<\varepsilon$. We denote it by $\mu_{F,C}(A)$. If $F=F_h=\phi^h$, then $\mu_{h,C}(A)/h$ is the discrete disturbance rate measured in the geometry determined by $C$. The sampling arguments of Section~\ref{sec:sampling} extend to this setting, with constants adjusted as described in Remark~\ref{rem:anisotropic-sampling-constants} below. For
smooth strictly convex $C$, the maximizer in \eqref{eq:support-function} also determines the anisotropic boundary dynamics developed in Section~\ref{sec:noise-boundary}; see also \citep{kourliouros2026invariant}.

\begin{remark}[Anisotropic sampling constants]
\label{rem:anisotropic-sampling-constants}
The reachability arguments of Section~\ref{sec:sampling} extend to anisotropic disturbance sets by replacing the Euclidean norm with the associated gauge $\gamma_C$. More precisely, the proofs use only the positive homogeneity and subadditivity of the control functional, together with the integral estimates arising from the interpolation and Gr\"onwall arguments. These properties remain valid for a general convex body $C$ containing the origin in its interior, although the gauge is not necessarily symmetric: $\gamma_C(v)\neq \gamma_C(-v)$ may occur. Thus the arguments must be interpreted as one-sided gauge estimates rather than estimates generated by a symmetric distance.

If \eqref{eq:C-ball-comparison} holds, then $\frac{\|v\|}{R} \leq \gamma_C(v) \leq \frac{\|v\|}{\rho}$ for $v\in\mathbb R^d$, and an $L$-Lipschitz vector field satisfies
\[
    \gamma_C\bigl(f(x)-f(y)\bigr)
    \leq
    \frac{L}{\rho}\|x-y\|
    \leq
    L\frac{R}{\rho}\gamma_C(x-y)
    =:
    L_C\gamma_C(x-y).
\]
Consequently, the interpolation estimate of
Lemma~\ref{lem:discrete-step-interpolation} and the one-step error bound
of Lemma~\ref{lem:one-step-control-error} remain valid with $L_C$ in
place of the Euclidean Lipschitz constant, and the time-alignment
argument of Lemma~\ref{lem:sampling-time-alignment} carries over.

Therefore,
\[
    \frac{h\,\mu_C(A)}{1+L_Ch}
    \leq
    \mu_{h,C}(A)
    \leq
    \mu_C(A)\frac{e^{L_Ch}-1}{L_C},
    \qquad
    \frac{\mu_{h,C}(A)}{h}
    =
    \mu_C(A)+O(h).
\]
The factor $R/\rho$ reflects the distortion between the Euclidean norm
and the gauge geometry and cannot be removed without additional
structure. If $f$ is assumed Lipschitz directly with respect to
$\gamma_C$, then the corresponding gauge-Lipschitz constant replaces
$L_C$.
\end{remark}

\section{Sampling and numerical discretization}
\label{sec:sampling}

We now compare the continuous-time intensity of an attractor with the
discrete intensity of its exact time-$h$ map. The main result gives
explicit two-sided bounds and proves that the normalized discrete
intensity converges to its continuous counterpart. The proof rests on
two complementary constructions. First, every perturbed orbit of the
time-$h$ map can be interpolated by a controlled trajectory of the
flow. Second, every controlled trajectory can be sampled to produce a
perturbed orbit of the time-$h$ map. Invariance of the attractor is
then used to align arbitrary continuous-time endpoints with the
sampling grid.

Throughout Sections~\ref{subsec:main-sampling-theorem}--%
\ref{subsec:sampling-consequences}, we consider
$\dot x=f(x)$, $x\in\mathbb R^d$,
with flow $\phi^t$, and write
$F_h:=\phi^h$.
For the principal theorem we impose the following assumptions.
\begin{enumerate}
\renewcommand{\theenumi}{(H\arabic{enumi})}
\renewcommand{\labelenumi}{\textnormal{\theenumi}}
\item
\label{hyp:forward-complete}
The vector field $f:\mathbb R^d\to\mathbb R^d$ is globally
$L$-Lipschitz for some $L\geq0$. In particular, the uncontrolled and
bounded-control systems are forward complete.
\item
\label{hyp:compact-attractor}
The set $A\subset\mathbb R^d$ is a nonempty compact attractor and is
invariant in the strong sense
    $\phi^t(A)=A$, $t\geq0$.
\item
\label{hyp:positive-finite-intensity}
Its continuous intensity satisfies
    $0<\mu(A)<\infty$.
\end{enumerate}
The global Lipschitz assumption can be localized to a common compact
trapping region; see Remark~\ref{rem:localized-sampling}. We first give
the global statement so that the reachability argument is transparent.

For later use, define
\begin{equation}
\label{eq:alpha-h-definition}
    \alpha_h(r)
    :=
    \begin{cases}
    \displaystyle
    r\frac{e^{Lh}-1}{L},
        & L>0,\\[2ex]
    rh, & L=0,
    \end{cases}
\end{equation}
and
\begin{equation}
\label{eq:beta-h-definition}
    \beta_h(\varepsilon)
    :=
    \left(\frac1h+L\right)\varepsilon.
\end{equation}

\subsection{The exact sampling theorem}
\label{subsec:main-sampling-theorem}

We now arrive at the principal result of this part: a quantitative correspondence between the continuous-time intensity $\mu(A)$ and the sampled discrete-time intensity $\mu_h(A)$. The following theorem establishes explicit upper and lower bounds, confirming that $\mu_h(A)/h$ converges to $\mu(A)$ as the sampling period vanishingly approaches zero; see also Figure~\ref{fig:sampling-bridge}.

\begin{theorem}[Intensity under exact sampling]
\label{thm:sampling}
Suppose that \ref{hyp:forward-complete}--\ref{hyp:positive-finite-intensity} hold. For every $h>0$, let $\mu_h(A)$ denote the discrete intensity of $A$ for the time-$h$ map $F_h=\phi^h$. Then
\begin{equation}
\label{eq:sampling-main-bounds}
    \frac{h\mu(A)}{1+Lh}
    \leq
    \mu_h(A)
    \leq
    \alpha_h\bigl(\mu(A)\bigr).
\end{equation}
Equivalently, when $L>0$,
\[
    \frac{\mu(A)}{1+Lh}
    \leq
    \frac{\mu_h(A)}{h}
    \leq
    \mu(A)\frac{e^{Lh}-1}{Lh}.
\]
For $L=0$, both bounds are interpreted by continuity and give $\mu_h(A)=h\mu(A)$. Consequently,
\begin{equation}
\label{eq:sampling-main-limit}
    \lim_{h\downarrow0}\frac{\mu_h(A)}{h}
    =
    \mu(A).
\end{equation}
\end{theorem}

\begin{figure}[htbp]
  \centering
    \includegraphics[width=0.6\linewidth]{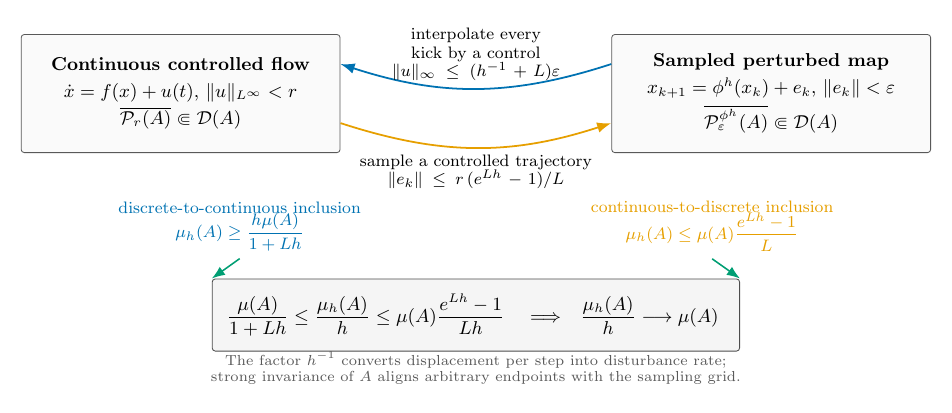}
  \caption{The two reachability constructions underlying the exact sampling
  theorem.  A discrete kick can be interpolated by a bounded continuous
  control (blue), while a controlled trajectory sampled every $h$ units of
  time produces a discrete pseudotrajectory (orange).  The resulting
  inclusions bracket the per-unit-time quantity $\mu_h(A)/h$ and yield its
  convergence to the continuous intensity $\mu(A)$.}
  \label{fig:sampling-bridge}
\end{figure}

\begin{remark}[Zero and infinite intensity]
\label{rem:zero-infinite-intensity}
The assumption $0<\mu(A)<\infty$ isolates the nontrivial case. If $\mu(A)=0$, the upper bound proved below implies $\mu_h(A)=0$ for every $h>0$. Infinite intensity requires a separate extended-real formulation because the upper bound in \eqref{eq:sampling-main-bounds} then carries no finite information.
\end{remark}

\begin{remark}[Localized hypotheses]
\label{rem:localized-sampling}
The proof does not intrinsically require global Lipschitz continuity. It is enough that there exist an open set $U$ and a compact set $K\Subset U$ such that:
\begin{enumerate}
\item $f$ is $L$-Lipschitz on $U$;
\item the continuous and discrete reachable sets used at all
subcritical amplitudes are contained in $K$;
\item for every discrete step occurring in the interpolation argument,
the segments
    $\left\{
    \phi^s(x)+\theta e:
    0\leq s\leq h,\;
    0\leq\theta\leq1
    \right\}$
remain in $U$.
\end{enumerate}
These conditions hold, for example, when a common robust trapping region is contained a positive distance inside $U$ and the admissible one-step perturbations are smaller than that distance. Thus the same bounds hold with a Lipschitz constant valid only on a common trapping neighborhood.
\end{remark}

The proof of Theorem~\ref{thm:sampling} relies on two directional reachability inclusions (continuous-to-discrete and discrete-to-continuous), which we establish next before completing the argument.

\subsection{Relationship between continuous controls and discrete perturbations}
\label{subsec:discrete-to-continuous}

To prove Theorem~\ref{thm:sampling}, we establish a two-way bounding relationship between continuous and discrete reachable sets. We first embed discrete perturbed orbits into continuous controlled trajectories to establish the lower bound, and then reverse the process to establish the upper bound.

Consider one step of an $\varepsilon$-perturbed orbit:
    $x_{k+1}=\phi^h(x_k)+e_k$,
    $\|e_k\|<\varepsilon$.
For $0\leq s\leq h$, define
\begin{equation}
\label{eq:linear-step-interpolation}
    \widetilde x_k(s)
    :=
    \phi^s(x_k)+\frac{s}{h}e_k.
\end{equation}
Then $\widetilde x_k(0)=x_k$, $\widetilde x_k(h)=x_{k+1}$. \eqref{eq:linear-step-interpolation} is a linear interpolation of the endpoint error, and this choice makes it possible to retain the exact flow segment $\phi^s(x_k)$ in the estimate.

\begin{lemma}[Interpolation of one perturbed step]
\label{lem:discrete-step-interpolation}
The curve $\widetilde x_k$ defined by \eqref{eq:linear-step-interpolation} is an absolutely continuous solution of \begin{equation}
\label{eq:interpolation-controlled-equation}
    \dot{\widetilde x}_k(s)
    =
    f\bigl(\widetilde x_k(s)\bigr)+u_k(s),
    \qquad 0\leq s\leq h,
\end{equation}
where
\begin{equation}
\label{eq:interpolation-control}
    u_k(s)
    =
    f\bigl(\phi^s(x_k)\bigr)
    -
    f\bigl(\widetilde x_k(s)\bigr)
    +
    \frac1h e_k.
\end{equation}
Moreover,
\begin{equation}
\label{eq:interpolation-control-bound-refined}
    \|u_k(s)\|
    \leq
    \left(\frac1h+\frac{Ls}{h}\right)\|e_k\|
    \leq
    \left(\frac1h+L\right)\|e_k\|
\end{equation}
for almost every $s\in[0,h]$.
\end{lemma}

\begin{proof}
Differentiating \eqref{eq:linear-step-interpolation} for almost every $s$ gives
    $\dot{\widetilde x}_k(s)
    =
    f\bigl(\phi^s(x_k)\bigr)+\frac1h e_k$. Adding and subtracting $f(\widetilde x_k(s))$ gives \eqref{eq:interpolation-controlled-equation} and \eqref{eq:interpolation-control}. By Lipschitz continuity,
\[
    \|u_k(s)\|
    \leq
    L\left\|
    \phi^s(x_k)-\widetilde x_k(s)
    \right\|
    +\frac1h\|e_k\|=
    L\frac{s}{h}\|e_k\|
    +\frac1h\|e_k\|,
\]
which proves \eqref{eq:interpolation-control-bound-refined}.
\end{proof}

Concatenating the controls from Lemma~\ref{lem:discrete-step-interpolation} produces an admissible continuous control over the entire discrete trajectory.

\begin{proposition}[Discrete-to-continuous reachability]
\label{prop:discrete-to-continuous}
For every $h>0$, $\varepsilon>0$, and $S\subset\mathbb R^d$,
\begin{equation}
\label{eq:discrete-to-continuous-inclusion}
    \mathcal P_\varepsilon^{F_h}(S)
    \subset
    \mathcal P_{\beta_h(\varepsilon)}(S),
\end{equation}
where $\beta_h$ is defined by \eqref{eq:beta-h-definition}. In particular,
    $\mathcal P_\varepsilon^{F_h}(A)
    \subset
    \mathcal P_{\varepsilon/h+L\varepsilon}(A)$.
\end{proposition}

\begin{proof}
Let $x_0,\ldots,x_N$ be an $\varepsilon$-perturbed orbit of $F_h$. On the interval $[kh,(k+1)h]$, use the time-translated interpolation $\widetilde x(kh+s):=\widetilde x_k(s)$, $0\leq s\leq h$. The endpoint identities imply that these pieces form a continuous, absolutely continuous curve. Define the concatenated control by $u(kh+s):=u_k(s)$. Since every $\|e_k\|<\varepsilon$, Lemma~\ref{lem:discrete-step-interpolation} gives
    $\|u\|_\infty
    <
    \left(\frac1h+L\right)\varepsilon
    =
    \beta_h(\varepsilon)$.
The controlled curve starts at $x_0$ and passes through $x_k$ at time $kh$. Hence every endpoint of the perturbed orbit belongs to the right-hand side of \eqref{eq:discrete-to-continuous-inclusion}. Taking the union over all finite perturbed orbits proves the result.
\end{proof}

\begin{remark}[The interpolation constant]
\label{rem:interpolation-constant}
The $\varepsilon/h$ term is forced by generating a displacement $\varepsilon$ over time $h$; the $L\varepsilon$ term corrects for evaluating $f$ on the interpolated curve.
\end{remark}

Having established the discrete-to-continuous inclusion, we now reverse the direction. To show that continuous reachability implies discrete reachability, we sample a continuous controlled solution at grid intervals of length $h$. The resulting one-step deviation is bounded using Gr\"onwall's inequality.

\begin{lemma}[One-step control error]
\label{lem:one-step-control-error}
Let $x_u(\cdot;x)$ solve
    $\dot x_u=f(x_u)+u(t)$,
    $x_u(0;x)=x$,
on $[0,h]$, and suppose that $\|u\|_{\infty,[0,h]}<r$. Then
    $\left\|
    x_u(h;x)-\phi^h(x)
    \right\|
    <
    \alpha_h(r)$,
where $\alpha_h$ is defined by \eqref{eq:alpha-h-definition}.
\end{lemma}

\begin{proof}
Set
    $z(t):=x_u(t;x)-\phi^t(x)$.
Using the integral forms of the controlled and uncontrolled equations,
\[
    z(t)
    =
    \int_0^t
    \left[
    f\bigl(x_u(s;x)\bigr)
    -
    f\bigl(\phi^s(x)\bigr)
    \right]ds
    +
    \int_0^t u(s)\,ds.
\]
Set $r':=\|u\|_{\infty,[0,h]}<r$. Then
    $\|z(t)\|
    \leq
    L\int_0^t\|z(s)\|\,ds+r't$,
and the variation-of-constants form of Gr\"onwall's inequality gives
    $\|z(t)\|
    \leq
    r'\int_0^t e^{L(t-s)}\,ds$.
At $t=h$, this becomes
\[
    \|z(h)\|
    \leq
    \begin{cases}
    \displaystyle
    r'\frac{e^{Lh}-1}{L},&L>0,\\[2ex]
    r'h,&L=0,
    \end{cases}
\]
which is strictly smaller than $\alpha_h(r)$ because $r'<r$.
\end{proof}

Sampling a controlled trajectory at times $kh$ therefore gives an $\alpha_h(r)$-perturbed orbit. A direct sampling approach bounds the error at exact grid times $kh$. However, a continuous trajectory may reach its target at an arbitrary time $T$ that is not an integer multiple of $h$. The next lemma resolves this grid-alignment problem by exploiting the exact flow-invariance of the attractor, allowing us to append a zero-cost time shift.

\begin{lemma}[Alignment with the sampling grid]
\label{lem:sampling-time-alignment}
Let $A$ satisfy $\phi^t(A)=A$ for every $t\geq0$. Given
$a\in A$, $T\geq0$, and an admissible control $u$ on $[0,T]$, there
exist an integer $N\geq0$, a point $a_-\in A$, and a control
$\widehat u$ on $[0,Nh]$ such that
\begin{equation}
\label{eq:alignment-endpoint}
    x_{\widehat u}(Nh;a_-)
    =
    x_u(T;a)
\end{equation}
and
\begin{equation}
\label{eq:alignment-control-norm}
    \|\widehat u\|_{\infty,[0,Nh]}
    =
    \|u\|_{\infty,[0,T]}.
\end{equation}
\end{lemma}

\begin{proof}
Choose $N:=\left\lceil\frac{T}{h}\right\rceil$, $\delta:=Nh-T\in[0,h)$. If $\delta=0$, take $a_-=a$ and $\widehat u=u$. Suppose $\delta>0$. Since $\phi^\delta(A)=A$, there exists $a_-\in A$ such that $\phi^\delta(a_-)=a$. Define
\[
    \widehat u(t)
    :=
    \begin{cases}
    0,&0\leq t<\delta,\\
    u(t-\delta),&\delta\leq t\leq Nh.
    \end{cases}
\]
The controlled trajectory follows the invariant attractor without control from $a_-$ to $a$ during the first $\delta$ units of time and then reproduces the original controlled trajectory. This proves \eqref{eq:alignment-endpoint}. Adding the initial zero segment does not change the essential supremum norm, proving \eqref{eq:alignment-control-norm}.
\end{proof}

\begin{remark}[No backward flow is required]
\label{rem:no-backward-flow}
Lemma~\ref{lem:sampling-time-alignment} does not require the flow to be backward complete on all of $\mathbb R^d$. It only uses the surjectivity of $\phi^\delta:A\to A$, which follows from the equality $\phi^\delta(A)=A$. Positive invariance $\phi^\delta(A)\subset A$ alone would not be sufficient.
\end{remark}

\begin{proposition}[Continuous-to-discrete reachability]
\label{prop:continuous-to-discrete}
For every $r>0$ and $h>0$,
\begin{equation}
\label{eq:continuous-to-discrete-inclusion}
    \mathcal P_r(A)
    \subset
    \mathcal P_{\alpha_h(r)}^{F_h}(A).
\end{equation}
\end{proposition}

\begin{proof}
Take any $y\in\mathcal P_r(A)$. Then $y=x_u(T;a)$ for some $a\in A$, $T\geq0$, and control satisfying
$\|u\|_\infty<r$. By Lemma~\ref{lem:sampling-time-alignment}, the same endpoint can be written as $y=x_{\widehat u}(Nh;a_-)$ for an $a_-\in A$ and a control $\widehat u$ with $\|\widehat u\|_\infty<r$.

Set $y_k:=x_{\widehat u}(kh;a_-)$, $k=0,\ldots,N$. On each interval $[kh,(k+1)h]$, apply Lemma~\ref{lem:one-step-control-error} to the time-translated control. Then $\|y_{k+1}-\phi^h(y_k)\| < \alpha_h(r)$. Thus $y_0,\ldots,y_N$ is an $\alpha_h(r)$-perturbed orbit of $F_h$ starting in $A$, and $y=y_N$ belongs to the right-hand side of \eqref{eq:continuous-to-discrete-inclusion}.
\end{proof}

We can now prove the main theorem.

\begin{proof}[Proof of Theorem~\ref{thm:sampling}]
Let
    $0<\varepsilon<
    \frac{h\mu(A)}{1+Lh}$.
Then
    $\beta_h(\varepsilon)
    =
    \left(\frac1h+L\right)\varepsilon
    <
    \mu(A)$.
By Definition~\ref{def:continuous-intensity} and monotonicity of reachable sets, $\overline{\mathcal P_{\beta_h(\varepsilon)}(A)} \subset K\Subset\mathcal D(A)$ for some compact $K$. Proposition~\ref{prop:discrete-to-continuous} therefore gives $\overline{\mathcal P_\varepsilon^{F_h}(A)}
    \subset
    \overline{\mathcal P_{\beta_h(\varepsilon)}(A)}
    \subset K\Subset\mathcal D(A)$.
Hence $\varepsilon$ is admissible in the definition of $\mu_h(A)$. Taking the supremum over such $\varepsilon$ yields $\mu_h(A)
    \geq
    \frac{h\mu(A)}{1+Lh}$.

For the upper bound, take any $0<\varepsilon<\mu_h(A)$ and set
\[
    r:=\alpha_h^{-1}(\varepsilon)
    =
    \begin{cases}
    \displaystyle
    \frac{L\varepsilon}{e^{Lh}-1},&L>0,\\[2ex]
    \displaystyle
    \frac{\varepsilon}{h},&L=0.
    \end{cases}
\]
Then $\alpha_h(r)=\varepsilon$. Since $\varepsilon<\mu_h(A)$,
    $\overline{\mathcal P_\varepsilon^{F_h}(A)}
    \subset K\Subset\mathcal D(A)$
for some compact $K$. Proposition~\ref{prop:continuous-to-discrete} implies
    $\overline{\mathcal P_r(A)}
    \subset
    \overline{\mathcal P_\varepsilon^{F_h}(A)}
    \subset K\Subset\mathcal D(A)$.
Therefore $r\leq\mu(A)$, and consequently
    $\varepsilon=\alpha_h(r)
    \leq
    \alpha_h\bigl(\mu(A)\bigr)$.
Letting $\varepsilon\uparrow\mu_h(A)$ proves the upper bound in \eqref{eq:sampling-main-bounds}. Dividing by $h$ and taking $h\downarrow0$ proves \eqref{eq:sampling-main-limit}.
\end{proof}

\subsection{Rates, scalar benchmarks and one-step numerical methods}
\label{subsec:sampling-consequences}

This part quantifies the convergence rates implied by the previous bounding bounds. We first establish an exact $O(h)$ convergence rate for exact time-$h$ sampling, and then extend this framework to practical one-step numerical methods by bounding the dual effects of local truncation and numerical attractor displacement. 
For $L>0$, the elementary expansions
\begin{align}
\label{eq:lower-factor-expansion}
    \frac1{1+Lh}
    &=
    1-Lh+L^2h^2+O(h^3),\\
\label{eq:upper-factor-expansion}
    \frac{e^{Lh}-1}{Lh}
    &=
    1+\frac{Lh}{2}+\frac{L^2h^2}{6}+O(h^3)
\end{align}
give the following consequence.

\begin{corollary}[First-order convergence]
\label{cor:sampling-first-order}
Under the hypotheses of Theorem~\ref{thm:sampling},
\begin{equation}
\label{eq:sampling-first-order}
    \frac{\mu_h(A)}{h}
    =
    \mu(A)+O(h).
\end{equation}
More precisely,
    $-\mu(A)Lh+O(h^2)
    \leq
    \frac{\mu_h(A)}{h}-\mu(A)
    \leq
    \frac{\mu(A)L}{2}h+O(h^2)$.
\end{corollary}

The two sides of \eqref{eq:sampling-main-bounds} arise from different
approximations and need not be simultaneously attained. The lower
bound reflects the cost of interpolating a discrete endpoint error,
while the upper bound is governed by Gr\"onwall amplification of a
continuous control. Additional smoothness or special geometry can
lead to higher-order convergence in individual systems.
Example~\ref{ex:sharp-scalar} below shows, however, that the general
$O(h)$ rate in \eqref{eq:sampling-first-order} cannot be replaced by
$o(h)$ without additional hypotheses.

Both that example and the scalar computations of
Section~\ref{sec:examples} concern one-dimensional systems whose basin
is bounded on the escape side only. For this class we first record a
one-sided refinement of Theorem~\ref{thm:sampling}. Its point is that
the constants involve only the restriction of $f$ to the escape
interval between the basin boundary and the attractor, even though the
perturbed reachable sets necessarily leave that interval on the
non-escape side. Theorem~\ref{thm:sampling} itself, including its
localized form in Remark~\ref{rem:localized-sampling}, would instead
require a Lipschitz constant valid on a region containing all
subcritical perturbed reachable sets and interpolation segments.

\begin{proposition}[One-sided sampling bounds for scalar systems]
\label{prop:scalar-one-sided}
Let $f:\mathbb R\to\mathbb R$ be locally Lipschitz with forward-complete
flow, let $b<a$, and assume:
\begin{enumerate}
\item $f(b)=f(a)=0$, $f>0$ on $(b,a)$, and $f\leq0$ on $[a,\infty)$;
\item $A=\{a\}$ is an attractor with basin
$\mathcal D(A)=(b,\infty)$;
\item for every $r\geq0$ there exists $x^+(r)>a$ such that
$f(x)\leq-r$ for all $x\geq x^+(r)$.
\end{enumerate}
Set $I:=[b,a]$, let $L\geq0$ be a Lipschitz constant for $f$ on $I$,
and write
    $\mu:=\max_{x\in I}f(x)$.
Then, for every $h>0$,
\begin{align}
\label{eq:scalar-one-sided-continuous}
    \mu(A)&=\mu,\\
\label{eq:scalar-one-sided-discrete}
    \mu_h(A)&=\max_{x\in I}\bigl(F_h(x)-x\bigr),\\
\label{eq:scalar-one-sided-bounds}
    \mu\,\frac{1-e^{-Lh}}{Lh}
    \;\leq\;
    \frac{\mu_h(A)}{h}
    \;&\leq\;
    \mu,
\end{align}
where the prefactor in \eqref{eq:scalar-one-sided-bounds} is
interpreted as $1$ when $L=0$.
\end{proposition}

\begin{proof}
Throughout, we use that controlled solutions are continuous in time and
that $F_h$ is nondecreasing, being the time-$h$ map of a scalar flow.

\emph{Step 1: barriers.}
Let $r\geq0$ and let $c\in(b,a)$ satisfy $f(c)>r$. No solution of
$\dot x=f(x)+u(t)$ with $\|u\|_\infty<r$ crosses $c$ from right to
left. Indeed, suppose $x(t_0)\geq c$ and $x(t_1)<c$ for some
$t_1>t_0$, and let
$\tau:=\sup\{t\in[t_0,t_1]:x(t)\geq c\}$, so that $x(\tau)=c$ and
$x<c$ on $(\tau,t_1]$. By continuity there is $\sigma\in(\tau,t_1]$
with $f(x(t))\geq\tfrac12(f(c)+r)$ for $t\in[\tau,\sigma]$, whence
\[
    x(t)-c
    =
    \int_\tau^t\bigl(f(x(s))+u(s)\bigr)\,ds
    \geq
    (t-\tau)\,\tfrac12\bigl(f(c)-r\bigr)
    >0
    \qquad\text{on }(\tau,\sigma],
\]
contradicting the definition of $\tau$. The mirrored argument, using
hypothesis (iii), shows that no such solution crosses $x^+(r)$ from
left to right.

\emph{Step 2: continuous intensity.}
Let $r<\mu$ and let $x_*\in(b,a)$ maximize $f$ on $I$, so
$f(x_*)=\mu>r$. By Step 1, every admissible controlled trajectory
starting at $a$ remains in $[x_*,x^+(r)]$, a compact subset of
$(b,\infty)=\mathcal D(A)$. Hence every $r<\mu$ is admissible in
Definition~\ref{def:continuous-intensity}, and $\mu(A)\geq\mu$.
Conversely, let $r>\mu$ and choose $\mu<r'<r$. Under the constant
control $u\equiv-r'$, one has $\dot x=f(x)-r'\leq\mu-r'<0$ on $I$, so
the solution starting at $a$ reaches $b\notin\mathcal D(A)$ in finite
time. Thus $r$ is not admissible, and $\mu(A)\leq\mu$.

\emph{Step 3: discrete intensity.}
Set $\varepsilon_h:=\max_{x\in I}(F_h(x)-x)$. Since $F_h(x)>x$ on
$(b,a)$ and $F_h-\operatorname{id}$ vanishes at $b$ and $a$, the
maximum is positive and is attained at some $x_h\in(b,a)$.

First, $\varepsilon_h$ is admissible. Consider a perturbed orbit from
$a$ with kicks $\|e_k\|<\varepsilon_h$. If $x\geq x_h$, monotonicity
gives
    $F_h(x)+e
    \geq
    F_h(x_h)-\|e\|
    >
    x_h+\varepsilon_h-\varepsilon_h
    =
    x_h$,
so the orbit remains to the right of $x_h>b$. For confinement on the
right, pick $r$ with $rh\geq\varepsilon_h$ and set
$M:=x^+(r)+\varepsilon_h$. For $x\geq x^+(r)$, the unperturbed
trajectory decreases at rate at least $r$ while it remains in
$[x^+(r),\infty)$, and it cannot re-enter that set once it leaves it
(Step 1 with $u\equiv0$); hence
$F_h(x)\leq\max\bigl(x-rh,\;x^+(r)\bigr)\leq\max(x-\varepsilon_h,\,
M-\varepsilon_h)$. By monotonicity, for every $x\leq M$,
    $F_h(x)+e
    \leq
    F_h(M)+\varepsilon_h
    \leq
    M$,
so $(-\infty,M]$ is forward invariant for the perturbed dynamics.
Every perturbed orbit from $a$ therefore remains in the compact set
$[x_h,M]\subset(b,\infty)$, and $\mu_h(A)\geq\varepsilon_h$.

Conversely, let $\varepsilon>\varepsilon_h$ and choose
$\varepsilon_h<\varepsilon'<\varepsilon$. Apply the constant kicks
$e_k\equiv-\varepsilon'$. For $x\in(b,a]$,
\[
    F_h(x)-\varepsilon'
    \leq
    x+\varepsilon_h-\varepsilon'
    =
    x-(\varepsilon'-\varepsilon_h),
\]
while for $x\geq a$ one has $F_h(x)\leq x$, because the flow is
nonincreasing on $[a,\infty)$ and cannot cross the equilibrium $a$.
Starting at $a$, the perturbed orbit therefore decreases by at least
$\varepsilon'-\varepsilon_h$ per step while it remains in $(b,a]$, and
after finitely many steps it produces a point $\leq b$, which lies
outside $\mathcal D(A)$. Hence $\varepsilon$ is not admissible, and
$\mu_h(A)\leq\varepsilon_h$. This proves
\eqref{eq:scalar-one-sided-discrete}.

\emph{Step 4: the bounds.}
For $x\in(b,a)$, the unperturbed trajectory $\phi^s(x)$ increases
toward $a$ and hence remains in $I$. Therefore
    $F_h(x)-x
    =
    \int_0^h f\bigl(\phi^s(x)\bigr)\,ds
    \leq
    h\mu$,
and taking the maximum over $I$ gives $\mu_h(A)\leq h\mu$, the upper
bound in \eqref{eq:scalar-one-sided-bounds}.

For the lower bound, set
\[
    g(s):=f\bigl(\phi^s(x_*)\bigr)\geq0,
    \qquad
    G(s):=\int_0^s g(\sigma)\,d\sigma
    =
    \phi^s(x_*)-x_*.
\]
Since $\phi^s(x_*)\in I$, the Lipschitz bound on $I$ gives
    $g(s)
    \geq
    f(x_*)-L\bigl(\phi^s(x_*)-x_*\bigr)
    =
    \mu-LG(s)$,
so $(e^{Ls}G)'=e^{Ls}(G'+LG)\geq\mu e^{Ls}$, and integration yields
    $G(h)\geq\mu\,\frac{1-e^{-Lh}}{L}$.
Since $\mu_h(A)\geq F_h(x_*)-x_*=G(h)$ by
\eqref{eq:scalar-one-sided-discrete}, the lower bound follows.
\end{proof}

\begin{remark}[Scope of the one-sided bounds]
\label{rem:one-sided-scope}
The constants use only $f|_I$ on the escape interval, and the upper bound has no Gr\"onwall amplification, so $\mu_h(A)/h\uparrow\mu(A)$. The mirrored and two-sided-basin versions hold by symmetry.
\end{remark}

When the vector field is twice differentiable on the escape interval,
the one-sided structure yields second-order convergence, explaining
the rates observed in Section~\ref{sec:examples}.

\begin{corollary}[Second-order convergence for smooth scalar folds]
\label{cor:scalar-second-order}
In the setting of Proposition~\ref{prop:scalar-one-sided}, suppose in
addition that $f\in C^2$ on $I$ and set $L_2:=\sup_I|f''|$. Then, for
every $h>0$,
    $0
    \;\leq\;
    \mu(A)-\frac{\mu_h(A)}{h}
    \;\leq\;
    \frac{L_2\,\mu(A)^2}{6}\,h^2$.
\end{corollary}

\begin{proof}
The first inequality is the upper bound in
\eqref{eq:scalar-one-sided-bounds}. For the second, let
$x_*\in(b,a)$ maximize $f$ on $I$. Since the maximum is interior,
$f'(x_*)=0$, and Taylor's theorem gives
    $f(y)\geq\mu-\frac{L_2}{2}(y-x_*)^2$,
     $y\in I$.
The trajectory $\phi^s(x_*)$ remains in $I$ and satisfies
    $0
    \leq
    \phi^s(x_*)-x_*
    =
    \int_0^s f\bigl(\phi^\sigma(x_*)\bigr)\,d\sigma
    \leq
    s\mu$.
Hence, by \eqref{eq:scalar-one-sided-discrete},
\[
    \mu_h(A)
    \geq
    F_h(x_*)-x_*
    =
    \int_0^h f\bigl(\phi^s(x_*)\bigr)\,ds
    \geq
    \int_0^h
    \left(
        \mu-\frac{L_2}{2}\mu^2s^2
    \right)ds
    =
    h\mu-\frac{L_2\mu^2}{6}\,h^3.
    \qedhere
\]
\end{proof}

\begin{example}[A sharp first-order scalar example]
\label{ex:sharp-scalar}
Let $f(x)=x$ for $x\le\tfrac12$ and $1-x$ for $x\ge\tfrac12$ (globally
Lipschitz, $L=1$), with attractor $A=\{1\}$, basin $(0,\infty)$. Escape
leftward must overcome the maximal drift on $[0,1]$, so $\mu(A)=\max_{[0,1]}
f=\tfrac12$. Proposition~\ref{prop:scalar-one-sided} applies ($b=0$, $a=1$,
$f\to-\infty$), and computing the time-$h$ map explicitly gives the exact
$\mu_h(A)=\max_{[0,1]}(F_h-\mathrm{id})=1-e^{-h/2}$ (maximizer
$x_h^\ast=\tfrac12 e^{-h/2}$), hence
    $\frac{\mu_h(A)}{h}=\tfrac12-\tfrac18 h+O(h^2)=\mu(A)-\tfrac18 h+O(h^2)$.
The nonzero first-order coefficient shows the $O(h)$ rate is sharp even in
dimension one; note $-\tfrac18\in[-\tfrac14,0]$, the interval predicted by
\eqref{eq:scalar-one-sided-bounds} with $L=1$.
\end{example}

\begin{remark}[Effect of smoothness]
\label{rem:smooth-superconvergence}
The kink in Example~\ref{ex:sharp-scalar} defeats Corollary~\ref{cor:scalar-second-order} (its maximizer is the nonsmooth point $x=\tfrac12$), forcing the first-order rate; a $C^2$ interior maximizer instead gives second order, with $L_2$ possibly large.
\end{remark}

While the preceding results assume the exact time-$h$ map $F_h=\phi^h$ is available, practical computations rely on one-step numerical approximations $\Psi_h:\mathbb R^d\to\mathbb R^d$. Classical results establish persistence and convergence of numerical attractors only under stability, attraction-rate, and robustness
hypotheses beyond local consistency; see \citep{kloeden1986stable,grune2001asymptotic,garay2005attractors,stuart1998dynamical,beyn1987numerical}. General background on one-step methods and their long-time behavior is in \citep{hairer1993solving,hairer2006geometric}. We therefore separate two logically distinct issues:
\begin{enumerate}
\item stability of block intensity under perturbation of a one-step map;
\item approximation or continuation of the attractor used as the seed of the reachable set.
\end{enumerate}
This distinction prevents a local truncation estimate from being mistaken for a global basin-convergence theorem.

We first discuss map stability. To isolate the numerical truncation error from the global escape dynamics, we localize the definition of intensity to a fixed trapping block $B$. 

For a compact set $B$ with nonempty interior and a compact seed set $S \subset \operatorname{int} B$, the $B$-block intensity of a continuous map $G$ is defined as
\[
    \mu_B(G;S) := \sup\left\{\varepsilon \ge 0 : \overline{\mathcal P^G_\varepsilon(S)} \subset \operatorname{int} B\right\}.
\]
Its continuous-time analogue is denoted by
\[
    \mu_B(A) := \sup\left\{r \ge 0 : \overline{\mathcal P_r(A)} \subset \operatorname{int} B\right\}.
\]
For the exact time-$h$ map $F_h = \phi^h$ and a numerical one-step map $\Psi_h$, we adopt the notations
\[
    \mu^\phi_{h,B}(A) := \mu_B(F_h;A) \quad \text{and} \quad \mu^\Psi_{h,B}(S) := \mu_B(\Psi_h;S).
\]

Clearly, $\mu_B(A) \le \mu(A)$, with equality holding whenever $B \Subset \mathcal D(A)$ contains every subcritical reachable set. Propositions~\ref{prop:discrete-to-continuous}--\ref{prop:continuous-to-discrete} apply verbatim on $B$, yielding
\begin{equation}
\label{eq:block-sampling-bounds}
    \frac{h\mu_B(A)}{1+Lh} \le \mu^\phi_{h,B}(A) \le \alpha_h\bigl(\mu_B(A)\bigr).
\end{equation}

The following elementary stability theorem is the main tool for numerical schemes.

\begin{theorem}[Lipschitz stability of block intensity]
\label{thm:block-intensity-map-stability}
Let $F,G:B\to\mathbb R^d$ be continuous, let $S\subset\operatorname{int}B$ be compact, and suppose $\sup_{x\in B}\|F(x)-G(x)\| \leq\delta$. Then
    $\left|
    \mu_B(F;S)-\mu_B(G;S)
    \right|
    \leq\delta$.
\end{theorem}

\begin{proof}
Let $0<\varepsilon<\mu_B(F;S)-\delta$. A $G$-perturbed step with error smaller than $\varepsilon$ satisfies
\[
    \|x_{k+1}-F(x_k)\|\leq
    \|x_{k+1}-G(x_k)\|
    +
    \|G(x_k)-F(x_k)\|<
    \varepsilon+\delta.
\]
Thus, until a possible first exit from $B$, every $\varepsilon$-perturbed orbit of $G$ is an $(\varepsilon+\delta)$-perturbed orbit of $F$. Since $\varepsilon+\delta<\mu_B(F;S)$, the latter orbit cannot leave $B$ and in fact remains compactly contained in $\operatorname{int}B$. Therefore $\varepsilon\leq\mu_B(G;S)$. Letting $\varepsilon\uparrow\mu_B(F;S)-\delta$ gives $\mu_B(G;S) \geq \mu_B(F;S)-\delta$. Interchanging $F$ and $G$ gives the reverse inequality.
\end{proof}

Suppose now that $\Psi_h$ has one-step order $q$ on $B$, in the sense that
\begin{equation}
\label{eq:numerical-one-step-defect}
    \delta_h
    :=
    \sup_{x\in B}
    \|\Psi_h(x)-\phi^h(x)\|
    \leq C_Bh^{q+1}.
\end{equation}

\begin{corollary}[Fixed-seed numerical intensity]
\label{cor:fixed-seed-numerical-intensity}
Under \eqref{eq:numerical-one-step-defect},
\begin{equation}
\label{eq:fixed-seed-intensity-error}
    \left|
    \mu_{h,B}^{\Psi}(A)
    -
    \mu_{h,B}^{\phi}(A)
    \right|
    \leq
    C_Bh^{q+1}.
\end{equation}
Consequently,
\[
    \frac{\mu_{h,B}^{\Psi}(A)}{h}
    =
    \mu_B(A)+O(h)+O(h^q).
\]
If $B$ realizes the full intensity, so that $\mu_B(A)=\mu(A)$, then
\[
    \frac{\mu_{h,B}^{\Psi}(A)}{h}
    =
    \mu(A)+O(h)+O(h^q).
\]
\end{corollary}

\begin{proof}
Apply Theorem~\ref{thm:block-intensity-map-stability} with $F=F_h$, $G=\Psi_h$, and $S=A$, and combine the result with \eqref{eq:block-sampling-bounds}.
\end{proof}

The estimate \eqref{eq:fixed-seed-intensity-error} is the precise version of the heuristic statement that a local one-step error of order $h^{q+1}$ perturbs discrete block intensity by the same order. It compares the maps using the same initial seed $A$. A separate term is required when the numerical reachable set is initiated from a different numerical attractor.

We discuss continued numerical attractors.

Assume that $\Psi_h$ admits a distinguished compact set $A_h\subset\operatorname{int}B$, typically a numerical attractor. Neither invariance nor attraction of $A_h$ is used in the comparison below; only its Hausdorff distance to $A$ enters. Let $\eta_h:=d_H(A_h,A)$, where $d_H$ denotes the Hausdorff distance. Suppose also that $F_h$ is $M_h$-Lipschitz on a neighborhood containing $A\cup A_h$. Under the global $L$-Lipschitz hypothesis on $f$, one may take $M_h=e^{Lh}$.

\begin{theorem}[Numerical intensity with attractor continuation]
\label{thm:numerical-intensity}
Assume that:
\begin{enumerate}
\item $B\Subset\mathcal D(A)$ is a fixed common trapping block;
\item $A\cup A_h\subset\operatorname{int}B$;
\item $\Psi_h$ and $F_h$ are defined on $B$ and satisfy $\sup_{x\in B}\|\Psi_h(x)-F_h(x)\|\leq\delta_h$;
\item $F_h$ is $M_h$-Lipschitz on a neighborhood of $A\cup A_h$;
\item $d_H(A_h,A)\leq\eta_h$.
\end{enumerate}
Then
\begin{equation}
\label{eq:numerical-attractor-intensity-error}
    \left|
    \mu_{h,B}^{\Psi}(A_h)
    -
    \mu_{h,B}^{\phi}(A)
    \right|
    \leq
    \delta_h+M_h\eta_h.
\end{equation}
In particular, if $\delta_h+M_h\eta_h = O(h^{q+1})$, then $\frac{\mu_{h,B}^{\Psi}(A_h)}{h} = \mu_B(A)+O(h)+O(h^q)$. If, additionally, $\mu_B(A)=\mu(A)$, then $\frac{\mu_{h,B}^{\Psi}(A_h)}{h} = \mu(A)+O(h)+O(h^q)$.
\end{theorem}

\begin{proof}
Let $a_h\in A_h$ and choose $a\in A$ with $\|a_h-a\|\leq\eta_h+\zeta$ for arbitrary $\zeta>0$. Consider a perturbed $\Psi_h$-orbit $x_0,x_1,\ldots$ with $x_0=a_h$ and perturbation amplitude $\varepsilon$. Define a sequence for $F_h$ by $y_0:=a$, $y_k:=x_k$, $k\geq1$. For its first step,
\begin{align*}
    \|y_1-F_h(y_0)\|
    &=
    \|x_1-F_h(a)\| \leq
    \|x_1-\Psi_h(a_h)\|
    +
    \|\Psi_h(a_h)-F_h(a_h)\|
    +
    \|F_h(a_h)-F_h(a)\|\\[2pt]
    &<
    \varepsilon+\delta_h
    +M_h(\eta_h+\zeta).
\end{align*}
For every subsequent step, $\|y_{k+1}-F_h(y_k)\| < \varepsilon+\delta_h$. Thus, apart from the initial point $a_h\in\operatorname{int}B$, the $\Psi_h$-reachable set from $A_h$ is contained in the $F_h$-reachable set from $A$ with perturbation enlarged by $\delta_h+M_h(\eta_h+\zeta)$.

The same argument in the opposite direction begins with $a\in A$, chooses $a_h\in A_h$ within Hausdorff distance, and interprets an $F_h$-perturbed orbit as a $\Psi_h$-perturbed orbit. Therefore, 
\[
\begin{aligned}
    \mu_{h,B}^{\Psi}(A_h)
    &\geq
    \mu_{h,B}^{\phi}(A)
    -\delta_h-M_h(\eta_h+\zeta),\\
    \mu_{h,B}^{\phi}(A)
    &\geq
    \mu_{h,B}^{\Psi}(A_h)
    -\delta_h-M_h(\eta_h+\zeta).
\end{aligned}
\]
Letting $\zeta\downarrow0$ proves \eqref{eq:numerical-attractor-intensity-error}. The normalized estimates follow from \eqref{eq:block-sampling-bounds}.
\end{proof}

\begin{corollary}[Equilibrium-preserving methods]
\label{cor:equilibrium-preserving}
Suppose that $A=\{a\}$ is an equilibrium and that the one-step method preserves equilibria exactly: $f(a)=0$ implies $\Psi_h(a)=a$. Then one may take $A_h=A$, so $\eta_h=0$. If
\eqref{eq:numerical-one-step-defect} holds, then
\[
    \left|
    \mu_{h,B}^{\Psi}(A)
    -
    \mu_{h,B}^{\phi}(A)
    \right|
    \leq
    C_Bh^{q+1},
\]
and
\[
    \frac{\mu_{h,B}^{\Psi}(A)}h
    =
    \mu_B(A)+O(h)+O(h^q).
\]
\end{corollary}

Most standard Runge--Kutta methods preserve exact equilibria, so Corollary~\ref{cor:equilibrium-preserving} applies directly to the scalar benchmark in Example~\ref{ex:sharp-scalar} and to many bistable equilibrium models.

\begin{remark}[What local consistency controls]
A defect $\sup_B\|\Psi_h-\phi^h\|=O(h^{q+1})$ does not by itself imply existence or convergence of a numerical attractor, convergence of the numerical basin, or equality of block and unrestricted intensity \citep{kloeden1986stable,grune2001asymptotic,garay2005attractors}; hence Theorem~\ref{thm:numerical-intensity} takes the continuation of $A_h$ as a hypothesis. If only $\eta_h=O(h^q)$ is known, the normalized estimate loses an order unless $A_h=A$ or $\eta_h=O(h^{q+1})$. Replacing block by unrestricted intensity further requires a common block realizing the escape threshold, which the block formulation makes explicit rather than hiding in a local consistency assumption.
\end{remark}

The results of this section establish the quantitative bridge between continuous controls, discrete kicks, and numerical one-step errors. For exact sampling, normalized intensity converges with a generally sharp $O(h)$ error. For a one-step approximation of order $q$, the additional fixed-block map error is $O(h^q)$ after normalization, provided the seed set is fixed or the continued numerical attractor is controlled at the corresponding effective order. In Section~\ref{sec:nonautonomous}, we show that the same reachability mechanism extends uniformly to attracting invariant graphs over a nonautonomous base.

\section{Nonautonomous intensity and averaging}
\label{sec:nonautonomous}

Many systems of interest are subject to external forcing that varies periodically, quasiperiodically, aperiodically, or chaotically in time. In this setting, solutions of a single nonautonomous differential equation do not generate an autonomous dynamical system on the state space alone. A standard remedy is to include the forcing phase as an additional variable and formulate the dynamics as a cocycle or skew-product flow; see \citep{kloeden2011nonautonomous,carvalho2012attractors}. The attracting object is then generally an invariant graph over the forcing base, rather than a stationary subset of the fiber.

This section extends intensity to attracting invariant graphs and proves that the continuous--discrete sampling theorem remains valid uniformly over the forcing base. We subsequently distinguish this global basin quantity from hyperbolic continuation and local response. Finally, we examine what averaging theory can and cannot imply about intensity.

\subsection{Cocycles, invariant graphs, and fiberwise intensity}
\label{subsec:skew-product-framework}

Let $(\Omega,\theta)$ be a compact metric flow ($\theta_t$ a homeomorphism for each $t\in\mathbb R$), and consider $\dot x=f(\theta_t\omega,x)$ with $f$ continuous and locally Lipschitz in $x$, uniformly on compact $\omega$-sets. The solution $\varphi(t,\omega,x)$ satisfies the cocycle identity $\varphi(t+s,\omega,x)=\varphi(t,\theta_s\omega,\varphi(s,\omega,x))$ and generates the skew-product semiflow $\Pi_t(\omega,x)=(\theta_t\omega,\varphi(t,\omega,x))$; the standard Bebutov hull turns a bounded uniformly continuous time-dependent field into such a family \citep{kloeden2011nonautonomous}.

\begin{definition}[Uniformly attracting invariant graph and fiber basin]
\label{def:attracting-invariant-graph}

A continuous $a:\Omega\to\mathbb R^d$ is an invariant graph if
$a(\theta_t\omega) = \varphi(t,\omega,a(\omega))$ for $t\geq0$, $\omega\in\Omega$.
Its graph $\mathcal A = \{(\omega,a(\omega)):\omega\in\Omega\}$
is compact and invariant.

The invariant graph is called uniformly attracting if there exists an open neighborhood $\mathcal U$ of $\mathcal A$ such that
$\lim_{t\to\infty}
\sup_{(\omega,x)\in\mathcal U}
\|
\varphi(t,\omega,x)-a(\theta_t\omega)
\|
=
0$.
The fiber basin is defined by
$\mathcal B(\mathcal A)
=
\left\{
(\omega,x):
\|
\varphi(t,\omega,x)-a(\theta_t\omega)
\|
\to0
\right\}$,
with fibers
$\mathcal B_\omega(\mathcal A)
=
\{x:(\omega,x)\in\mathcal B(\mathcal A)\}$.

Under the uniform attraction assumption, $\mathcal B(\mathcal A)$ is open and positively $\Pi_t$-invariant. Indeed, if $(\omega,x)\in\mathcal B(\mathcal A)$, then for some $T>0$, $\Pi_T(\omega,x)\in\mathcal U$; by continuity of $\Pi_T$,
$\Pi_T^{-1}(\mathcal U)$ is a neighborhood of $(\omega,x)$ contained in $\mathcal B(\mathcal A)$.
\end{definition}

To quantify the resilience of such nonautonomous attractors, we subject the system to bounded persistent disturbances. Specifically, we perturb only the fiber dynamics and leave the forcing base unchanged: $\dot x = f(\theta_t\omega,x)+u(t)$, $\Vert{}u\Vert{}_\infty<r$. The solution starting at $x$ above $\omega$ is denoted by $\varphi_u(t,\omega,x)$. The base coordinate of the controlled trajectory remains $\theta_t\omega$.

Analogous to the autonomous framework in Section~\ref{sec:intensity}, for $S\subset\Omega\times\mathbb R^d$, we define the exact-time fiber-controlled reachable set by
\[
    \mathcal R_r^{\mathrm{na}}(S,t)
    :=
    \bigl\{
        (\theta_t\omega,\varphi_u(t,\omega,x)):
        (\omega,x)\in S,
        u\in\mathcal U_r([0,t])
    \bigr\}.
\]
The corresponding reachable tube is
    $\mathcal P_r^{\mathrm{na}}(S)
    :=
    \bigcup_{t\geq0}
    \mathcal R_r^{\mathrm{na}}(S,t)$.
For the invariant graph,
\[
    \mathcal P_r^{\mathrm{na}}(\mathcal A)
    =
    \bigl\{
    (\theta_t\omega,
     \varphi_u(t,\omega,a(\omega))):
    \omega\in\Omega,\ t\geq0,
    u\in\mathcal U_r([0,t])
    \bigr\}.
\]

\begin{definition}[Uniform nonautonomous intensity]
\label{def:nonautonomous-intensity}

Let $\mathcal A$ be a uniformly attracting invariant graph.
Its uniform fiberwise intensity is defined by
\[
    \mu_{\mathrm{na}}(\mathcal A)
    :=
    \sup\biggl\{
    r\geq0:
    \overline{
        \mathcal P_r^{\mathrm{na}}(\mathcal A)
    }
    \subset K
    \Subset_{\Omega\times\mathbb R^d}
    \mathcal B(\mathcal A)
    \text{ for some compact }
    K\subset\Omega\times\mathbb R^d
    \biggr\}.
\]
\end{definition}

The strict compact containment in the product space $\Omega\times\mathbb R^d$ is essential. Since $\mathcal B(\mathcal A)$ is open and $K\Subset\mathcal B(\mathcal A)$, compact containment yields a positive distance from the basin boundary:
$\operatorname{dist}
    \left(
        K,
        (\Omega\times\mathbb R^d)
        \setminus
        \mathcal B(\mathcal A)
    \right)
    >0$.
Consequently, the fiber sections satisfy the uniform estimate
$\inf_{\omega\in\Omega}
    \operatorname{dist}
    \left(
        K_\omega,
        \mathbb R^d\setminus
        \mathcal B_\omega(\mathcal A)
    \right)
    >0$,
where
$K_\omega
:=
\{x:(\omega,x)\in K\}$ and
$\mathcal B_\omega(\mathcal A)
:=
\{x:(\omega,x)\in\mathcal B(\mathcal A)\}$.

\begin{remark}[Why the definition is uniform over the base]
\label{rem:uniform-over-base}
In a nonautonomous system the stability properties of an invariant graph may depend on the base point $\omega$. In particular, different phases may have different distances from the fiber basin boundary. Defining the intensity for a single fixed phase may therefore produce a phase-dependent quantity and may miss the most vulnerable base states. Definition~\ref{def:nonautonomous-intensity} instead requires one common disturbance amplitude and one compact safety margin in the product space $\Omega\times\mathbb R^d$, valid simultaneously for all $\omega\in\Omega$. Hence the resulting intensity is intrinsic to the invariant graph and independent of the choice of initial phase.
\end{remark}

If $\Omega$ consists of one point, then Definition~\ref{def:nonautonomous-intensity} reduces to Definition~\ref{def:continuous-intensity}. More generally, if $f(\omega,x)=f_0(x)$ is independent of $\omega$ and $a(\omega)\equiv A$ is a constant equilibrium graph, then $\mu_{\mathrm{na}}(\mathcal A)=\mu(A)$.

\subsection{Sampling an attracting invariant graph}
\label{subsec:nonautonomous-sampling}

For $h>0$, define the sampled skew-product map $F_h(\omega,x) := \bigl(\theta_h\omega, \varphi(h,\omega,x) \bigr)$. A fiberwise $\varepsilon$-perturbed orbit satisfies
\begin{equation*}
\begin{aligned}
    \omega_{k+1}
    &=
    \theta_h\omega_k,\\
    x_{k+1}
    &=
    \varphi(h,\omega_k,x_k)+e_k,
    \qquad
    \|e_k\|<\varepsilon.
\end{aligned}
\end{equation*}
Only the fiber is perturbed. Denote the resulting reachable set from $\mathcal A$ by $\mathcal P_{\varepsilon,h}^{\mathrm{na},d}(\mathcal A)$.

\begin{definition}[Sampled nonautonomous intensity]
\label{def:sampled-na-intensity}

The discrete fiberwise intensity of $\mathcal A$ for $F_h$ is
\[
    \mu_{\mathrm{na},h}(\mathcal A)
    :=
    \sup\bigl\{
    \varepsilon\geq0:
    \overline{
        \mathcal P_{\varepsilon,h}^{\mathrm{na},d}
        (\mathcal A)
    }
    \subset K
    \Subset_{\Omega\times\mathbb R^d}
    \mathcal B(\mathcal A)
    \text{ for some compact }
    K\subset\Omega\times\mathbb R^d
    \bigr\}.
\]
\end{definition}

To establish the relationship between the sampled discrete intensity and its continuous counterpart, we impose the following standard assumptions on the system.
\begin{enumerate}
\renewcommand{\theenumi}{(NA\arabic{enumi})}
\renewcommand{\labelenumi}{\textnormal{\theenumi}}
\item
\label{hyp:na-compact-base}
$(\Omega,\theta)$ is a compact metric flow defined for all $t\in\mathbb R$.
\item
\label{hyp:na-uniform-lipschitz}
There exists $L\geq0$ such that $\|f(\omega,x)-f(\omega,y)\| \leq L\|x-y\|$ for every $\omega\in\Omega$ and $x,y\in\mathbb R^d$.
\item
\label{hyp:na-forward-complete}
The controlled fiber equation is forward complete for bounded controls.
\item
\label{hyp:na-attracting-graph}
$\mathcal A$ is a continuous uniformly attracting invariant graph and
$0<\mu_{\mathrm{na}}(\mathcal A)<\infty$.
\end{enumerate}
As in the autonomous theorem, global Lipschitz continuity may be replaced by a uniform Lipschitz bound on a common compact trapping tube containing the interpolated trajectories.

To bound the discrete reachable set using the continuous dynamics, we first establish a forward inclusion by interpolating every fiber kick. Specifically, on a step beginning at $(\omega_k,x_k)$, set $\widetilde x_k(s)
:= \varphi(s,\omega_k,x_k) + \frac{s}{h}e_k$, $0\leq s\leq h$.
Then $\dot{\widetilde x}_k(s) = f(\theta_s\omega_k,\widetilde x_k(s)) +u_k(s)$, where $u_k(s)
    =
    f\bigl(
        \theta_s\omega_k,
        \varphi(s,\omega_k,x_k)
    \bigr)
    -
    f\bigl(
        \theta_s\omega_k,
        \widetilde x_k(s)
    \bigr)
    +
    \frac1h e_k$.
Uniform fiber Lipschitz continuity gives $\|u_k(s)\| \leq \left(\frac1h+L\right)\|e_k\|$. Concatenating these controls yields
\begin{equation}
\label{eq:na-discrete-continuous-inclusion}
    \mathcal P_{\varepsilon,h}^{\mathrm{na},d}
    (\mathcal A)
    \subset
    \mathcal P_{\varepsilon/h+L\varepsilon}^{\mathrm{na}}
    (\mathcal A).
\end{equation}

Conversely, Gr\"onwall's inequality applied with the same evolving base phase gives
\begin{equation}
\label{eq:na-gronwall-one-step}
    \left\|
    \varphi_u(h,\omega,x)
    -
    \varphi(h,\omega,x)
    \right\|
    <
    \alpha_h(r)
\end{equation}
whenever $\|u\|_\infty<r$, where
\[
    \alpha_h(r)
    =
    \begin{cases}
    \displaystyle
    r\frac{e^{Lh}-1}{L},
        &L>0,\\[2ex]
    rh,&L=0.
    \end{cases}
\]
Thus sampled controlled trajectories are $\alpha_h(r)$-perturbed fiber orbits. The only remaining issue is that an arbitrary controlled endpoint need not occur at an integer multiple of $h$.

\begin{lemma}[Nonautonomous time alignment]
\label{lem:nonautonomous-time-alignment}
Let $y=\varphi_u(T,\omega,a(\omega))$ for some $T\geq0$ and $\|u\|_\infty<r$. Then there exist $N\in\mathbb N_0$, a phase $\omega_-\in\Omega$, and a control $\widehat u$ on $[0,Nh]$ such that $\theta_{Nh}\omega_-
    =
    \theta_T\omega$, 
    $\varphi_{\widehat u}
    (Nh,\omega_-,a(\omega_-))
    =
    y$, and 
    $\|\widehat u\|_\infty
    =
    \|u\|_\infty$.
\end{lemma}

\begin{proof}
Write $T=nh+s$, $n\in\mathbb N_0$, $0\leq s<h$. If $s=0$, take $N=n$, $\omega_-=\omega$, and $\widehat u=u$. Suppose $0<s<h$ and define
    $\omega_-
    :=
    \theta_{s-h}\omega
    =
    \theta_{-(h-s)}\omega$.
Start at the graph point $a(\omega_-)$ and use zero control for $h-s$ units of time. Invariance of the graph gives
    $\varphi
    (h-s,\omega_-,a(\omega_-))
    =
    a(\theta_{h-s}\omega_-)
    =
    a(\omega)$.
After this initial unforced segment, apply the original control $u$. The total duration is $(h-s)+T=(n+1)h$. Set $N=n+1$. The final base phase is
    $\theta_{Nh}\omega_-
    =
    \theta_{(n+1)h}\theta_{s-h}\omega
    =
    \theta_{nh+s}\omega
     =
    \theta_T\omega$.
The fiber endpoint is the original endpoint $y$. Prepending a zero control does not change the essential supremum norm.
\end{proof}

\begin{remark}[Why a two-sided base flow is needed]
\label{rem:base-invertibility}
The shifted phase $\theta_{s-h}\omega$ requires negative-time base motion, so the alignment applies to an invertible base; a noninvertible semiflow needs a natural extension or grid-time-restricted definition.
\end{remark}

By resolving the endpoint alignment issue, Lemma~\ref{lem:nonautonomous-time-alignment} combined with the local bound in \eqref{eq:na-gronwall-one-step} yields the converse inclusion:
\begin{equation}
\label{eq:na-continuous-discrete-inclusion}
    \mathcal P_r^{\mathrm{na}}(\mathcal A)
    \subset
    \mathcal P_{\alpha_h(r),h}^{\mathrm{na},d}
    (\mathcal A).
\end{equation}

\begin{theorem}[Sampling of nonautonomous intensity]
\label{thm:nonautonomous-sampling}
Suppose that \ref{hyp:na-compact-base}--\ref{hyp:na-attracting-graph} hold. By combining the mutual inclusions established above, we obtain explicit bounds connecting the discrete and continuous intensities:
\[
    \frac{h\mu_{\mathrm{na}}(\mathcal A)}{1+Lh}
    \leq
    \mu_{\mathrm{na},h}(\mathcal A)
    \leq
    \alpha_h
    \bigl(
        \mu_{\mathrm{na}}(\mathcal A)
    \bigr).
\]
For $L>0$, equivalently,
\[
    \frac{\mu_{\mathrm{na}}(\mathcal A)}{1+Lh}
    \leq
    \frac{\mu_{\mathrm{na},h}(\mathcal A)}{h}
    \leq
    \mu_{\mathrm{na}}(\mathcal A)
    \frac{e^{Lh}-1}{Lh}.
\]
Consequently,
    $\lim_{h\downarrow0}
    \frac{
        \mu_{\mathrm{na},h}(\mathcal A)
    }{h}
    =
    \mu_{\mathrm{na}}(\mathcal A)$,
with
    $\frac{
        \mu_{\mathrm{na},h}(\mathcal A)
    }{h}
    =
    \mu_{\mathrm{na}}(\mathcal A)+O(h)$.
\end{theorem}

\begin{proof}
Let
    $0<\varepsilon
    <
    \frac{
        h\mu_{\mathrm{na}}(\mathcal A)
    }{1+Lh}$.
Then
    $\frac{\varepsilon}{h}+L\varepsilon
    <
    \mu_{\mathrm{na}}(\mathcal A)$.
By \eqref{eq:na-discrete-continuous-inclusion}, the discrete reachable set is contained in a continuous reachable set that is compactly contained in $\mathcal B(\mathcal A)$. Hence $\varepsilon \leq \mu_{\mathrm{na},h}(\mathcal A)$. Taking the supremum proves the lower bound.

For the upper bound, let $0<\varepsilon<\mu_{\mathrm{na},h}(\mathcal A)$ and choose $r=\alpha_h^{-1}(\varepsilon)$. By \eqref{eq:na-continuous-discrete-inclusion},
$\mathcal P_r^{\mathrm{na}}(\mathcal A) \subset \mathcal P_{\varepsilon,h}^{\mathrm{na},d}(\mathcal A)$.
Taking closures yields $\overline{\mathcal P_r^{\mathrm{na}}(\mathcal A)} \subset \overline{\mathcal P_{\varepsilon,h}^{\mathrm{na},d}(\mathcal A)}$. Since $\varepsilon<\mu_{\mathrm{na},h}(\mathcal A)$,
the definition of the sampled intensity implies that there exists a compact set $K\Subset\mathcal B(\mathcal A)$ such that $\overline{\mathcal P_{\varepsilon,h}^{\mathrm{na},d}(\mathcal A)}\subset K$. Therefore, $\overline{\mathcal P_r^{\mathrm{na}}(\mathcal A)}\Subset \mathcal B(\mathcal A)$, and hence, by the definition of the continuous intensity, $r\leq\mu_{\mathrm{na}}(\mathcal A)$. Because $\varepsilon=\alpha_h(r)$, we obtain $\varepsilon \leq \alpha_h \bigl(\mu_{\mathrm{na}}(\mathcal A)\bigr)$. Letting $\varepsilon\uparrow\mu_{\mathrm{na},h}(\mathcal A)$ and using continuity of $\alpha_h$ proves the upper bound. The limit and rate follow from the expansions in \eqref{eq:lower-factor-expansion} and \eqref{eq:upper-factor-expansion}.
\end{proof}

\subsection{Hyperbolicity, averaging, and response diagnostics}
\label{subsec:averaging-limitations}

Before connecting reachable sets to deterministic boundary dynamics, it is crucial to distinguish the global intensity $\mu_{\mathrm{na}}(\mathcal A)$ from classical local robustness notions. Suppose now that $f$ is continuously differentiable in $x$. An invariant graph $a$ is uniformly exponentially attracting if there exist constants $C\geq1$ and $\lambda>0$ and a neighborhood $\mathcal U$ of $\mathcal A$ such that $\|\varphi(t,\omega,x)-a(\theta_t\omega)\| \leq Ce^{-\lambda t} \|x-a(\omega)\|$ for all $(\omega,x)\in\mathcal U$ and $t\geq0$.  Such a graph is a nonautonomous analogue of a hyperbolic attracting equilibrium. Standard persistence results imply that sufficiently small perturbations of the vector field possess a nearby attracting invariant graph. In a representative formulation, for every $\epsilon_0>0$ there exists $\delta_{\mathrm{hyp}}(\epsilon_0)>0$ such that $\|\widetilde f-f\|_{C^1(\mathcal N)} < \delta_{\mathrm{hyp}}(\epsilon_0)$ on a suitable neighborhood $\mathcal N$ of $\mathcal A$ implies the existence of a hyperbolic invariant graph $\widetilde a$ satisfying $\sup_{\omega\in\Omega} \|\widetilde a(\omega)-a(\omega)\| < \epsilon_0$. This is a local continuation statement. It should not be identified with intensity.

\begin{proposition}[Local persistence and global intensity]
\label{prop:hyperbolicity-intensity-distinction}
The following quantities describe different aspects of robustness.
\begin{enumerate}
\item
The hyperbolic persistence radius $\delta_{\mathrm{hyp}}(\epsilon_0)$ controls existence of a nearby invariant graph in a prescribed neighborhood. 
\item
The graph displacement $\sup_{\omega\in\Omega}\|\widetilde a(\omega)-a(\omega)\|$ is a local response quantity.
\item
The intensity $\mu_{\mathrm{na}}(\mathcal A)$ controls the global amplitude of fiber forcing required to destroy compact containment inside the basin.
\end{enumerate}
In general, neither the exponential constants $(M,\lambda)$ nor $\delta_{\mathrm{hyp}}(\epsilon_0)$ determine
$\mu_{\mathrm{na}}(\mathcal A)$.
\end{proposition}

\begin{proof}
The first two quantities depend on the vector field and its derivative near the graph. By contrast, Definition~\ref{def:nonautonomous-intensity} depends on every region of the basin that can be reached by bounded controls.

This distinction is already visible in autonomous scalar equations, which are special cases of the skew-product setting. For $\eta>0$, consider on $[0,\eta]$, $f_\eta(x) = -x+\frac{x^2}{\eta}$. Extend $f_\eta$ outside this interval to a forward-complete globally Lipschitz vector field so that $0$ remains attracting, $\eta$ is the right basin boundary, and points immediately to the right of $\eta$ move away from $0$. At the attractor, $f_\eta'(0)=-1$ for every $\eta$, so the local exponential rate is unchanged. However, the positive control required to cross the right basin boundary must overcome
    $-f_\eta(x)
    =
    x-\frac{x^2}{\eta}$,
whose maximum is
    $\max_{0\leq x\leq\eta}
    \left(
        x-\frac{x^2}{\eta}
    \right)
    =
    \frac{\eta}{4}$.
Thus the intensity can tend to zero while the linearized attraction rate remains fixed. The construction changes the basin geometry arbitrarily close to the attractor while preserving the linearized rate. Hence even local hyperbolicity data alone cannot determine the global intensity. Local hyperbolicity therefore supplies no universal positive lower bound for global intensity.
\end{proof}

A common robust trapping tube can relate the two theories in one direction. If a compact neighborhood of the graph is simultaneously a domain of hyperbolic continuation and an $r$-robust trapping tube for the controlled inclusion, then $\mu_{\mathrm{na}}(\mathcal A)\geq r$. This conclusion comes from the trapping geometry, not from hyperbolicity alone.

While hyperbolic continuation addresses the existence of perturbed attractors, classical averaging theory addresses trajectory approximation. However, it presents similar limitations regarding global basin geometry. For slowly responding nonautonomous systems driven by a base flow $\theta_t$ on a compact metric space $\Omega$,
\begin{equation}
\label{eq:slow-nonautonomous-system}
    \dot x_\epsilon = \epsilon f(\theta_t\omega, x_\epsilon), \qquad 0 < \epsilon \ll 1,
\end{equation}
one compares the dynamics with an averaged autonomous model
\begin{equation}
\label{eq:averaged-system}
    \dot{\overline x}_\epsilon = \epsilon\overline f(\overline x_\epsilon).
\end{equation}
For a compact subset $B \subset \mathbb{R}^d$, the time-averaged vector field
\[
    \overline f(x) := \lim_{T\to\infty} \frac{1}{T} \int_0^T f(\theta_t\omega, x)\,dt
\]
is assumed to exist uniformly for $(\omega,x) \in \Omega \times B$, which constitutes the uniform Krylov--Bogoliubov--Mitropolsky (UKBM) condition \citep{longo2025nonautonomous}. The rate of averaging convergence is captured by the UKBM gauge
\[
    \widetilde\delta_B(\epsilon) := \sup_{\substack{\omega\in\Omega,\ x\in B\\ R\ge 0,\ 0\le \epsilon R \le 1}} \epsilon \left\| \int_0^R \left[ f(\theta_t\omega, x) - \overline f(x) \right] dt \right\|.
\]
Under the UKBM condition, one has $\lim_{\epsilon\downarrow0}\widetilde\delta_B(\epsilon)=0$. In a hull formulation, taking the supremum over $\omega \in \Omega$ is equivalent to maximizing over the initial time of the averaging interval. For ergodic representation and trajectory approximation, when the base flow $\theta_t$ is uniquely ergodic with unique invariant measure $\nu \in \mathcal{M}_\theta(\Omega)$, the uniform ergodic theorem \citep{cornfeld2012ergodic,walters2000introduction} guarantees that
\[
    \lim_{T\to\infty} \frac{1}{T} \int_0^T q(\theta_t\omega)\,dt = \int_\Omega q\,d\nu
\]
holds uniformly in $\omega \in \Omega$ for every continuous observable $q: \Omega \to \mathbb{R}$. Applying this componentwise to $q_x(\omega) := f(\omega, x)$ yields the spatial integral formula $\overline f(x) = \int_\Omega f(\omega, x)\,d\nu(\omega)$, where uniformity across $x \in B$ follows from a finite-net argument whenever $f(\omega, \cdot)$ is uniformly Lipschitz on $B$. While unique ergodicity implies UKBM, it is not strictly necessary: uniform time averages can exist even if $\Omega$ admits multiple invariant measures, provided they agree on $f(\cdot, x)$. Conversely, if distinct invariant measures yield different spatial averages, no single phase-independent averaged model can represent all forcing realizations.

Classical averaging theory bounds the error between \eqref{eq:slow-nonautonomous-system} and \eqref{eq:averaged-system} on time intervals of length $O(1/\epsilon)$ in terms of $\widetilde\delta_B(\epsilon)$, with typical error estimates of order $O(\sqrt{\widetilde\delta_B(\epsilon)})$, or $O(\widetilde\delta_B(\epsilon))$ under stronger cancellation or smoothness hypotheses \citep{bogoliubov1961asymptotic,verhulst2005methods,sanders2007averaging}. Near an exponentially attracting graph, contraction mechanisms allow one to restart these $O(1/\epsilon)$ tracking estimates repeatedly, extending trajectory approximation to the positive half-line $[0, \infty)$.

For conditional block comparison and global obstructions, despite long-time trajectory control, intensity is inherently a global basin quantity and cannot be inferred from trajectory approximation alone. To evaluate intensity transfer under averaging, consider the sampled discrete-time skew-product maps $F_h$ and $\overline F_h$ for the original flow $\varphi(t,\omega,x)$ and averaged flow $\overline\phi^t(x)$:
\begin{align*}
    F_h(\omega, x) &:= \bigl( \theta_h\omega, \varphi(h,\omega,x) \bigr), \\
    \overline F_h(\omega, x) &:= \bigl( \theta_h\omega, \overline\phi^h(x) \bigr).
\end{align*}
Let $B \subset \mathbb{R}^d$ be a common compact trapping block, and suppose the nonautonomous and averaged systems possess attracting graphs $\mathcal A = \{(\omega, a(\omega)) : \omega \in \Omega\}$ and $\overline{\mathcal A} = \{(\omega, \overline a) : \omega \in \Omega\}$, respectively. We define the map defect $\Delta_h$ and graph defect $\eta$ as
\begin{align}
\label{eq:averaging-map-defect}
    \Delta_h &:= \sup_{\omega\in\Omega, x\in B} \left\| \varphi(h,\omega,x) - \overline\phi^h(x) \right\|, \\
\label{eq:averaging-graph-defect}
    \eta &:= \sup_{\omega\in\Omega} \|a(\omega) - \overline a\|.
\end{align}

\begin{proposition}[Conditional fixed-block comparison]
\label{prop:averaged-block-comparison}
Assume that:
\begin{enumerate}
    \item $\Omega \times B$ is a common compact trapping block;
    \item both graphs $\mathcal A$ and $\overline{\mathcal A}$ lie in $\Omega \times \operatorname{int} B$;
    \item the time-$h$ fiber map $x \mapsto \varphi(h, \omega, x)$ is uniformly $M_h$-Lipschitz on $B$;
    \item the defects $\Delta_h$ and $\eta$ in \eqref{eq:averaging-map-defect}--\eqref{eq:averaging-graph-defect} are finite.
    \item there exists $\rho>0$ such that the reachable sets defining the block intensities remain at distance at least $\rho$ from the basin boundary inside $B$.
\end{enumerate}
Then the corresponding discrete block intensities satisfy
\begin{equation}
\label{eq:averaged-block-intensity-comparison}
    \left| \mu_{h,B}^{\mathrm{na}}(\mathcal A) - \mu_{h,B}^{\mathrm{av}}(\overline{\mathcal A}) \right| \leq \Delta_h + M_h\eta.
\end{equation}
If both intensities are defined relative to the same seed graph, the term $M_h\eta$ is absent.
\end{proposition}

\begin{proof}
Both skew-product maps have the same base component. A fiberwise $\varepsilon$-perturbed step of one map is a fiberwise $(\varepsilon+\Delta_h)$-perturbed step of the other. If the initial graphs differ, choose corresponding graph points with fiber distance at most $\eta$. Their first images differ by at most $M_h\eta+\Delta_h$. After the first step, the two perturbed sequences can be identified. The block-intensity stability argument of Theorem~\ref{thm:block-intensity-map-stability} then gives \eqref{eq:averaged-block-intensity-comparison}.
\end{proof}

Proposition~\ref{prop:averaged-block-comparison} underscores that transferring normalized intensity requires controlling the ratio $(\Delta_h + M_h\eta)/h$ as well as identifying a shared trapping block that realizes the relevant escape threshold. Without global trapping and controlled reachability, trajectory agreement near an attractor fails to constrain basin intensity, as demonstrated by the following construction.

\begin{example}[Local agreement does not determine intensity]
\label{ex:local-agreement-global-intensity}
For every $M > 1$, one can construct two smooth forward-complete scalar vector fields $f_1$ and $f_M$ such that $f_1(x) = f_M(x) = -x$ for $|x| \leq 1$, and $0$ is a hyperbolic attracting equilibrium for both systems. Outside this common neighborhood, choose smooth extensions with right basin boundaries $b_1, b_M > 1$ such that $\max_{0 \leq x \leq b_1} [-f_1(x)] = 1$,   $\max_{0 \leq x \leq b_M} [-f_M(x)] = M$. The scalar escape characterization then gives $\mu_{f_1}(\{0\}) = 1$,  $\mu_{f_M}(\{0\}) = M$. Thus the systems have identical trajectories and identical linearization in a full neighborhood of the attractor, but different basin intensities.
\end{example}

Example~\ref{ex:local-agreement-global-intensity} proves the following warning:
\begin{quote}
Uniform long-time approximation near an attracting graph does not, by itself, preserve the basin boundary or intensity.
\end{quote}
Averaging may reproduce local response while losing remote basin geometry, transient barriers, or alternative attractors. A global intensity comparison therefore requires global trapping or basin hypotheses in addition to trajectory approximation.

To systematically compare the original nonautonomous system with such approximations---whether derived from averaging, perturbations, or numerical models---we can employ explicit response diagnostics. As we have established, these diagnostics measure local displacement rather than global basin escape and should not be conflated with intensity. Let $\varphi$ and $\widetilde\varphi$ be two cocycles over the same base flow. Starting both systems from $(\omega,x_0)$, define
\[
    R_1(t;\omega,x_0)
    :=
    \widetilde\varphi(t,\omega,x_0)
    -
    \varphi(t,\omega,x_0)
\]
and
\[
    R_2(T;\omega,x_0)
    :=
    \frac1T
    \int_0^T
    \|R_1(t;\omega,x_0)\|\,dt,
    \qquad T>0.
\]
For scalar systems, $R_1$ may be retained as a signed response; $R_2$ uses its absolute value.

Suppose the two cocycles possess attracting invariant graphs $a$ and $\widetilde a$ and that, on a common attracting neighborhood, 
\begin{align}
\label{eq:first-graph-exponential-tracking}
    \left\|
    \varphi(t,\omega,x_0)-a(\theta_t\omega)
    \right\|
    &\leq
    C e^{-\lambda t},\\
\label{eq:second-graph-exponential-tracking}
    \left\|
    \widetilde\varphi(t,\omega,x_0)
    -
    \widetilde a(\theta_t\omega)
    \right\|
    &\leq
    C e^{-\lambda t}.
\end{align}

\begin{proposition}[Asymptotic response of attracting graphs]
\label{prop:response-asymptotics}
Under
\eqref{eq:first-graph-exponential-tracking}--%
\eqref{eq:second-graph-exponential-tracking},
\begin{equation}
\label{eq:R1-graph-asymptotic}
\begin{split}
    \left\|
    R_1(t;\omega,x_0)
    -
    \left[
    \widetilde a(\theta_t\omega)
    -
    a(\theta_t\omega)
    \right]
    \right\|
    \leq
    2Ce^{-\lambda t}.
\end{split}
\end{equation}
If
    $q(\omega)
    :=
    \|\widetilde a(\omega)-a(\omega)\|$,
then
\[
    \left|
    R_2(T;\omega,x_0)
    -
    \frac1T
    \int_0^T q(\theta_t\omega)\,dt
    \right|
    \leq
    \frac{2C}{\lambda T}.
\]
\end{proposition}

\begin{proof}
Equation \eqref{eq:R1-graph-asymptotic} follows by adding and subtracting the two invariant-graph values and applying the triangle inequality. The reverse triangle inequality gives
    $\left|
    \|R_1(t;\omega,x_0)\|
    -
    q(\theta_t\omega)
    \right|
    \leq
    2Ce^{-\lambda t}$.
Integrating and dividing by $T$ yields
    $\frac1T
    \int_0^T2Ce^{-\lambda t}\,dt
    \leq
    \frac{2C}{\lambda T}$.
\end{proof}

If the base is uniquely ergodic with invariant measure $\nu$, then $q$ is continuous and the uniform ergodic theorem gives $\lim_{T\to\infty} R_2(T;\omega,x_0) = \int_\Omega \|\widetilde a(\omega)-a(\omega)\|\,d\nu(\omega)$, uniformly in the initial phase and, under uniform attraction, uniformly for $x_0$ in a common compact attracting neighborhood.

When the base has several invariant measures, a single limiting response need not exist. Define
\begin{equation*}
    \underline R
    :=
    \min_{\nu\in\mathcal M_\theta(\Omega)}
    \int_\Omega q\,d\nu,\qquad
    \overline R
    :=
    \max_{\nu\in\mathcal M_\theta(\Omega)}
    \int_\Omega q\,d\nu.
\end{equation*}
The extrema exist because $\mathcal M_\theta(\Omega)$ is weakly compact and $q$ is continuous.

\begin{proposition}[Response interval for a nonuniquely ergodic base]
\label{prop:response-invariant-measure-interval}
For every $\omega\in\Omega$ and every $x_0$ satisfying the tracking bounds
\eqref{eq:first-graph-exponential-tracking}--\eqref{eq:second-graph-exponential-tracking},
    $\underline R
    \leq
    \liminf_{T\to\infty}
    R_2(T;\omega,x_0)
    \leq
    \limsup_{T\to\infty}
    R_2(T;\omega,x_0)
    \leq
    \overline R$.
If $\omega$ is generic for an invariant measure $\nu$, then
\begin{equation}
\label{eq:R2-generic-measure-limit}
    \lim_{T\to\infty}
    R_2(T;\omega,x_0)
    =
    \int_\Omega q\,d\nu.
\end{equation}
\end{proposition}

\begin{proof}
Consider the empirical measures $\nu_T^\omega := \frac1T \int_0^T \delta_{\theta_t\omega}\,dt$.
Every weak limit point of $\{\nu_T^\omega\}_{T>0}$ is $\theta$-invariant. Therefore every subsequential limit of $\frac1T\int_0^Tq(\theta_t\omega)\,dt$ lies between the minimum and maximum of $\int q\,d\nu$ over invariant measures. Proposition~\ref{prop:response-asymptotics} shows that $R_2$ has the same subsequential limits. If $\omega$ is generic for $\nu$, the empirical measures converge to $\nu$, proving \eqref{eq:R2-generic-measure-limit}.
\end{proof}

Accordingly, numerical studies with a uniquely ergodic forcing base may report one phase-independent asymptotic response. With several invariant measures, the mathematically meaningful summary is the interval $[\underline R,\overline R]$, possibly supplemented by measure-specific or phase-specific response values. By contrast, the worst-case intensity $\mu_{\mathrm{na}}(\mathcal A)$ remains well defined without selecting an invariant probability measure, because it is uniform over the entire base.

Ultimately, together with the sampling result established above, this section extends the robustness framework from autonomous attractors to nonautonomous invariant graphs while preserving the same quantitative bounds. They formally separate three notions that are often conflated: hyperbolic continuation is local, response functions measure displacement, and intensity measures global controlled basin escape. Building on this, Section~\ref{sec:noise-boundary} next interprets sampled intensity as a bounded-noise escape threshold and connects the reachable-set geometry with deterministic boundary dynamics.

\section{Bounded-noise escape and boundary dynamics}
\label{sec:noise-boundary}

We now interpret intensity as a sharp worst-case escape threshold for bounded random perturbations. We then relate the boundary map associated with the sampled set-valued dynamics to the boundary differential equation
arising from the Pontryagin maximum principle. Throughout this section, $C\subset\mathbb R^d$ denotes a compact convex set satisfying $0\in\operatorname{int}C$. Its Minkowski functional is $\gamma_C(z):=\inf\{\lambda>0:z\in\lambda C\}$. When $C=\overline B_1(0)$, one has $\gamma_C(z)=\|z\|$, and the subscript $C$ will be omitted.

\subsection{Bounded random perturbations}
\label{subsec:bounded-random-perturbations}

Let $F_h=\phi^h$ be the time-$h$ map of the autonomous system $\dot x=f(x)$, and consider the randomly perturbed recurrence 
\begin{equation}
    X_{k+1}=F_h(X_k)+\xi_k,
    \qquad
    \operatorname{supp}\xi_k\subseteq hrC .
    \label{eq:bounded-random-recurrence}
\end{equation}
No independence or stationarity assumption is initially imposed. Indeed, the first question is purely deterministic: what can happen under any sequence of kicks compatible with the support constraint? To answer this, we utilize the exact-time controlled reachable sets introduced earlier. For $\rho>0$, $$
\mathcal P^d_{\rho,C}(A) := \left\{
        x_N:
        \begin{array}{l}
        N\geq0,\ x_0\in A,\\
        x_{k+1}=F_h(x_k)+e_k,\quad
        \gamma_C(e_k)<\rho
        \end{array}
    \right\}.
$$
The corresponding discrete $C$-intensity is
$$
    \mu_{h,C}(A)
    :=
    \sup\left\{
        \rho>0:
        \overline{\mathcal P^d_{\rho,C}(A)}
        \subset K\Subset\mathcal D(A)
        \text{ for some compact }K
    \right\}.
$$
Here $\mathcal D(A)$ denotes, as throughout, the basin of $A$ for the flow; by Lemma~\ref{lem:basin-identity} it coincides with the basin of $A$ for the sampled map $F_h$. This definition uses compact containment rather than mere inclusion in $\mathcal D(A)$. Consequently, intensities below the threshold provide a uniform safety margin from the basin boundary.

\begin{theorem}[Worst-case bounded-noise threshold]
\label{thm:bounded-noise-threshold}
Let $A$ be a compact attractor of $F_h$, and suppose $\mu_{h,C}(A)>0$.
\begin{enumerate}
    \item If $hr<\mu_{h,C}(A)$, then there exists a compact set $K\Subset\mathcal D(A)$, independent of the realization of $(\xi_k)$, such that every trajectory starting in $A$ remains in $K$. In particular, escape is impossible for every admissible noise realization.
    \item Suppose, in addition, that loss of compact containment occurs only through finite-time exit from the basin, namely that
    \begin{equation}
        \mu_{h,C}(A)
        =
        \inf\left\{
            \rho>0:
            \begin{array}{l}
            \text{there exist }N\geq1,\ x_0\in A,\text{ and }
            e_0,\ldots,e_{N-1}\in\rho C\\
            \text{such that }x_{k+1}=F_h(x_k)+e_k
            \text{ and }x_N\notin\mathcal D(A)
            \end{array}
        \right\}
        \label{eq:finite-exit-characterization}
    \end{equation}
    holds. If $hr>\mu_{h,C}(A)$, then there exists a finite admissible kick sequence $e_k\in hrC$ whose trajectory leaves $\mathcal D(A)$.
\end{enumerate}
\end{theorem}

\begin{proof}
If $hr<\mu_{h,C}(A)$, choose $hr<\rho<\mu_{h,C}(A)$. Every sequence satisfying $e_k\in hrC$ also satisfies $\gamma_C(e_k)<\rho$. By the definition of intensity, there is a compact set $K\Subset\mathcal D(A)$ containing $\overline{\mathcal P^d_{\rho,C}(A)}$. Hence every admissible sample path starting in $A$ remains in $K$. For the second assertion, choose $\mu_{h,C}(A)<\rho<hr$. By the definition of the infimum in \eqref{eq:finite-exit-characterization}, there exists $\rho'\in(\mu_{h,C}(A),\rho)$ and a finite escaping sequence with $\gamma_C(e_k)<\rho'<hr$. Hence each kick belongs to $hrC$.
\end{proof}

Condition \eqref{eq:finite-exit-characterization} rules out loss of compact containment caused solely by an unbounded perturbed orbit that never leaves the basin. It is precisely the conclusion \eqref{eq:discrete-first-contact-identity} of Proposition~\ref{prop:discrete-first-contact}, in the gauge form of Remark~\ref{rem:discrete-first-contact-variants}(ii), applied to $F=F_h$; it therefore holds whenever the perturbed reachable sets remain in a common compact set at every amplitude below the threshold, as in the setting of the sampling theorem. By Remark~\ref{rem:discrete-first-contact-variants}(i), it is immaterial
whether the kicks in \eqref{eq:finite-exit-characterization} are constrained to $\rho C$ or to its interior.

\begin{definition}[Worst-case escape threshold]
\label{def:escape-threshold}
The escape threshold expressed per unit physical time is $r_{\mathrm{esc},C}(h) := \frac{\mu_{h,C}(A)}{h}$. For Euclidean disturbances we write $r_{\mathrm{esc}}(h)=r_{\mathrm{esc},\overline B_1}(h)$.
\end{definition}

Theorem~\ref{thm:bounded-noise-threshold} identifies $r_{\mathrm{esc},C}(h)$ with the quantity announced in the
introduction: it is the supremum of the disturbance rates $r$ for which every admissible realization of
\eqref{eq:bounded-random-recurrence} starting in $A$ remains in a common compact subset of $\mathcal D(A)$. Safety holds for every $r<r_{\mathrm{esc},C}(h)$ by part~(i), while part~(ii) provides an admissible finite escape sequence for every $r>r_{\mathrm{esc},C}(h)$; the supremum is therefore $\mu_{h,C}(A)/h$ regardless of what happens at the threshold itself, which is the identity of Corollary~\ref{cor:bounded-noise}.

\begin{corollary}[Bounded-noise escape threshold under sampling]
\label{cor:bounded-noise}
For Euclidean disturbances, $C=\overline B_1(0)$, under the hypotheses of Theorem~\ref{thm:sampling},
$$
    \frac{\mu(A)}{1+Lh}
    \leq
    r_{\mathrm{esc}}(h)
    \leq
    \mu(A)\frac{e^{Lh}-1}{Lh}.
$$
It follows that $\lim_{h\downarrow0}r_{\mathrm{esc}}(h)=\mu(A)$, $r_{\mathrm{esc}}(h)=\mu(A)+O(h)$.
\end{corollary}

Thus the continuous-time intensity is precisely the small-sampling-time limit of the largest bounded disturbance rate for which pathwise safety is guaranteed.

Having established the condition for absolute safety, we now turn to the scenario where the threshold is exceeded. The two assertions in Theorem~\ref{thm:bounded-noise-threshold} are geometric, worst-case statements. They do not, by themselves, determine the probability of escape. In particular, the existence of an admissible escaping sequence does not imply that the random variables $\xi_k$ can produce that sequence with positive probability. The support of the law may be a proper subset of $hrC$, or the law may assign zero probability to the neighborhoods required to shadow a particular escape word.

The following statement separates accessibility from almost-sure escape.

\begin{proposition}[Positive-probability and almost-sure escape]
\label{prop:probabilistic-escape}
Assume $hr>\mu_{h,C}(A)$, and suppose there is a robust escaping word
    $e^\ast=(e_0^\ast,\ldots,e_{N-1}^\ast)
    \in(\operatorname{int}(hrC))^N$
such that the associated endpoint belongs to
$\operatorname{int}(\mathbb R^d\setminus\mathcal D(A))$.
\begin{enumerate}
    \item If the kicks are independent and each law has support equal to $hrC$, then escape in at most $N$ steps has positive probability.
    \item Let $\tau_{\mathcal D} := \inf\{k\geq0:X_k\notin\mathcal D(A)\}$. Suppose the process is Markov and that there exist $N\geq1$ and $p>0$ such that
    \begin{equation}
        \inf_{x\in\mathcal R}
        \mathbb P_x\{\tau_{\mathcal D}\leq N\}
        \geq p,
        \label{eq:uniform-exit-minorization}
    \end{equation}
    where $\mathcal R\subset\mathbb R^d$ is a Borel set containing $A$ such that, for every $x\in\mathcal R$, one has $X_k\in\mathcal R$ for all $k<\tau_{\mathcal D}$, $\mathbb P_x$-almost surely; one may take for $\mathcal R$ the closure of the set of states reachable from $A$ by admissible kicks before exit. (Since time is discrete,
    $\tau_{\mathcal D}$ is a stopping time.) Then
        $\mathbb P_x\{\tau_{\mathcal D}>mN\}
        \leq(1-p)^m$,
         $m\geq0$,
    for every $x\in\mathcal R$. Consequently,
\[
\mathbb P_x\{\tau_{\mathcal D}<\infty\}=1,
\qquad
\mathbb E_x[\tau_{\mathcal D}]
\leq
N\sum_{m=0}^{\infty}(1-p)^m
=
\frac{N}{p}.
\]
\end{enumerate}
\end{proposition}

\begin{proof}
Continuous dependence of a finite trajectory on its kicks, together with
the robust terminal crossing, gives open neighborhoods
$U_k\ni e_k^\ast$ such that every kick sequence in
$U_0\times\cdots\times U_{N-1}$ escapes. Full support gives
$\mathbb P\{\xi_k\in U_k\}>0$; independence therefore yields
    $\mathbb P\{(\xi_0,\ldots,\xi_{N-1})
    \in U_0\times\cdots\times U_{N-1}\}>0$.
For the second statement, the Markov property and
\eqref{eq:uniform-exit-minorization} imply inductively that
$$
    \mathbb P_x\{\tau_{\mathcal D}>(m+1)N\}
    \leq
    (1-p)\mathbb P_x\{\tau_{\mathcal D}>mN\}.
$$
The geometric estimate follows. Letting $m\to\infty$ proves almost-sure
escape, while summing the tail bound gives the stated expectation estimate.
\end{proof}

Full support and independence are therefore sufficient for a given robust
finite escape word to occur with positive probability, but they are not, in
general, sufficient for escape with probability one. Almost-sure escape
requires a recurrent opportunity to realize an escape word, such as the
uniform minorization condition above, a Harris-recurrence argument, or a
Lyapunov condition forcing repeated returns to a suitable launch region.

\subsection{Discrete and continuous boundary systems}
\label{subsec:continuous-boundary-system}

In this subsection, we trace the evolution of worst-case reachable boundaries. We begin with a discrete-time geometric boundary map, take its continuous-time limit, and reveal its fundamental equivalence to the extremals of the Pontryagin maximum principle. We next describe the geometry of the boundary of the set-valued time-$h$ map
\begin{equation}
    \mathcal F_{h,r}(x)
    :=
    F_h(x)+hr\overline B_1(0).
    \label{eq:set-valued-sampled-map}
\end{equation}
Assume in this subsection that $f\in C^2$ and that $F_h=\phi^h$ is a $C^2$ diffeomorphism on the compact region under consideration. Let $x$ be a smooth boundary point of a set $M$, and let $n\in\mathbb S^{d-1}$ be its outward unit normal.

Define the transported normal
\begin{equation}
    \eta_h(x,n)
    :=
    \frac{DF_h(x)^{-T}n}
         {\|DF_h(x)^{-T}n\|}.
    \label{eq:transported-normal}
\end{equation}
At smooth exposed boundary points for which the metric projection onto $F_h(M)$ is unique, the boundary map associated with \eqref{eq:set-valued-sampled-map} is
\begin{equation}
    \beta_{h,r}(x,n)
    =
    \left(
        F_h(x)+hr\,\eta_h(x,n),
        \eta_h(x,n)
    \right).
    \label{eq:discrete-boundary-map}
\end{equation}
This is the time-$h$ and radius-$hr$ specialization of the finite-dimensional boundary map used for random diffeomorphisms with spherical bounded noise; see \citep{lamb2026bifurcations}.

The inverse transpose in \eqref{eq:transported-normal} is geometric: a normal $m$ to $F_h(M)$ at $F_h(x)$ satisfies $m^\top DF_h(x)v=0$ whenever $n^\top v=0$, i.e.\ $DF_h(x)^\top m\parallel n$, so $m\parallel
DF_h(x)^{-T}n$; normalizing gives $\eta_h(x,n)$. Minkowski addition of $hr\overline B_1(0)$ then displaces an exposed boundary point by $hr$ along its normal, giving the first component of \eqref{eq:discrete-boundary-map}.

For a compact set $S$ with $C^1$ boundary, let $n_S(x)$ denote the outward unit normal at $x\in\partial S$ and write $$
    N^+_1\partial S
    :=
    \bigl\{(x,n_S(x)):x\in\partial S\bigr\}
    \subset\mathbb R^d\times\mathbb S^{d-1}
$$
for its outward unit normal bundle. The following proposition makes precise the propagation of exposed normal directions under $\beta_{h,r}$.

\begin{proposition}[Propagation of exposed normal directions]
\label{prop:normal-bundle-invariance}
Set $\rho:=hr>0$. Let $M\subset\mathbb R^d$ be compact with $C^1$ boundary, invariant under the set-valued map
\eqref{eq:set-valued-sampled-map} in the sense that
$$
    M
    =
    \mathcal F_{h,r}(M)
    =
    F_h(M)+\rho\overline B_1(0)
    =:
    E+\rho\overline B_1(0),
$$
and suppose that $F_h$ is a $C^2$ diffeomorphism of a neighborhood of $M$ onto its image. Call a pair $(x,n)\in N^+_1\partial M$ exposed if the point $z:=F_h(x)+\rho\,\eta_h(x,n)$ satisfies $\operatorname{dist}(z,E)=\rho$ with unique nearest point $F_h(x)\in E$. Then:
\begin{enumerate}
    \item every exposed $(x,n)\in N^+_1\partial M$ satisfies
    $\beta_{h,r}(x,n)\in N^+_1\partial M$;
    \item every $(z,\nu)\in N^+_1\partial M$ such that $z$ has a unique nearest point in $E$ equals $\beta_{h,r}(x,n)$ for exactly one $(x,n)\in N^+_1\partial M$, and this pair is exposed.
\end{enumerate}
In particular, if every $z\in\partial M$ has a unique nearest point in $E$ and every pair in $N^+_1\partial M$ is exposed, then $\beta_{h,r}
    \bigl(N^+_1\partial M\bigr)
    =
    N^+_1\partial M$.
\end{proposition}

\begin{proof}
Write $F:=F_h$. We use four elementary facts.

(a) \emph{Boundary points of a compact offset lie at distance exactly $\rho$.}
If $E$ is compact and $z\in\partial(E+\rho\overline B_1(0))$, then $\operatorname{dist}(z,E)=\rho$. Indeed, $z\in E+\rho\overline B_1(0)$ gives
$\operatorname{dist}(z,E)\leq\rho$; and if $\operatorname{dist}(z,E)=s<\rho$ with nearest point $y$, then
every $w$ with $\|w-z\|<\rho-s$ satisfies $\|w-y\|\leq\|w-z\|+s<\rho$, so a full neighborhood of $z$ lies in
$y+\rho\overline B_1(0)\subset E+\rho\overline B_1(0)$, contradicting $z\in\partial(E+\rho\overline B_1(0))$.

(b) \emph{Normal transport.} $E=F(M)$ is compact with $C^1$ boundary $\partial E=F(\partial M)$, and
$n_E(F(x))=\eta_h(x,n_M(x))$ for every $x\in\partial M$. Indeed, if $g$ is a local $C^1$ defining function for $M$ at $x$, so that $M=\{g\leq0\}$ locally and $\nabla g(x)=n_M(x)$, then $g\circ F^{-1}$ is a local defining function for $E$ at $F(x)$ with gradient $DF(x)^{-T}n_M(x)$ there; normalizing gives the claim, including the outward orientation.

(c) \emph{Nearest points on a $C^1$ boundary.} 
If $\operatorname{dist}(z,E)=\rho>0$ and $y\in E$ is a nearest point, then $y\in\partial E$ and $z=y+\rho\,n_E(y)$, since otherwise an interior neighborhood of $y$ would contain points closer to $z$. Minimality of $w\mapsto\|z-w\|^2$ along curves in $\partial E$ through $y$ gives $z-y\perp T_y\partial E$, so $z-y=\pm\rho\,n_E(y)$. If the sign were negative, the points $y-t\,n_E(y)$, which lie in $E$ for small $t>0$, would satisfy $\|z-(y-t\,n_E(y))\|=\rho-t<\rho$, a contradiction.

(d) \emph{Interior tangent balls determine the normal.} Let $S$ be compact with $C^1$ boundary,
$\overline B_\rho(c)\subset S$, and $z\in\partial S\cap\partial B_\rho(c)$; set $m:=(z-c)/\rho$. For every unit vector $v$ with $\langle m,v\rangle<0$, the points $z+tv$ lie in $B_\rho(c)\subset S$ for all small $t>0$, so a local defining function $g$ of $S$ at $z$ satisfies $g(z+tv)\leq0$, whence $\langle n_S(z),v\rangle\leq0$. By continuity, $\langle m,v\rangle\leq0$ implies $\langle n_S(z),v\rangle\leq0$; this inclusion of closed half-spaces between unit vectors forces $n_S(z)=m$.

\emph{Proof of (i).} Let $(x,n)\in N^+_1\partial M$ be exposed, and set $y:=F(x)$, $\eta:=\eta_h(x,n)$, and $z:=y+\rho\eta$, so that $\eta=n_E(y)$ by (b). Clearly $z\in E+\rho\overline B_1(0)=M$. We claim that $z+t\eta\notin M$ for all sufficiently small $t>0$, which gives $z\in\partial M$. First, $\operatorname{dist}(z,E)=\rho$ means $\|z-w\|\geq\rho$ for all $w\in E$; writing $w=y+v$ and expanding $\|z-w\|^2=\rho^2-2\rho\langle\eta,v\rangle+\|v\|^2$, this is equivalent to the supporting-ball inequality
\begin{equation}
\label{eq:supporting-ball}
    \langle\eta,v\rangle
    \leq
    \frac{\|v\|^2}{2\rho}
    \qquad
    \text{for every }w=y+v\in E.
\end{equation}
For $w=y+v\in E$ with $\|v\|\leq\rho$, using
\eqref{eq:supporting-ball},
$$
    \|z+t\eta-w\|^2
    =
    (\rho+t)^2-2(\rho+t)\langle\eta,v\rangle+\|v\|^2
    \geq
    (\rho+t)^2-\frac{t}{\rho}\|v\|^2
    \geq
    (\rho+t)^2-t\rho
    >
    \rho^2
$$
for every $t>0$. For $w\in E\setminus B_\rho(y)$, uniqueness of the nearest point and compactness of $E$ give
$c:=\min\{\|z-w\|:w\in E\setminus B_\rho(y)\}-\rho>0$, so $\|z+t\eta-w\|\geq\rho+c-t>\rho$ for $t<c$. Hence
$\operatorname{dist}(z+t\eta,E)>\rho$ for $0<t<c$, so $z+t\eta\notin M$ and $z\in\partial M$. Finally, the ball
$\overline B_\rho(y)\subset E+\rho\overline B_1(0)=M$ touches $\partial M$ at $z$, so (d) gives $n_M(z)=(z-y)/\rho=\eta$. Thus $\beta_{h,r}(x,n)=(z,\eta)\in N^+_1\partial M$.

\emph{Proof of (ii).} Let $(z,\nu)\in N^+_1\partial M$ and let $y$ be the unique nearest point of $z$ in $E$. By (a), $\operatorname{dist}(z,E)=\rho$; by (b), $\partial E$ is $C^1$; by (c), $y\in\partial E$ and $z=y+\rho\,n_E(y)$; and by (d), applied to the ball $\overline B_\rho(y)\subset M$, $\nu=n_M(z)=n_E(y)$. Since $F$ is a diffeomorphism, there is a unique $x:=F^{-1}(y)\in\partial M$, and (b) gives $n_E(y)=\eta_h(x,n_M(x))$. Hence $(z,\nu)=\beta_{h,r}(x,n_M(x))$, and the pair $(x,n_M(x))$ is exposed by construction. For uniqueness, suppose $\beta_{h,r}(x',n')=(z,\nu)$ for some $(x',n')\in N^+_1\partial M$. Then $\|z-F(x')\|=\rho\,\|\eta_h(x',n')\|=\rho$, so $F(x')$ is a nearest point of $z$ in $E$, whence $F(x')=y$, $x'=x$, and $n'=n_M(x')$ is determined by $x'$.
\end{proof}

The exposedness qualification is important. The boundary map propagates smooth exposed points and their normals, but it need not globally parameterize the entire boundary after folds, loss of exposure, multiple
contributors, or self-intersections. Nonsmooth boundaries must instead be described using normal cones. Recent computations for the H\'enon map show that local and global bifurcations of the boundary map can produce both
topological changes of a bounded-noise attractor and the creation of wedge or shallow boundary singularities \citep{lamb2026bifurcations}. Rather than tracking these complex topological bifurcations, we use the discrete boundary map strictly to derive the continuous-time boundary dynamics. Taking the limit $h \downarrow 0$ at a fixed $r$, we obtain the following expansion uniformly on compact subsets on which $f\in C^2$,
\begin{align*}
    F_h(x)
    &=
    x+hf(x)+O(h^2),\\
    DF_h(x)
    &=
    I+hDf(x)+O(h^2),\\
    DF_h(x)^{-T}
    &=
    I-hDf(x)^\top+O(h^2).
\end{align*}
For $n\in\mathbb S^{d-1}$, set $B=Df(x)^\top$. Then $DF_h(x)^{-T}n=n-hBn+O(h^2)$ and $\|DF_h(x)^{-T}n\| = 1-hn^\top Bn+O(h^2)$. It follows that
    $\eta_h(x,n)
    =
    n-hBn+h(n^\top Bn)n+O(h^2)
    =
    n-h(I-nn^\top)Df(x)^\top n+O(h^2)$.
Substitution into \eqref{eq:discrete-boundary-map} gives
\begin{equation}
    \beta_{h,r}(x,n)
    =
    (x,n)
    +
    h
    \begin{pmatrix}
        f(x)+rn\\[1mm]
        -(I-nn^\top)Df(x)^\top n
    \end{pmatrix}
    +O(h^2).
    \label{eq:boundary-map-consistency}
\end{equation}

\begin{theorem}[Small-step limit of the boundary map]
\label{thm:boundary-limit}
On every compact regular part of the unit normal bundle, $\beta_{h,r}$ is a first-order consistent discretization of the boundary system
\begin{align}
    \dot x
    &=
    f(x)+rn,
    \label{eq:continuous-boundary-x}\\
    \dot n
    &=
    -(I-nn^\top)Df(x)^\top n.
    \label{eq:continuous-boundary-n}
\end{align}
More precisely, the local truncation error in \eqref{eq:boundary-map-consistency} is $O(h^2)$, uniformly on compact regular subsets of the unit normal bundle.
\end{theorem}

The vector field in \eqref{eq:continuous-boundary-n} is tangent to $\mathbb S^{d-1}$, since
    $\frac{d}{dt}\|n\|^2
    =
    -2n^\top(I-nn^\top)Df(x)^\top n
    =
    0.$
Thus a unit initial normal remains a unit normal for as long as the solution exists.

The boundary system follows from the Pontryagin maximum principle. Consider the controlled system
$\dot x=f(x)+u$, $\|u\|\leq r$. The maximized Hamiltonian is
\[
    H_r(x,p)
    :=
    \max_{\{\|u\|\leq r\}}
    p^\top(f(x)+u)
    =
    p^\top f(x)+r\|p\|.
\]
For $p\neq0$, the maximizing control is $u^\ast = r\frac{p}{\|p\|}$. The associated costate equation is $\dot p=-Df(x)^\top p$.  For nonzero costates $p\neq0$, introducing the normalized direction $n=\frac{p}{\|p\|}$, and differentiating yields $\dot n = -(I-nn^\top)Df(x)^\top n$, while the state equation becomes $\dot x=f(x)+rn$. Hence \eqref{eq:continuous-boundary-x}--\eqref{eq:continuous-boundary-n} describe the normalized Pontryagin extremal system corresponding to exposed boundary directions.

This interpretation should not be confused with a global optimality statement. The boundary system generates extremal trajectories associated with supporting directions. After the occurrence of folds, conjugate points, or loss of exposure, such an extremal may enter the interior of the reachable set and no longer represent the global boundary. Therefore, intensity computation requires both integration of the boundary system and an independent global first-contact test with $\partial\mathcal D(A)$.

The same construction extends to anisotropic disturbances. Let $h_C(p) := \max_{c\in C}p^\top c$
be the support function of the convex disturbance body $C$. The Hamiltonian becomes $H_{r,C}(x,p) = p^\top f(x)+r h_C(p)$. If $C$ is strictly convex and $h_C$ is differentiable away from the origin, the maximizing disturbance is uniquely given by $u^\ast = r\nabla h_C(p)$. Consequently, the Pontryagin system takes the form $\dot x = f(x)+r\nabla h_C(p)$ and $\dot p = -Df(x)^\top p$. Unlike the Euclidean ball case, the maximizing disturbance direction is not generally parallel to the costate. Hence the normalized direction $n=p/\|p\|$ no longer closes the system by itself; the natural anisotropic boundary dynamics are formulated in the costate variable $p$.

For nonsmooth or non-strictly-convex disturbance bodies, the maximizer need not be unique. In that case, the gradient is replaced by the subdifferential of the support function,
$\partial h_C(p)
    =
    \operatorname{co}
    \left\{
        c\in C:
        p^\top c=h_C(p)
    \right\}$,
and the state equation becomes the differential inclusion
$\dot x
    \in
    f(x)+r\,\partial h_C(p)$.
The state equation becomes set-valued. State-dependent disturbance sets would introduce additional geometric terms and are not considered here.

\subsection{Nonautonomous boundary dynamics and unbounded noise}
\label{subsec:nonautonomous-and-unbounded}

The boundary system construction extends naturally to nonautonomous environments. The same geometric construction applies fiberwise to the skew-product system introduced in Section~\ref{sec:nonautonomous}.
For a fixed base trajectory $\omega\in\Omega$, the fiberwise Hamiltonian is
$H_r^\omega(t,x,p)=p^\top f(\theta_t\omega,x)+r\|p\|$. For nonzero costates $p$, setting $n=p/\|p\|$ gives the fiberwise boundary system
\[
\dot x=f(\theta_t\omega,x)+rn,
\qquad
\dot n
=
-(I-nn^\top)D_xf(\theta_t\omega,x)^\top n,
\]
which defines a skew-product flow on $\Omega\times\mathbb R^d\times\mathbb S^{d-1}$.

The sampled boundary map is
\[
\mathcal B_{h,r}
(\omega,x,n)
=
\left(
\theta_h\omega,
\varphi(h,\omega,x)+hr\,\eta_h^\omega(x,n),
\eta_h^\omega(x,n)
\right),
\]
where
\[
\eta_h^\omega(x,n)
=
\frac{
D_x\varphi(h,\omega,x)^{-T}n
}{
\|D_x\varphi(h,\omega,x)^{-T}n\|
}.
\]
It is a first-order discretization of the boundary system when $f$ is $C^2$ in $x$ and $D_xf$ is uniformly bounded on a common trapping region. Deterministic finite-dimensional boundary systems for nonautonomous
differential inclusions are developed in greater generality in \citep{kourliouros2026invariant}; we do not claim that formalism here. The key point is that the scaling of the sampled noise as $hrC$ is the
natural scaling both for the intensity limit and for the convergence of the discrete boundary map to the Pontryagin boundary system.

While the boundary system formalism extends to nonautonomous environments, the threshold theory fundamentally relies on bounded noise support. If the noise law has unbounded support and assigns positive probability to every nonempty open set, then for every $R>0$, $\mathbb P(\|\xi_k\|>R)>0$. Hence no finite deterministic amplitude bound can guarantee pathwise safety, and the worst-case support threshold degenerates. The relevant questions become probabilistic---first exit time $\tau_{\mathcal D}$, its mean,
finite-horizon escape probabilities, and, in the small-noise regime, the Freidlin--Wentzell action and quasipotential. Under the classical Freidlin--Wentzell assumptions for metastable small-noise diffusions \citep{freidlinrandom,berglund2006noise},
\[
\varepsilon
\log\mathbb E_x[\tau_{\mathcal D}^{\varepsilon}]
\to
V_{\partial\mathcal D}
\]
.
Intensity, based on an $L^\infty$-type amplitude constraint, and the quasipotential, based on an action-minimizing principle for exponentially unlikely paths, answer different questions and need not determine each
other. Numerical methods for quasipotential computation are surveyed in \citep{grafke2019numerical,cameron2012finding}, and diffusion simulation is standard \citep{kloeden2012numerical}. A systematic comparison is left for future work.

\section{Computation, examples, and conclusions}
\label{sec:examples}

This section does three things beyond illustrating trajectories: it verifies the sampling theorems against exact and independently computed thresholds, it tests the boundary-system computation on a genuinely
two-dimensional escape problem validated by independent methods, and it demonstrates the distinction between worst-case and probabilistic escape under bounded noise. The experiments are:
\begin{enumerate}
    \item an exact scalar verification with a separate sharpness test for the general $O(h)$ rate (Section~\ref{subsec:exact-scalar-benchmark});
    \item an applied bistable grazing model in which intensity is an early-warning indicator vanishing at a tipping fold (Section~\ref{subsec:grazing});
    \item a planar damped-Duffing escape threshold by three independent methods with an extremal-control replay check (Sections~\ref{subsec:planar-radial-example}--\ref{subsec:duffing-escape}), and its anisotropic $C$-intensity for elliptical disturbance bodies (Section~\ref{subsec:anisotropic-example});
    \item nonautonomous invariant-graph experiments---a periodically forced graph with verified fold residuals, sampling to $N=512$, phase uniformity, and an averaging sweep, and a genuinely quasiperiodic graph on the two-torus (Section~\ref{subsec:nonautonomous-computation});
    \item bounded-noise escape separating the support-wise threshold from finite-horizon escape probabilities via an absorbing transfer operator validated by Monte Carlo (Section~\ref{subsec:bounded-noise-experiment}).
\end{enumerate}

All optimization problems were solved independently for each parameter value. Reference and replay integrations used the DOP853 method with relative and absolute tolerances $2\times10^{-12}$ and $2\times10^{-14}$; vectorized sweeps used fixed-step classical Runge--Kutta with step sizes stated in each experiment, chosen so that the integration error is at least two orders of magnitude below the smallest reported difference. Scalar maximizations used bounded Brent optimization with tolerance $10^{-15}$. Random seeds are stated where randomness enters. Reported convergence rates are $p_j := \frac{\log(E_j/E_{j+1})}{\log(h_j/h_{j+1})}$, where $E_j$ denotes the relevant absolute error.

Table~\ref{tab:assumption-result-summary} records the hypotheses actually used by each principal result, preventing numerical agreement in the examples from blurring unrestricted intensity, block intensity, and probabilistic escape.

\begin{table}[htbp]
    \centering
    \caption{Assumption-and-result summary.}
    \label{tab:assumption-result-summary}
    % 取消了 \resizebox，直接利用 p 列配合 @{} 消除边缘留白
    \begin{tabular}{@{}p{.22\linewidth} p{.36\linewidth} p{.36\linewidth}@{}}
        \toprule
        Result & Principal assumptions & Consequence\\
        \midrule
        Autonomous sampling
        &
        Forward completeness; compact attractor; common compact controlled
        region; Lipschitz constant \(L\)
        &
        \(\mu_h/h\to\mu\), with explicit two-sided bounds
        \\
        Numerical discretization
        &
        One-step consistency of order \(q\); common trapping block;
        continued attractor inside that block
        &
        \(\mu_h^\Psi/h=\mu+O(h)+O(h^q)\) for block intensity
        \\
        Nonautonomous sampling
        &
        Compact invertible base flow; uniform fiberwise Lipschitz bound;
        attracting invariant graph; common compact fiber block
        &
        \(\mu_{\mathrm{na},h}/h\to\mu_{\mathrm{na}}\) uniformly over the base
        \\
        Pathwise bounded-noise safety
        &
        \(\operatorname{supp}\xi_k\subseteq hrC\) and
        \(hr<\mu_{h,C}\)
        &
        No admissible realization leaves the deterministic basin
        \\
        Possible finite escape
        &
        \(hr>\mu_{h,C}\) and the finite-exit characterization
        &
        Existence of an admissible finite escape word
        \\
        Almost-sure escape
        &
        Robust accessibility plus independence/Markov property and a
        uniform minorization or recurrence condition
        &
        Escape with probability one; geometric survival bound
        \\
        Boundary-map limit
        &
        \(C^2\) flow; locally invertible time-\(h\) map; smooth exposed
        boundary point
        &
        Convergence to the normalized Pontryagin boundary system
        \\
        \bottomrule
    \end{tabular}
\end{table}

\subsection{Exact scalar benchmarks: validation, sharpness, and numerical error separation}
\label{subsec:exact-scalar-benchmark}

This example serves three purposes. First, it provides an exactly solvable benchmark for validating the sampling theorem without numerical ambiguity. Second, it illustrates that the general first-order estimate may be non-sharp for smooth systems due to higher-order cancellations. Third, it provides a reference point for separating the intrinsic sampling error from the error introduced by numerical integrators.

Consider
\begin{equation}
    \dot x=x-x^3=x(1-x^2).
    \label{eq:scalar-benchmark}
\end{equation}
The equilibria at \(x=\pm1\) are attracting, while \(x=0\) is unstable. We take $A=\{1\}$, $\mathcal D(A)=(0,\infty)$. A control driving the system from \(1\) toward the basin boundary acts in the negative direction:
\[
    \dot x=x-x^3-u,
    \qquad 0\leq u\leq r.
\]
For this one-dimensional bistable system, the reachable-set definition reduces to overcoming the maximal restoring drift along the interval connecting the attractor to the basin boundary. Since $f'(x)=1-3x^2$, the maximum is attained at $x_*=\frac1{\sqrt3}$, which gives $\mu(A)=f(x_*)=\frac{2}{3\sqrt3}$.

The flow of \eqref{eq:scalar-benchmark} is explicit. For \(x>0\), $F_h(x)=\phi^h(x) = \frac{xe^h}{\sqrt{1+x^2(e^{2h}-1)}}$. Since \(F_h\) is order preserving, the most adverse discrete perturbation is the constant negative kick; by Proposition~\ref{prop:scalar-one-sided}, applied with \(b=0\) and
\(a=1\) (its hypotheses hold because \(f\leq0\) on \([1,\infty)\) and \(f(x)\to-\infty\) as \(x\to\infty\)), the discrete intensity is $\mu_h(A) = \max_{0<x<1}\bigl(F_h(x)-x\bigr)$. Indeed, an $\varepsilon$-kick can cross the basin boundary only when the negative displacement exceeds the gap between the sampled trajectory and the unstable equilibrium at the origin. Therefore the critical kick amplitude is the maximal value of $F_h(x)-x$ over the interval $(0,1)$. Order preservation makes the constant negative kick extremal; maximizing $F_h(x)-x$ (the maximizer solves $F_h'(x_h)=1$, i.e. $x_h^2=(e^{2h/3}-1)/(e^{2h}-1)$) yields the exact closed form
\begin{equation}
    \mu_h(A)
    =
    \frac{
        \bigl(e^{2h/3}-1\bigr)^{3/2}
    }{
        \bigl(e^{2h}-1\bigr)^{1/2}
    }.
    \label{eq:scalar-exact-discrete-intensity}
\end{equation}
The closed-form expression provides an exact reference value for the sampled intensity. The numerical maximization by Brent's method agrees with this formula to machine precision, confirming the implementation of the reachable-set computation. 
This example also highlights a subtle point concerning the localization of the Lipschitz constant in Theorem~\ref{thm:sampling}.
Theorem~\ref{thm:sampling} cannot be used with $L=\max_{[0,1]}|f'|=2$
here: the reachable set from $A=\{1\}$ extends past the attractor to $x=2/\sqrt3$ (where $|f'|=3$) at the critical amplitude, where the fold point of the controlled vector field is located and outside the region where $L=2$ is valid. For scalar one-sided basins this is unnecessary---Proposition~\ref{prop:scalar-one-sided} gives bounds using only $f|_{[0,1]}$,
\begin{equation}
    \mu(A)\,\frac{1-e^{-2h}}{2h}
    \leq
    \frac{\mu_h(A)}h
    \leq
    \mu(A).
    \label{eq:scalar-theoretical-bounds}
\end{equation}
Table~\ref{tab:scalar-bounds} verifies both inequalities and reports the observed rates $p_h=\log_2(E_h/E_{h/2})$; in accordance with the one-sided upper bound, the normalized discrete intensity approaches \(\mu(A)\) from below.

\begin{table}[t]
    \centering
    \caption{Exact normalized sampled intensity for the smooth scalar benchmark. The numerical values satisfy the theoretical one-sided bounds and converge to $\mu(A)$ with second-order accuracy.}
    \label{tab:scalar-bounds}
    \begin{tabular}{cccccc}
        \toprule
        \(h\)
        &
        lower bound
        &
        \(\mu_h/h\)
        &
        upper bound
        &
        \(E_h=|\mu_h/h-\mu|\)
        &
        \(p_h\)
        \\
        \midrule
        \(0.4000\) & \(0.264941726\) & \(0.381519859\)
                   & \(0.384900179\) & \(3.3803\times10^{-3}\) & \(1.99\)\\
        \(0.2000\) & \(0.317234684\) & \(0.384047435\)
                   & \(0.384900179\) & \(8.5274\times10^{-4}\) & \(2.00\)\\
        \(0.1000\) & \(0.348852828\) & \(0.384686508\)
                   & \(0.384900179\) & \(2.1367\times10^{-4}\) & \(2.00\)\\
        \(0.0500\) & \(0.366280949\) & \(0.384846731\)
                   & \(0.384900179\) & \(5.3448\times10^{-5}\) & \(2.00\)\\
        \(0.0250\) & \(0.375436065\) & \(0.384886816\)
                   & \(0.384900179\) & \(1.3364\times10^{-5}\) & \(2.00\)\\
        \(0.0125\) & \(0.380128772\) & \(0.384896838\)
                   & \(0.384900179\) & \(3.3411\times10^{-6}\) & --\\
        \bottomrule
    \end{tabular}
\end{table}

Interestingly, the smooth benchmark exhibits a higher-order convergence than predicted by the general theorem: a notable cancellation occurs in this example.  Expanding \eqref{eq:scalar-exact-discrete-intensity} gives
\begin{equation}
    \frac{\mu_h(A)}h
    =
    \mu(A)
    \left(
        1-\frac{h^2}{18}+O(h^3)
    \right).
    \label{eq:scalar-second-order-expansion}
\end{equation}
Thus the exact normalized intensity converges with order two, even though Theorem~\ref{thm:sampling} guarantees only order one in general. This does not contradict the theorem: the \(O(h)\) bounds are uniform
worst-case estimates and need not be sharp. The cancellation follows from the expansion
\[
F_h(x)-x
=
hf(x)+\frac{h^2}{2}f'(x)f(x)+O(h^3).
\]
At the maximizing point $x_*$ of $f$, one has $f'(x_*)=0$, and hence the first correction term vanishes, in accordance with Corollary~\ref{cor:scalar-second-order}, which gives the a priori bound \(0\leq\mu-\mu_h/h\leq(L_2\mu^2/6)h^2\) with \(L_2=\max_{[0,1]}|f''|=6\).

Therefore, this smooth example cannot test the optimality of the first-order theorem. To show that the $O(h)$ rate is unavoidable, we next consider a nonsmooth scalar example. That role is played by the kinked vector field of Example~\ref{ex:sharp-scalar}, for which
    $\frac{\mu_h(A)}h-\mu(A)
    =
    \frac{1-e^{-h/2}}{h}-\frac12
    =
    \Theta(h)$.
The computed errors for $h\in\{0.4,0.2,0.1,0.05,0.025,0.0125\}$ are $4.68\times10^{-2}$ down to $1.56\times10^{-3}$ with observed slopes $0.953,\;0.976,\;0.988,\;0.994,\;0.997$, approaching one from below. The one-sided bounds themselves are plotted in Figure~\ref{fig:autonomous-validation}(a).

\begin{figure}[htbp]
  \centering
  \includegraphics[width=0.5\textwidth]{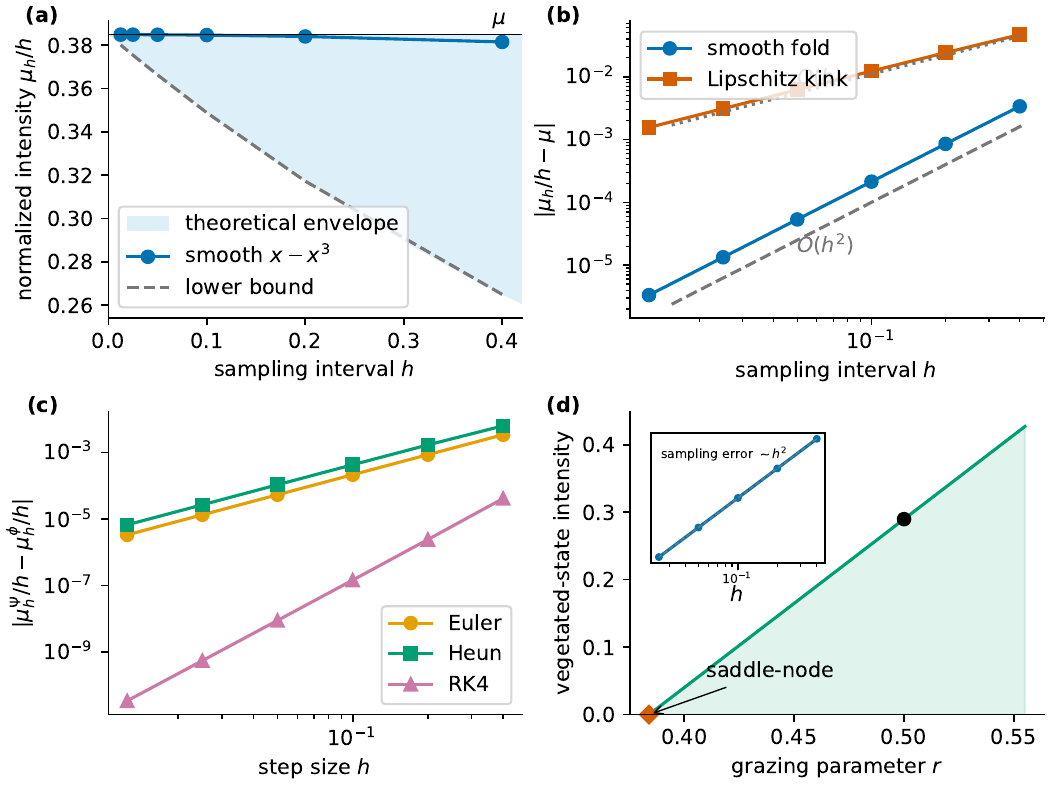}
  \caption{Autonomous sampling and numerical validation.  (a)~For
  $\dot x=x-x^3$, the exact normalized sampled intensity lies inside the
  theoretical envelope and converges to $\mu=2/(3\sqrt3)$.  (b)~A Lipschitz
  kink realizes first-order convergence, whereas a smooth scalar fold exhibits
  second-order cancellation.  (c)~The additional one-step discretization error
  follows the expected method-dependent rate.  (d)~In the grazing model,
  intensity tends to zero at the saddle-node; the inset confirms second-order
  sampling convergence in the representative bistable regime.}
  \label{fig:autonomous-validation}
\end{figure}

Figure~s\ref{fig:autonomous-validation}(b)--(c) separates the sampling error from the numerical one-step error. For a numerical map \(\Psi_h\), we compute the block intensity
    $\mu_h^\Psi
    =
    \max_{0<x<1}\bigl(\Psi_h(x)-x\bigr)$
and compare it with the exact-map value \(\mu_h\). Table~\ref{stab:integrators} gives the corresponding errors for explicit Euler, Heun's method, and classical fourth-order Runge--Kutta. The observed slopes are $2,2,4$ for the three methods relative to the exact sampled map.

\begin{table}[htbp]
    \centering
    \caption{Sampling and numerical-integrator errors for the scalar benchmark. The method columns report \(\bigl|\mu_h^\Psi/h-\mu_h^\phi/h\bigr|\).}
    \label{stab:integrators}
    \begin{tabular}{ccccc}
        \toprule
        \(h\)
        &
        \(\bigl|\mu_h^\phi/h-\mu\bigr|\)
        &
        Euler
        &
        Heun
        &
        RK4
        \\
        \midrule
        \(0.2000\) & \(8.5274\times10^{-4}\)
                   & \(8.5274\times10^{-4}\)
                   & \(1.6750\times10^{-3}\)
                   & \(2.4224\times10^{-6}\)\\
        \(0.1000\) & \(2.1367\times10^{-4}\)
                   & \(2.1367\times10^{-4}\)
                   & \(4.2547\times10^{-4}\)
                   & \(1.4675\times10^{-7}\)\\
        \(0.0500\) & \(5.3448\times10^{-5}\)
                   & \(5.3448\times10^{-5}\)
                   & \(1.0678\times10^{-4}\)
                   & \(9.0374\times10^{-9}\)\\
        \(0.0250\) & \(1.3364\times10^{-5}\)
                   & \(1.3364\times10^{-5}\)
                   & \(2.6721\times10^{-5}\)
                   & \(5.6076\times10^{-10}\)\\
        \(0.0125\) & \(3.3411\times10^{-6}\)
                   & \(3.3411\times10^{-6}\)
                   & \(6.6818\times10^{-6}\)
                   & \(3.4930\times10^{-11}\)\\
        \midrule
        estimated slope
                   & \(2.00\) & \(2.00\) & \(2.00\) & \(4.01\)\\
        \bottomrule
    \end{tabular}
\end{table}

For explicit Euler,
    $\Psi_h^{\mathrm E}(x)=x+h(x-x^3)$,
so that
    $\frac{\mu_h^{\mathrm E}}h
    =
    \max_{0<x<1}(x-x^3)
    =
    \mu(A)$
exactly. Its error relative to the exact sampled map is consequently precisely the second-order sampling correction in \eqref{eq:scalar-second-order-expansion}. This apparent higher-order behavior is not a generic property of Euler's method; it results from the exact coincidence between the Euler drift increment and the continuous maximizer in this particular polynomial example.

\subsection{An applied resilience example: intensity as an early-warning indicator}
\label{subsec:grazing}

The previous scalar examples validated the mathematical properties of sampling intensity. We now demonstrate its interpretation as a resilience indicator in an applied bistable system. Consider May's grazing model
\citep{menck2013basin,scheffer2001catastrophic}
\begin{equation}
    \dot x=r\,x\Bigl(1-\frac{x}{K}\Bigr)-\frac{x^2}{1+x^2},
    \qquad K=10.
    \label{eq:grazing}
\end{equation}
For a range of grazing parameters $r$, the system exhibits bistability: a vegetated equilibrium $x_+$ and a degraded equilibrium coexist, separated by an unstable equilibrium $x_{\mathrm u}$. The basin of the vegetated state is the half-line $(x_{\mathrm u},\infty)$ in the nonnegative state space. Since escape can occur only through the left basin boundary, Proposition~\ref{prop:scalar-one-sided} yields $\mu(x_+) = \max_{x_{\mathrm u}<x<x_+} f(x;r)$. As $r$ approaches the saddle-node value $r_*$, the stable and unstable equilibria coalesce and the restoring drift barrier vanishes (Figure~\ref{fig:tipping}). Consequently, $\mu(x_+)\to0$. The vegetated equilibrium therefore loses resilience before the bifurcation destroys its existence. Unlike statistical early-warning indicators based on fluctuations \citep{scheffer2009early}, intensity provides a deterministic measure of the disturbance amplitude required to leave the basin of attraction. For a representative bistable regime, we choose $r=0.5$, where both equilibria are well separated: $x_{\mathrm u}=2.000$, $x_+=7.317$. The intensity is $\mu=0.28966$, and the sampled intensity converges as $\mu_h(x_+)/h\to\mu$ with observed slope $2.00$ (Figure~\ref{fig:autonomous-validation}(d)).
As in the smooth benchmark, the observed second-order convergence is a problem-dependent cancellation and is not a contradiction of the general first-order estimate of Theorem~\ref{thm:sampling}. 
The quantity $r_{\mathrm{esc}}(h)=\mu_h/h$ therefore represents the critical bounded disturbance rate for the
sampled dynamics, and converges to the continuous-time resilience threshold $\mu(x_+)$ as $h\to0$. This example demonstrates that the abstract intensity theory yields a sampling-independent resilience metric in an applied bistable system.

\begin{figure}[htbp]
    \centering
    \includegraphics[width=0.5\linewidth]{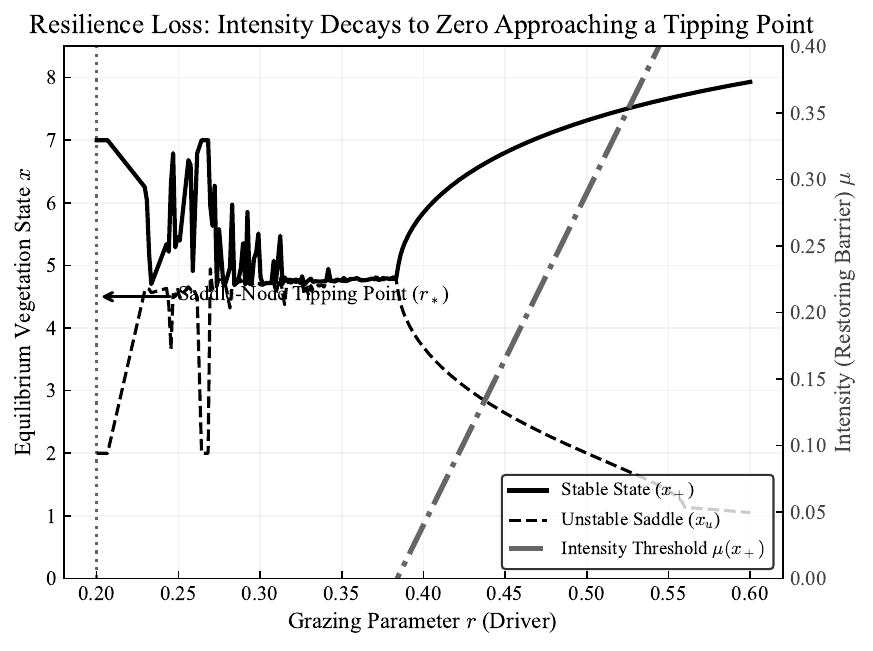}
    \caption{Resilience loss: intensity decays to zero approaching a tipping point.}
    \label{fig:tipping}
\end{figure}

\subsection{A planar radial example: global resilience beyond linear stability}
\label{subsec:planar-radial-example}

The previous examples were one-dimensional and therefore reduced the escape problem to an interval barrier. We now consider a genuinely two-dimensional system to demonstrate that intensity captures global basin
geometry beyond local linear stability. Consider the radial vector field
\[
\dot z=f_a(z)
=
\frac1{a^2}
(\|z\|^2-a^2)(1-\|z\|^2)z,
\qquad z\in\mathbb R^2,
\]
with $0<a<1$. In polar coordinates, \(\rho=\|z\|\) satisfies 
\begin{equation}
\dot\rho = g_a(\rho) := \frac1{a^2} \rho(\rho^2-a^2)(1-\rho^2).
    \label{eq:radial-scalar-equation}
\end{equation}
Since $g_a(\rho)<0$ for $0<\rho<a$, trajectories inside the circle $\rho=a$ converge to the origin, and therefore $\mathcal D(A_0)=B_a(0)$. The circle \(\|z\|=a\) is the basin boundary. Furthermore, the circle $\|z\|=1$ is an outer attracting continuum of equilibria; in the radial equation it appears as the stable equilibrium $\rho=1$.

The linearization at the origin is \(Df_a(0)=-I\) for every \(a\in(0,1)\). Thus both local eigenvalues equal \(-1\), independently of the basin radius and of the transient radial drift. This provides a direct counterexample to any attempt to infer resilience from local eigenvalues alone. Nevertheless, the intensity varies with \(a\). For \(0<\rho<a\), the inward restoring speed is
\begin{equation}
    q_a(\rho)
    :=
    -g_a(\rho)
    =
    \frac{\rho(a^2-\rho^2)(1-\rho^2)}{a^2}.
    \label{eq:radial-restoring-speed}
\end{equation}
Rotational symmetry and the comparison principle imply
    $\mu(A_0)
    =
    \max_{0<\rho<a}q_a(\rho)$.
Indeed, every admissible controlled trajectory satisfies \(\tfrac{d}{dt}\|z\|\leq g_a(\|z\|)+\|u(t)\|\) for $z\neq 0$, so for \(r<\max q_a\) no trajectory from the origin can cross a circle on which \(q_a>r\); conversely, the outward radial control \(u=r\,z/\|z\|\) whenever $\rho>0$ realizes equality in this estimate, so every amplitude \(r>\max q_a\) produces escape. This is the mirrored, radial form of the argument proving Proposition~\ref{prop:scalar-one-sided}.

If \(\rho_\ast\) denotes the maximizing radius, differentiation of \eqref{eq:radial-restoring-speed} yields \(5\rho_\ast^4 - 3(a^2+1)\rho_\ast^2 + a^2 = 0\), with relevant root
    $\rho_\ast^2
    =
    \frac{
        3(a^2+1)
        -
        \sqrt{9(a^2+1)^2-20a^2}
    }{10}$,
and therefore
    $\mu(A_0)
    =
    \frac{
        \rho_\ast
        (a^2-\rho_\ast^2)
        (1-\rho_\ast^2)
    }{a^2}$.
Changing \(a\) while leaving \(Df_a(0)=-I\) unchanged gives:
\[
\begin{array}{c|ccc}
    a & 0.45 & 0.60 & 0.70\\
    \hline
    \mu(A_0)
      &0.1620566373
      &0.2055508988
      &0.2304804052 .
\end{array}
\]
Hence, local eigenvalues alone cannot determine intensity.

For a control bound \(r<\mu(A_0)\), let $\rho_-(r)$ denote the smaller positive root of $q_a(\rho)=r$, which determines the outermost radius reachable without crossing the restoring barrier. The reachable set from \(A_0\) is then the disk
    $\overline{\mathcal P_r(A_0)}
    =
    \overline B_{\rho_-(r)}(0)$.
(The union over amplitudes \(r'<r\) of the disks of radius \(\rho_-(r')\) has this closure because \(r'\mapsto\rho_-(r')\) is continuous and increasing below the fold.) The boundary system of Section~\ref{subsec:continuous-boundary-system} reduces exactly to the radial equation: along a radial extremal, \(n=z/\|z\|\):
\begin{equation}\label{eq:radial-boundary-system}
\dot z=f_a(z)+rn,\qquad \dot n=0.
\end{equation}
Because the Jacobian preserves radial directions, the adjoint direction remains constant after normalization, yielding $\dot n=0$. At the critical amplitude, $r=\mu(A_0)=q_a(\rho_*)$, the two solutions of $q_a(\rho)=r$ coalesce at the maximizer $\rho_*$; for \(r>\mu(A_0)\), the radial extremal crosses the basin boundary (Figure~\ref{sfig:radial}).

\begin{figure}[htbp]
    \centering
    \includegraphics[width=0.5\linewidth]{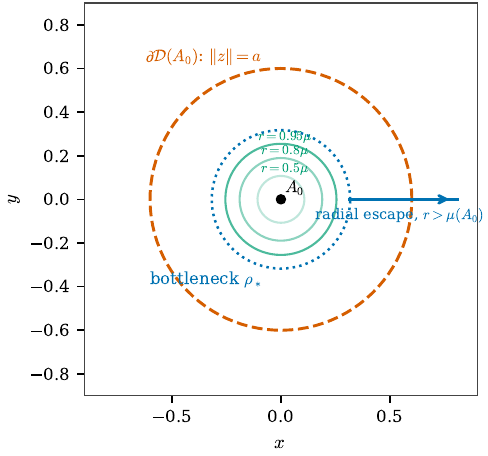}
    \caption{Radially symmetric verification case with \(a=0.6\). Reachable-set boundaries (circles of radius $\rho_-(r)$) for increasing control amplitudes approach the bottleneck circle $\rho_\ast$; for $r>\mu(A_0)$ the radial extremal \eqref{eq:radial-boundary-system} crosses the basin boundary
    \(\|z\|=a\).}
    \label{sfig:radial}
\end{figure}

Let \(\Phi_h(\rho)\) denote the time-\(h\) map of \eqref{eq:radial-scalar-equation}. Since \(\|\phi^h(z)\|=\Phi_h(\|z\|)\) and \(\|z+e\|\leq\|z\|+\|e\|\), the radius along any perturbed orbit is dominated by the extremal radial recursion \(\rho_{k+1}=\Phi_h(\rho_k)+\varepsilon\), which outward radial kicks realize; the exact-map discrete intensity is therefore
    $\mu_h(A_0)
    =
    \max_{0<\rho<a}
    \bigl(\rho-\Phi_h(\rho)\bigr)$.
For \(a=0.6\), the resulting normalized values are:
\[
\begin{array}{c|cc}
    h & \mu_h(A_0)/h
      & |\mu_h(A_0)/h-\mu(A_0)|\\
    \hline
    0.400 &0.2040412035&1.5097\times10^{-3}\\
    0.200 &0.2051704530&3.8045\times10^{-4}\\
    0.100 &0.2054555963&9.5302\times10^{-5}\\
    0.050 &0.2055270612&2.3838\times10^{-5}\\
    0.025 &0.2055449387&5.9601\times10^{-6}.
\end{array}
\]
As in the scalar smooth benchmark, the observed second-order convergence (slope \(2.00\) at \(a=0.6\)) is a special cancellation of this smooth radial system and does not contradict the general first-order estimate of Theorem~\ref{thm:sampling}.

\subsection{A genuinely planar escape computation: validating the boundary system on the damped Duffing system}
\label{subsec:duffing-escape}

The radial example demonstrated that intensity contains global basin information beyond local stability. However, that example still admitted a one-dimensional reduction. We now consider the damped Duffing system as a genuinely planar test case where the optimal escape direction is not known a priori \citep{guckenheimer2013nonlinear,thompson2002nonlinear,kovacic2011duffing}
\begin{equation}
    \dot x=y,\qquad \dot y=x-x^3-\delta y,\qquad \delta=0.5,
    \label{eq:duffing-system}
\end{equation}
with attracting wells $(\pm1,0)$ and a saddle at the origin (eigenvalues $\lambda_u=0.7808$, $\lambda_s=-1.2808$; Figures~\ref{fig:duffing}--\ref{fig:energy}). The basin boundary of $A=\{(1,0)\}$ is $W^s(0,0)$, a spiral (backward time reverses the damping), computed by backward integration along $v_s$; a basin indicator was built on a $961^2$ grid ($\Delta x=0.005$ on $[-2.4,2.4]^2$) by forward integration to $t=30$. The attractor eigenvalues are
$\lambda_\pm=-0.25\pm1.3919\,i$.
Thus local linearization captures only the local spiral decay and frequency, whereas it contains no information about the global spiral geometry of the basin boundary.

\begin{figure}[htbp]
    \centering
    \includegraphics[width=0.5\linewidth]{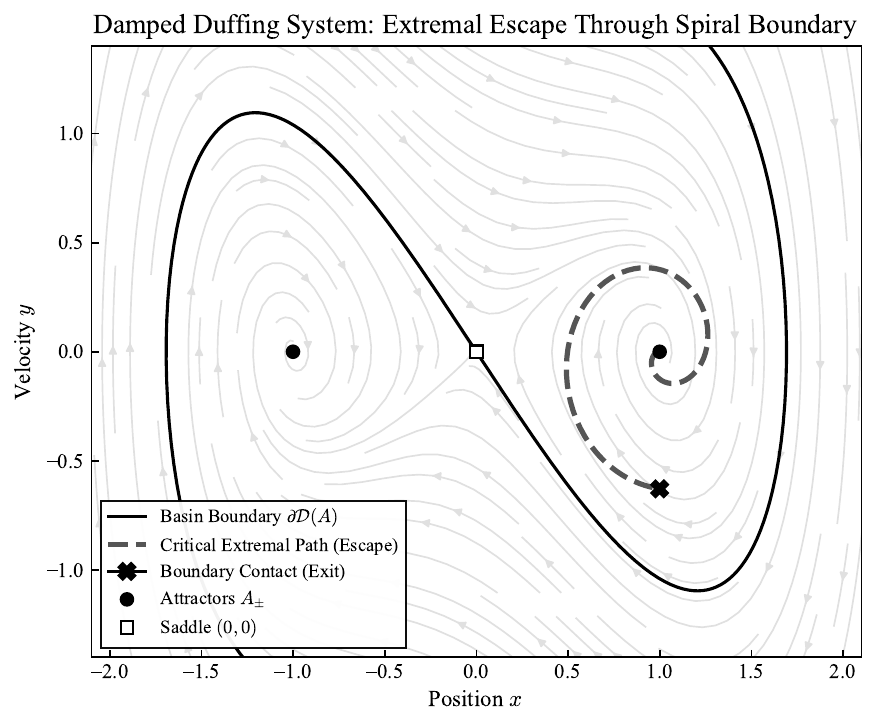}
    \caption{Damped Duffing system: extremal escape through spiral boundary.}
    \label{fig:duffing}
\end{figure}

The boundary system provides the primary computation of the intensity. We then validate this value using three independent checks. (i) \emph{Boundary
system}: integrating the state--normal system \eqref{eq:continuous-boundary-x}--\eqref{eq:continuous-boundary-n} from $(1,0)$ over $3600$ initial normals and bisecting in $r$ gives \begin{equation}\mu(A)=0.14309,\quad\text{bracket }[0.14305,0.14312].\label{eq:duffing-threshold}\end{equation} 
(ii) \emph{Extremal-control replay}: the critical extremal trajectory $u=r\,n(t)$ at $r=0.14380$ ($=1.005\mu$), replayed through DOP853 at tolerances $10^{-10}/10^{-12}$, crosses $W^s(0,0)$---an independent-integrator verification pinning $0.14305\le\mu(A)\le0.14380$. 
(iii) \emph{Grid set propagation}: iterating the sampled reachable set (flow map plus $hr$-dilation) with bisection yields resolution brackets $[0.1335,0.1459]$ ($h=0.2$) and $[0.1249,0.1502]$ ($h=0.1$) from the $O(\Delta x/h)$ cell-center bias, both containing $0.14309$. A fourth method, direct multi-start Nelder--Mead
optimization of a $32$-node control, gives the upper bound $\mu(A)\le0.1506$ (within $5.3\%$); the gap is a parametrization limit---the optimal direction co-rotates with the adjoint at period $2\pi/\omega_\delta\approx4.5$---not a discrepancy, as the replay shows.

\emph{Sampling.} The discrete boundary map \eqref{eq:discrete-boundary-map}
(with $D\phi^h$ from the variational equation) gives $\mu_h/h=0.14285, 0.14366, 0.14312, 0.14366, 0.14352$ at $h=0.4,0.2,0.1,0.05,0.025$ ($\pm5\times10^{-4}$), all agreeing with \eqref{eq:duffing-threshold} over a
sixteenfold range (Figure~\ref{sfig:duffconv}); the sampled boundary-map orbit at $h=0.1$ overlays the continuous extremal escape path (Figure~\ref{fig:duffing-phase}), illustrating Theorem~\ref{thm:boundary-limit}. \emph{Local indicators miss this.} The escape path spirals against the local rotation and threads the saddle bottleneck; at $\delta=0.2$ the boundary system gives $\mu(A)=0.0616$ while $\operatorname{Re}\lambda=-0.10$---intensity is set by the global restoring geometry, not the local decay rate, exactly as the radial family (Section~\ref{subsec:planar-radial-example}) shows with identical linearizations but different intensities.

\begin{figure}[htbp]
    \centering
    \includegraphics[width=0.5\linewidth]{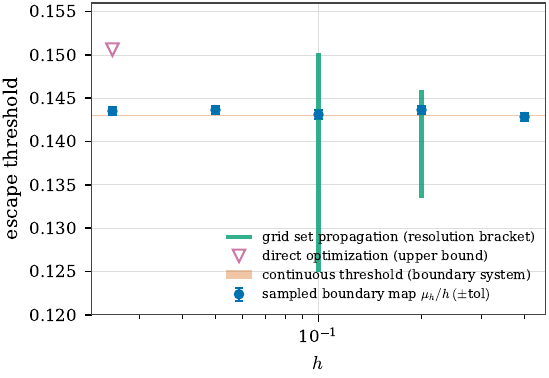}
    \caption{Duffing escape threshold: sampled boundary-map values $\mu_h/h$ (dots, $\pm$ sweep tolerance), grid-propagation resolution brackets, and the direct-optimization upper bound, against the continuous boundary-system value (band). Escape at $r=0.14380$ is additionally verified by the DOP853 control
    replay.}
    \label{sfig:duffconv}
\end{figure}

\begin{figure}[htbp]
    \centering
    \includegraphics[width=0.5\linewidth]{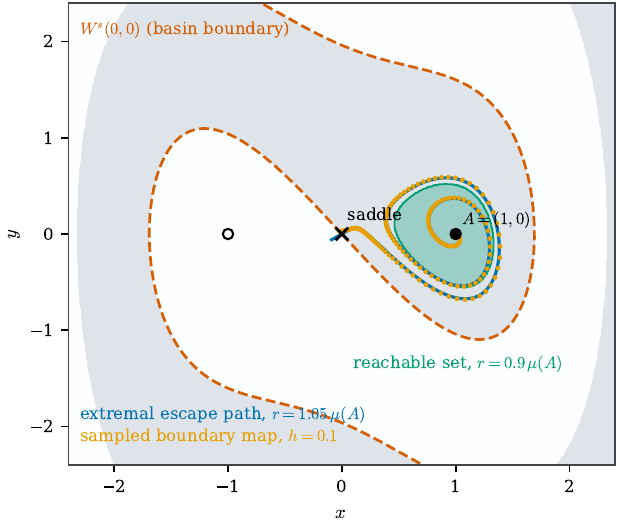}
    \caption{Damped Duffing system \eqref{eq:duffing-system}. Shaded: basin of $(-1,0)$; dashed: the spiral basin boundary $W^s(0,0)$; green: perturbed reachable set from $A=(1,0)$ at $r=0.9\,\mu(A)$ ($h=0.1$ grid propagation); blue: continuous extremal escape path at $r=1.05\,\mu(A)$; orange dots: sampled
    boundary-map orbit at $h=0.1$, overlaying the continuous extremal.}
    \label{fig:duffing-phase}
\end{figure}

\subsection{Anisotropic disturbances: disturbance geometry as a resilience metric}
\label{subsec:anisotropic-example}

The previous Duffing example quantified the resilience of a basin under isotropic perturbations. In many applications, however, disturbances are directional: different state components may experience different forcing magnitudes. The natural object is therefore not the Euclidean intensity but the $C$-intensity associated with a prescribed disturbance body. We illustrate this effect on the Duffing system using elliptical disturbance sets
\[
C_s=\left\{(u,v):
(u/s_x)^2+(v/s_y)^2\le1\right\},
\]
with support function $h_C(n)=\sqrt{s_x^2 n_x^2+s_y^2 n_y^2}$ and $\nabla h_C(n)=(s_x^2 n_x,\,s_y^2 n_y)/h_C(n)$, by integrating the anisotropic boundary system, $\dot x=f(x)+r\nabla h_C(n)$, $\dot n=-(I-nn^\top)Df(x)^\top n$, swept over initial normals with bisection in $r$ for first contact with $W^s(0,0)$.

The isotropic body $s_x=s_y=1$ recovers the Euclidean value \eqref{eq:duffing-threshold} to $0.35\%$ (computed $0.1436$), validating the anisotropic boundary integration. The disturbance geometry then changes the intensity measurably:
\[
\begin{array}{c|ccc}
    C & \text{isotropic} & \text{velocity-elongated } (s{=}(0.6,1.4))
      & \text{position-elongated } (s{=}(1.4,0.6))\\
    \hline
    \mu_C(A) & 0.1436 & 0.1398 & 0.1323
\end{array}
\]
Increasing the disturbance extent in a direction that aligns with the escape corridor reduces the required amplitude, whereas elongating the orthogonal direction has a weaker effect. Every value lies in the comparison interval $[\mu(A)/R,\ \mu(A)/\rho]=[0.102,0.238]$ with $\rho=0.6$, $R=1.4$ (Figure~\ref{fig:anisotropic}). Position-aligned disturbances leave the vegetated-analogue state least resilient here, because the escape channel near the saddle is more nearly aligned with the position axis---an effect invisible to the Euclidean intensity alone. An independent anisotropic grid set-propagation (dilating the reachable set by $hrC$ through an anisotropic Euclidean distance transform) brackets each value to grid resolution, and the sampled anisotropic boundary map---replacing the kick by $hr\,\nabla h_C(\eta_h)$---gives $\mu_{h,C}/h$ agreeing with $\mu_C$ to within the sweep tolerance $5\times10^{-4}$ across $h\in\{0.4,0.2,0.1,0.05\}$, confirming the gauge form of the sampling theorem on a genuinely
two-dimensional anisotropic problem.

\begin{figure}[htbp]
    \centering
    \includegraphics[width=0.5\linewidth]{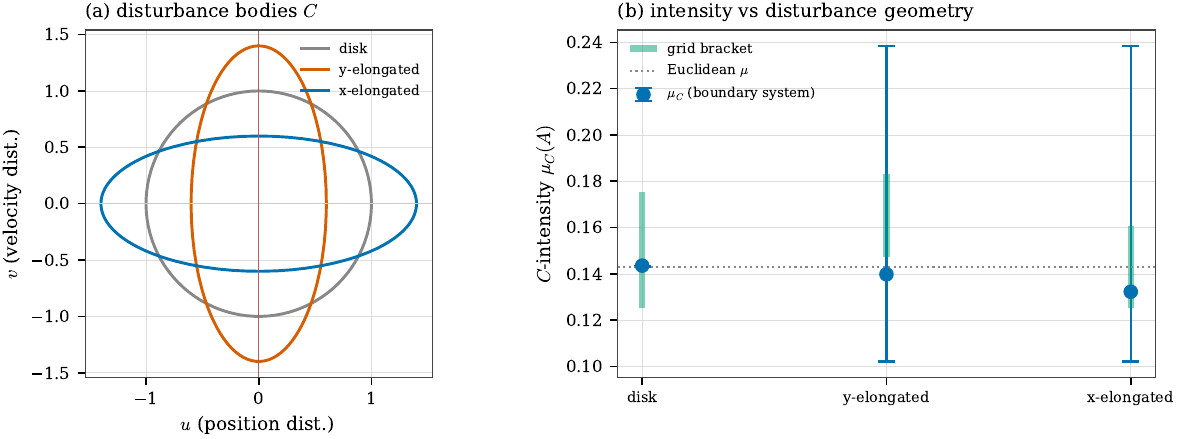}
    \caption{Anisotropic $C$-intensity of the Duffing attractor. (a) Elliptical disturbance bodies with different directional weights. (b) Computed intensities $\mu_C(A)$ from the anisotropic boundary system,
together with comparison bounds from the inclusion $\rho B\subset C\subset RB$. Directional disturbance geometry changes the escape threshold even for the same dynamical system.
}
    \label{fig:anisotropic}
\end{figure}

\subsection{A nonautonomous invariant graph}
\label{subsec:nonautonomous-computation}

We next consider
    $\dot x
    =
    x-x^3+\alpha\cos\theta$,
    $\dot\theta=\omega
    \pmod{2\pi}$.
A single angular frequency gives periodic rather than genuinely quasiperiodic forcing.  We use this one-frequency system because its critical intensity can be independently verified through a periodic-orbit
fold. It is a nonautonomous invariant-graph example over the compact base \(\mathbb S^1\), and the same graph-transform and sampling procedures apply to a genuinely quasiperiodic base.

We take as the primary case
    $\alpha=0.1$,
    $\omega=1$,
    $T=2\pi$.
The system has an attracting upper invariant graph
    $\mathcal A_+
    =
    \{(\theta,a_+(\theta)):\theta\in\mathbb S^1\}$
and a repelling middle graph that forms the fiberwise basin boundary. The graph \(a_+\) was computed as the attracting fixed point of the period-\(T\) graph transform.

For the controlled system
    $\dot x
    =
    x-x^3+0.1\cos t+u(t)$,
    $\|u(t)\|\leq r$,
order preservation implies that the lowest reachable boundary is generated by \(u\equiv-r\).  The intensity is therefore the smallest \(r\) for which the attracting upper periodic orbit of $\dot x=x-x^3+0.1\cos t-r$
collides with the repelling periodic orbit.  At the fold, the critical periodic solution \(x_\ast(t)\) satisfies $x_\ast(T)=x_\ast(0)\exp\left(
        \int_0^T
        \bigl(1-3x_\ast(t)^2\bigr)\,dt
    \right)=1$.
The fold was located by bisection (tolerance $10^{-12}$) on the maximum of $P_r(x)-x$ over the escape interval, with the period map $P_r$ computed by fixed-step RK4 ($4\times10^4$ steps per period) and
the maximizer refined locally; the critical point was then polished by secant iteration on the variational multiplier. The result is
\begin{equation}
    \mu_{\mathrm{na}}(\mathcal A_+)
    =
    0.3763812519\ldots,
    \label{eq:forced-continuous-intensity}
\end{equation}
with verified residuals
\[
    |x_\ast(T)-x_\ast(0)|
    =
    1.8\times10^{-12},
    \qquad
    \left|
        e^{\int_0^T(1-3x_\ast^2)\,dt}-1
    \right|
    =
    2.9\times10^{-12},
\]
both below the acceptance level $10^{-10}$.

The averaged autonomous equation is $\dot{\bar x}=\bar x-\bar x^3$, whose upper attractor has intensity
    $\mu_{\mathrm{av}}
    =
    \frac{2}{3\sqrt3}
    =
    0.3849001795\ldots$.
Thus averaging overestimates the nonautonomous intensity by approximately \(2.21\%\):
    $\frac{
        \mu_{\mathrm{av}}
        -
        \mu_{\mathrm{na}}
    }{
        \mu_{\mathrm{av}}
    }
    =
    0.02213\ldots$.
Although the attracting forced graph remains close to \(x=1\), the forcing changes the global bottleneck between the attracting and repelling graphs. This is precisely the type of information that a local averaging estimate near \(\mathcal A_+\) does not control.

The attracting graph stays close to $x=1$ ($0.9528\le a_+(t)\le1.0427$); the response $R_1=a_+-1$ has near-cancelling signed mean $-1.51\times10^{-3}$ but running absolute mean $R_2(T)=2.86\times10^{-2}$, detecting the phase-dependent discrepancy that averaging misses. The controlled tube at $r=0.95\mu_{\mathrm{na}}$ ($u\equiv\mp r$) has boundaries $a_r^-\in[0.5845,0.7700]$, $a_r^+\in[1.1122,1.1768]$ (Figure~\ref{fig:nonautonomous-intensity}(a)).

\begin{figure}[htbp]
  \centering
  \includegraphics[width=0.6\textwidth]{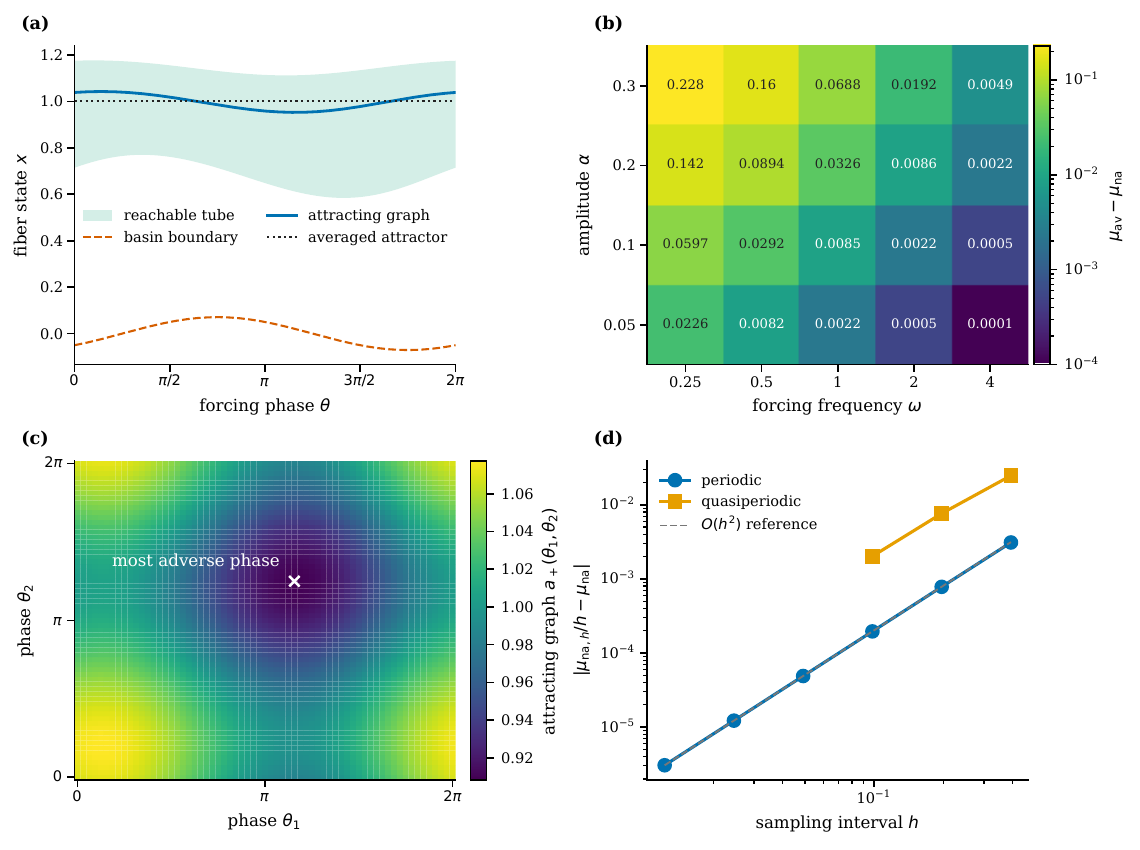}
  \caption{Uniform nonautonomous intensity and the limitation of averaging.
  (a)~For periodic forcing, the attracting graph remains close to the averaged
  attractor, but the controlled tube approaches the repelling basin-boundary
  graph.  (b)~The averaging error grows for stronger, slower forcing.
  (c)~The quasiperiodic attracting graph over $\mathbb T^2$ exposes a worst
  forcing phase that no one-phase Poincar\'e section can represent.
  (d)~Normalized sampled intensities converge for both periodic and genuinely
  quasiperiodic bases.}
  \label{fig:nonautonomous-intensity}
\end{figure}

We discuss sampled thresholds and phase uniformity.

Let \(h=T/N\).  
For the extremal sampled map
    $x_{k+1}
    =
    \varphi(h,\theta_k,x_k)-\varepsilon$, the threshold \(\varepsilon_*\) is the fold of the \(N\)-step
Poincar\'e map, located by the same refined bisection ($32$ RK4 substeps, tolerance $10^{-11}$); the
$N=16$--$512$ results (slopes $2.00$) are in Table~\ref{stab:naconv}.

\begin{table}[htbp]
    \centering
    \caption{Convergence of the normalized discrete nonautonomous
    intensity.  Here \(h=2\pi/N\) and
    \(p=\log_2(E_N/E_{2N})\).}
    \label{stab:naconv}
    \begin{tabular}{ccccc}
        \toprule
        \(N\)
        &
        \(h\)
        &
        \(\mu_{\mathrm{na},h}/h\)
        &
        absolute error
        &
        observed slope
        \\
        \midrule
        \(16\)  & \(0.392699\) & \(0.373285269\)
                & \(3.0960\times10^{-3}\) & --\\
        \(32\)  & \(0.196350\) & \(0.375600828\)
                & \(7.8042\times10^{-4}\) & \(1.99\)\\
        \(64\)  & \(0.098175\) & \(0.376185739\)
                & \(1.9551\times10^{-4}\) & \(2.00\)\\
        \(128\) & \(0.049087\) & \(0.376332348\)
                & \(4.8904\times10^{-5}\) & \(2.00\)\\
        \(256\) & \(0.024544\) & \(0.376369024\)
                & \(1.2227\times10^{-5}\) & \(2.00\)\\
        \(512\) & \(0.012272\) & \(0.376378195\)
                & \(3.0570\times10^{-6}\) & \(2.00\)\\
        \bottomrule
    \end{tabular}
\end{table}

As in the autonomous scalar example, the observed rate is second order. The general nonautonomous sampling theorem guarantees \(O(h)\); the additional cancellation here is associated with the smooth fold of scalar extremal graphs, in analogy with Corollary~\ref{cor:scalar-second-order} in the autonomous case.
Repeating the \(N=32\) computation from eight initial phases \(\theta_0=jh\), \(j=0,\dots,7\), reproduces the threshold with spread \(1.0\times10^{-11}\), numerically confirming the exact phase-invariance that conjugation of the sampled maps predicts, and consistent with the uniformity over the base built into
Definition~\ref{def:nonautonomous-intensity}.

We perform averaging parameter sweeps.

To map out when averaging preserves the escape threshold, the fold computation was repeated over $\alpha\in\{0.05,0.1,0.2,0.3\}$ and $\omega\in\{0.25,0.5,1,2,4\}$ (Table~\ref{tab:averaging-sweep} and Figure~\ref{fig:nonautonomous-intensity}(b)). The averaged intensity is $\mu_{\mathrm{av}}=2/(3\sqrt3)$ for every $\alpha$, so the tabulated averaging error $\mu_{\mathrm{av}}-\mu_{\mathrm{na}}(\alpha,\omega)$ is entirely a nonautonomous effect. This deviation grows with the forcing amplitude and decays with the forcing frequency: the averaged threshold is accurate for weak or fast forcing, but for strong slow forcing, the system quasi-statically tracks the instantaneously tilted potential. 
In this regime, the worst forcing phase erodes the threshold far below the averaged value.
This serves as numerical evidence delimiting the validity of averaged thresholds, consistent with the conditional results of Section~\ref{subsec:averaging-limitations}.

\begin{table}[t]
    \centering
    \caption{Averaging error
    \(\mu_{\mathrm{av}}-\mu_{\mathrm{na}}(\alpha,\omega)\) of the
    escape threshold (numerical evidence; fold tolerance
    $10^{-9}$).}
    \label{tab:averaging-sweep}
    \begin{tabular}{cccccc}
        \toprule
        & \(\omega=0.25\) & \(\omega=0.5\) & \(\omega=1\)
        & \(\omega=2\) & \(\omega=4\)\\
        \midrule
        \(\alpha=0.05\) & \(0.0226\) & \(0.0082\) & \(0.0022\)
                        & \(0.0005\) & \(0.0001\)\\
        \(\alpha=0.1\)  & \(0.0597\) & \(0.0292\) & \(0.0085\)
                        & \(0.0022\) & \(0.0005\)\\
        \(\alpha=0.2\)  & \(0.1416\) & \(0.0894\) & \(0.0326\)
                        & \(0.0086\) & \(0.0022\)\\
        \(\alpha=0.3\)  & \(0.2276\) & \(0.1596\) & \(0.0688\)
                        & \(0.0192\) & \(0.0049\)\\
        \bottomrule
    \end{tabular}
\end{table}

We discuss genuinely quasiperiodic graphs on the two-torus.

Periodic forcing reduces the fiberwise intensity to a Poincar\'e-map fold, so it does not exercise the two-torus structure that Theorem~\ref{thm:nonautonomous-sampling} was built for. We therefore treat
\begin{equation}
    \dot x=x-x^3+\alpha(\cos\theta_1+\cos\theta_2),\qquad
    \dot\theta_1=1,\quad \dot\theta_2=\omega_2,
    \label{eq:quasiperiodic-system}
\end{equation}
with $\alpha=0.1$ and $\omega_2=(1+\sqrt5)/2$ (the golden mean, irrational number). The attracting upper graph $a_+(\theta_1,\theta_2)$ and the repelling middle graph $a_{\mathrm m}$ live over $\mathbb T^2$ and are
computed by iterating the graph transform on a $64\times64$ torus grid, with the base shift $(\theta_1,\theta_2)\mapsto(\theta_1-\Delta,\theta_2-\omega_2 \Delta)$ applied by an exact spectral (FFT) shift and each fiber step integrated by RK4; both graphs are iterated to a fixed point ($\|a_{n+1}-a_n\|_\infty<10^{-11}$), which is essential for a $\Delta$-independent threshold. Under the extremal control $u\equiv-r$ the continued upper attracting graph collides with the repelling graph, and the uniform fiberwise intensity is
    $\mu_{\mathrm{na}}(\mathcal A_+)
    =\inf\Bigl\{r:\min_{\mathbb T^2}\bigl(a_+^{\,r}-a_{\mathrm m}^{\,r}\bigr)=0\Bigr\}$,
uniform over the base by construction. As a check on the whole apparatus, setting the second amplitude to zero recovers the one-frequency Poincar\'e fold \eqref{eq:forced-continuous-intensity} to $2\times10^{-5}$ (computed $0.376388$ versus $0.3763813$). For the genuinely quasiperiodic case \eqref{eq:quasiperiodic-system},
\begin{equation}
    \mu_{\mathrm{na}}(\mathcal A_+)=0.37282,
    \qquad
    \frac{\mu_{\mathrm{av}}-\mu_{\mathrm{na}}}{\mu_{\mathrm{av}}}=3.14\%,
    \label{eq:quasi-value}
\end{equation}
a larger averaging deficit than the single-frequency case because two incommensurate phases produce a more adverse worst-case combination; the attracting graph, though confined to $0.908\le a_+\le1.077$, has a
phase-dependent separation from the basin boundary ranging over $[0.842,1.154]$, and the smallest such separation---not the mean---sets the intensity (Figure~\ref{fig:nonautonomous-intensity}(c)). Sampling the flow gives $\mu_{\mathrm{na},h}/h=0.3481,0.3652,0.3708$ at $h=2\pi/16,2\pi/32,2\pi/64$, converging to \eqref{eq:quasi-value} with observed slopes approaching two (Figure~\ref{fig:nonautonomous-intensity}(d)), confirming Theorem~\ref{thm:nonautonomous-sampling} on a genuinely quasiperiodic base.

\subsection{Bounded-noise escape: worst case versus probability}
\label{subsec:bounded-noise-experiment}

We finally return to the scalar benchmark and simulate
    $X_{k+1}=F_h(X_k)+\xi_k$,
    $X_0=1$,
    $\operatorname{supp}\xi_k=[-hr,hr]$.
For each \(h\), the exact pathwise threshold is
    $r_{\mathrm{esc}}(h)
    =
    \frac{\mu_h(A)}h$,
with \(\mu_h(A)\) given by \eqref{eq:scalar-exact-discrete-intensity}.  We test \(r/r_{\mathrm{esc}}(h)\in\{0.95,1,1.025,1.05,1.10,1.20\}\) at \(h\in\{0.2,0.1,0.05\}\), over the horizon \(T_{\max}=200\), and
separate two computations.

We conduct worst-case experiments.

The deterministic realization \(\xi_k\equiv-hr\) verifies the threshold exactly.  For \(r/r_{\mathrm{esc}}\in\{0.95,1\}\) no escape occurs at any tested \(h\); above the threshold escape occurs after finitely many steps (at \(h=0.1\): \(223,\ 151,\ 101,\ 66\) steps for the four supercritical ratios).  At equality the orbit approaches the nonhyperbolic bottleneck \(x_h^\ast\) (with \((x_h^\ast)^2=(e^{2h/3}-1)/(e^{2h}-1)\)) without crossing it: at \(T_{\max}=200\) the distance is \(|X_N-x_h^\ast|\approx2.8\times10^{-3}\) at every tested \(h\), consistent with the algebraically slow approach to a parabolic tangency. In particular, the supremum defining \(\mu_h(A)\) is attained in this example: at \(\varepsilon=\mu_h(A)\), the strict-inequality convention for kicks still confines every admissible
orbit behind the tangency point, as in Step~3 of the proof of Proposition~\ref{prop:scalar-one-sided}.

We conduct probabilistic experiments involving transfer operators and Monte Carlo simulations.

As the primary noise law we take uniform bounded noise, \(\xi_k\sim\operatorname{Unif}[-hr,hr]\); as a deliberately adverse stress test we retain the mixture
\begin{equation}
    \xi_k
    \sim
    \begin{cases}
        -hr, & \text{with probability }0.9,\\
        \operatorname{Unif}[-hr,hr],
            & \text{with probability }0.1,
    \end{cases}
    \label{eq:stress-test-noise-law}
\end{equation}
which biases every step toward escape.  Rather than relying on Monte Carlo alone, survival probabilities were computed from the absorbing transfer operator: with \(V_k(x):=\mathbb P_x\{\tau_{\mathcal D}>k\}\) and escape upon \(X\leq0\),
\[
    V_{k+1}(x)
    =
    \frac{1}{2hr}
    \int_{\max(\phi^h(x)-hr,\,0)}^{\phi^h(x)+hr}
    V_k(y)\,dy,
\]
iterated on a grid of \(n=4001\) nodes (the mixture law adds the point-mass term \(0.9\,V_k(\phi^h(x)-hr)\)); the state space is bounded above because \(\phi^h\) is bounded, so no truncation error enters from the right.  The operator resolves probabilities down to the floating-point floor \(10^{-12}\).  Grid refinement at the representative point (\(h=0.1\), mixture, \(r/r_{\mathrm{esc}}=1.05\)) gives escape probabilities \(9.087\times10^{-2}\), \(8.397\times10^{-2}\), and \(8.357\times10^{-2}\) at \(n=1001,4001,16001\), converging at first order in the grid.  Monte Carlo validation used \(10^5\) trajectories per case, with Wilson score intervals.

\begin{table}[t]
    \centering
    \caption{Escape probabilities \(\mathbb P(\tau\leq200)\) under the adverse mixture \eqref{eq:stress-test-noise-law}: transfer operator (op) at three sampling steps, and Monte Carlo with \(95\%\) Wilson intervals at \(h=0.1\).  Entries \(0\) are below the operator floor \(10^{-12}\). Under the uniform law, all entries of the corresponding table are below \(10^{-12}\), and no escape occurred in \(10^5\) Monte Carlo trajectories for any ratio up to \(1.20\) (Wilson upper bound \(3.8\times10^{-5}\)).}
    \label{tab:bounded-noise-probabilities}
    \small
    \begin{tabular}{l|rrrrrr}
        \toprule
        \(r/r_{\mathrm{esc}}(h)\)
        & \(0.95\) & \(1.00\) & \(1.025\) & \(1.05\) & \(1.10\)
        & \(1.20\)\\
        \midrule
        op, \(h=0.2\)  & \(0\) & \(0\) & \(2.44\times10^{-2}\)
            & \(0.609\) & \(1.000\) & \(1.000\)\\
        op, \(h=0.1\)  & \(0\) & \(0\) & \(5.99\times10^{-5}\)
            & \(8.40\times10^{-2}\) & \(0.9988\) & \(1.000\)\\
        op, \(h=0.05\) & \(0\) & \(0\) & \(5.6\times10^{-10}\)
            & \(8.21\times10^{-4}\) & \(0.9813\) & \(1.000\)\\
        \midrule
        MC, \(h=0.1\)  & \(0\) & \(0\) & \(2.0\times10^{-5}\)
            & \(8.38\times10^{-2}\) & \(0.9988\) & \(1.000\)\\
        Wilson \(95\%\) & \(\leq3.8\!\times\!10^{-5}\)
            & \(\leq3.8\!\times\!10^{-5}\)
            & \([0.5,7.3]\!\times\!10^{-5}\)
            & \([8.21,8.56]\!\times\!10^{-2}\)
            & \([.9985,.9990]\) & \(1.000\)\\
        \bottomrule
    \end{tabular}
\end{table}

Table~\ref{tab:bounded-noise-probabilities} shows three features. (i) Corollary~\ref{cor:bounded-noise} is visible in the first two columns: no realization escapes at or below the pathwise threshold, for either law,
by either method. (ii) Under the \emph{uniform} law no escape is resolvable even at $1.2\,r_{\mathrm{esc}}(h)$---operator $<10^{-12}$, no escape in $10^5$ trajectories: an admissible escape word exists (length $66$ at $h=0.1$) but realizing $66$ consecutive near-extreme uniform kicks is unobservably rare, the worst-case/probabilistic gap of Section~\ref{subsec:bounded-random-perturbations} quantified. (iii) Under
the adverse mixture the transition is smooth and sharpens as $h\downarrow$; operator and Monte Carlo agree within the Wilson intervals, with conditional mean exit times $114.3,109.6,41.8,11.3$ for the four supercritical ratios, and post-threshold escape can be slow. Survival curves $\widehat S(t)=\mathbb P(\tau>t)$ for the adverse mixture can be seen from Figure~\ref{fig:bounded-noise}.

\begin{figure}[htbp]
  \centering
  \includegraphics[width=0.6\textwidth]{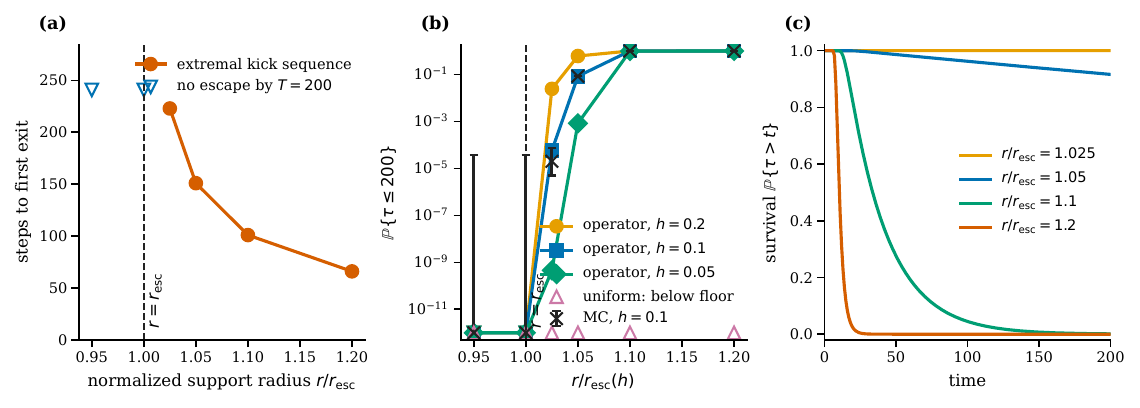}
  \caption{Worst-case and probabilistic bounded-noise escape for the scalar
  benchmark.  (a)~The extremal admissible kick sequence cannot escape at or
  below $r_{\rm esc}(h)$ and escapes in finite time above it.  (b)~Finite-horizon
  escape probabilities depend strongly on the noise law and sampling interval;
  transfer-operator values agree with Monte Carlo confidence intervals.
  (c)~Survival curves show that the post-threshold transition can remain slow.
  Thus the support threshold is pathwise sharp but is not itself an
  almost-sure or finite-horizon probability threshold.}
  \label{fig:bounded-noise}
\end{figure}

\section{Conclusions}
\label{subsec:conclusions}

This work places the intensity of attraction on a common quantitative footing for continuous flows, sampled maps, numerical discretizations, nonautonomous systems, and bounded-support random perturbations. Three main conclusions emerge.

\emph{First, normalized discrete intensity provides a consistent approximation of continuous-time resilience.}
Under the compactness, trapping, and Lipschitz hypotheses of the sampling theorem, the intensity of the time-$h$ map satisfies $\frac{\mu(A)}{1+Lh}
    \leq
    \frac{\mu_h(A)}{h}
    \leq
    \mu(A)\frac{e^{Lh}-1}{Lh}$,
and therefore
$\frac{\mu_h(A)}{h}\to\mu(A)$ as $h\downarrow0$.
The normalization by \(h\) is essential: \(\mu_h(A)\) measures displacement per sampling step, whereas \(\mu(A)\) measures disturbance amplitude per unit time. The general \(O(h)\) estimate is sharp, as demonstrated by the kinked scalar example. Under the smooth nondegeneracy assumptions of Corollary~\ref{cor:scalar-second-order}, the leading first-order term cancels and the normalized intensity converges at second order. For numerical one-step methods, the sampling error and the integrator error must be distinguished: on a common trapping block, and under the required approximation of the distinguished attractor,
$\frac{\mu_{B,h}^{\Psi}}{h}
    =
    \mu_B(A)+O(h)+O(h^q)$
for a method of order \(q\). The planar Duffing computation further shows that boundary-system integration, sampled boundary maps, direct optimization, control replay, and grid propagation identify the same escape threshold within their stated tolerances. The associated spiral extremal path is determined by global basin geometry and is invisible to the local eigenvalues of the attractor. Moreover, the small-step limit of the discrete boundary map recovers the Pontryagin boundary system, including its extension to convex anisotropic disturbance bodies through their support functions.

\emph{Second, the sampling principle extends uniformly to nonautonomous invariant graphs, but averaging does not automatically preserve intensity.}
For an attracting invariant graph over an invertible compact base flow, the fiberwise normalized intensity satisfies
$\frac{\mu_{\mathrm{na},h}(\mathcal A)}{h}
    \to
    \mu_{\mathrm{na}}(\mathcal A)$
uniformly over the forcing base. The periodic computations confirm phase independence to numerical precision, while the quasiperiodic example demonstrates that the result is genuinely uniform over a two-dimensional torus rather than merely a consequence of a Poincar\'e reduction. At the same time, intensity remains a global basin quantity. Long-time trajectory agreement near an attracting graph, hyperbolic continuation, or convergence to an averaged model does not by itself control the fiberwise basin boundary. In the family studied here, averaging is accurate for weak or rapidly varying forcing but increasingly overestimates the escape threshold for strong, slowly varying forcing. This is numerical evidence for a specific family, not a universal averaging theorem; rigorous transfer of intensity requires a common trapping block and quantitative control of both the sampled-map defect and the invariant-graph defect.

\emph{Third, bounded-support noise possesses an intrinsic pathwise escape threshold, but that threshold is not a probability law.}
For
$X_{k+1}=\phi^h(X_k)+\xi_k$, $\operatorname{supp}\xi_k\subseteq hrC$,
the worst-case threshold is
$r_{\mathrm{esc},C}(h)
    =
    \frac{\mu_{h,C}(A)}{h}
    \to
    \mu_C(A)$.
If \(r<r_{\mathrm{esc},C}(h)\), every admissible realization remains in a common compact subset of the basin. If \(r>r_{\mathrm{esc},C}(h)\), an admissible finite escape sequence exists; no general assertion is made at equality. Converting this geometric possibility into positive-probability or almost-sure escape requires additional assumptions on the noise law, such as suitable support, independence, or a Markov minorization condition. The transfer-operator and Monte Carlo experiments make this distinction explicit: an adverse law produces observable escape above the threshold, whereas uniform bounded noise exhibits no resolvable escape over the chosen horizon even at \(20\%\) above it.

Taken together, these results identify intensity as a dimensionally consistent, computable, and genuinely global \(L^\infty\)-type resilience rate. It complements rather than replaces hyperbolicity, recovery-rate indicators, basin-stability measures, and quasipotential theory. For unbounded noise, a strict support-wise safety threshold generally disappears, and the appropriate objects are finite-horizon escape probabilities, exit-time distributions, quasipotentials, and large-deviation actions. Developing quantitative connections among these probabilistic quantities, sampled intensity, and data-driven resilience estimation is a natural direction for future work.

\bibliography{reference}

\end{document}